\documentclass[twoside,12pt]{amsart}
\date{\today}

\usepackage{enumerate}
\usepackage{graphicx}
\usepackage{amsmath}
\usepackage{verbatim} 
\usepackage{amssymb}
\usepackage{epsfig}
\usepackage{amsthm}
\usepackage{amscd}
\usepackage[usenames]{color}
\usepackage{yfonts}
\usepackage[T1]{fontenc}
\newcommand{\Ab}{\mbox{\normalfont\gothfamily A}}
\newcommand{\Cb}{\mathfrak{C}}

\definecolor{r}{rgb}{.9,0.1,.3}

\definecolor{g}{rgb}{0,.6,0.6}

\newcounter{intro}

\usepackage[latin1]{inputenc}

\newcommand{\ir}{{\rm i}}

\newcommand{\D}{\mathbb{D}}

\newcommand{\R}{\mathbb{R}}

\newcommand{\n}{\mathbb{N}}
\newcommand{\h}{\mathbb{H}}
\newcommand{\sph}{\mathbb{S}}

\newcommand{\partiali}{\partial_{\operatorname{int}}}
\renewcommand{\Im}{\operatorname{Im}}
\renewcommand{\Re}{\operatorname{Re}}

\newcommand{\dist}{\operatorname{dist}}
\newcommand{\imag}{\mathrm{i}}

\def\flux {\operatorname{Flux}}

\newcommand{\CC}{\mathbb C}
\newcommand{\HH}{\mathbb H}

\newcommand{\RR}{\mathbb R}
\newcommand{\ZZ}{\mathbb Z}
\newcommand{\del}{\partial}

\newcommand{\eps}{\epsilon}

\newcommand{\calC}{{\mathcal C}}
\newcommand{\calD}{{\mathcal D}}

\newcommand{\calL}{{\mathcal L}}
\newcommand{\calM}{{\mathcal M}}
\newcommand{\calN}{{\mathcal N}}
\newcommand{\calO}{{\mathcal O}}

\newcommand{\calR}{{\mathcal R}}

\newcommand{\calU}{{\mathcal U}}
\newcommand{\calV}{{\mathcal V}}
\newcommand{\calW}{{\mathcal W}}

\newcommand{\fC}{\mathfrak{C}}
\newcommand{\fA}{\mbox{\normalfont\gothfamily A}}
\newcommand{\fJ}{\mathfrak{J}}
\newcommand{\fU}{\mathfrak{U}}
\newcommand{\calX}{\mathcal X}

\newcommand{\wh}{\widehat}
\newcommand{\wt}{\widetilde}

\def\z{\mathbb{Z}}
\numberwithin{equation}{section}
\newtheorem{theorem}{Theorem}[section]
\newtheorem{corollary}[theorem]{Corollary}
\newtheorem{lemma}[theorem]{Lemma}
\newtheorem{claim}[theorem]{Claim}
\newtheorem{proposition}[theorem]{Proposition}
\newtheorem*{maintheorem*}{Main Theorem}
\newtheorem*{theorem*}{Theorem}
\newtheorem*{mainlemma*}{Main Lemma}
\newtheorem*{keylemma*}{Key Lemma}

\theoremstyle{definition}
\newtheorem{remark}[theorem]{Remark}
\newtheorem*{remark*}{Remark}
\newtheorem{definition}[theorem]{Definition}

\title[Properly embedded minimal annuli in $\h^2 \times \R$]{Properly embedded minimal annuli in $\h^2 \times \R$}%
\author[L. Ferrer]{Leonor Ferrer}
\address[Ferrer]{
  Departamento de Geometr\'\i{}a y Topolog\'\i{}a,
  Universidad de Granada,
  18071 Granada, Spain.
}
\email{lferrer@ugr.es}
\author[F. Martín]{Francisco Mart\'\i{}n}
\address[Martín]{
  Departamento de Geometr\'\i{}a y Topolog\'\i{}a,
  Universidad de Granada,
  18071 Granada, Spain.
}
\email{fmartin@ugr.es}
\author[R. Mazzeo]{Rafe Mazzeo}
\address[Mazzeo]{%
Department of Mathematics,
Stanford University,
Building 380, Sloan Hall,
Stanford, California 94305, U.S.A.
}
\email{mazzeo@math.stanford.edu}
\author[M. Rodríguez]{Magdalena Rodríguez}
\address[Rodríguez]{
  Departamento de Geometr\'\i{}a y Topolog\'\i{}a,
  Universidad de Granada,
  18071 Granada, Spain.
}
\email{magdarp@ugr.es}
\thanks{
 L. Ferrer, F. Martín and M.M. Rodríguez are partially supported by the
  MINECO/FEDER grant MTM2014-52368-P and MTM2017-89677-P;  R. Mazzeo supported by
the NSF grant DMS-1105050 and DMS-1608223; F. Martin is also partially supported by the
Leverhulme Trust grant IN-2016-019. }
\begin{document}

\maketitle
\begin{abstract}
In this paper we study the moduli space of properly Alexandrov-embedded, minimal annuli in $\h^2 \times \R$ with 
horizontal ends. We say that the ends are horizontal when they are graphs of $\mathcal{C}^{2, \alpha}$ functions over 
$\partial_\infty \h^2$. Contrary to expectation, we show that one can not fully prescribe the two 
boundary curves at infinity, but rather, one can prescribe one of the boundary curves, but the other one only up to
a translation and a tilt, along with the position of the neck and the vertical flux of the annulus. 
We also prove general existence theorems for minimal annuli with discrete groups of symmetries. 
\end{abstract}

\section{Introduction}
This paper studies the space of properly embedded minimal annuli with horizontal ends in $\HH^2 \times \RR$.  
Prototypes of such surfaces are the so-called vertical catenoids $C$. These are surfaces of revolution 
with respect to some vertical axis $\{p\} \times \RR$. Their asymptotic boundary is the union of two 
parallel circles in $\del_\infty \HH^2 \times \RR$ and  there are
functions $u^\pm$ defined on $\HH^2 \setminus K$ for some compact set $K$ such that the ends of 
$C$ are graphs $t = u^\pm(z)$. In this particular case, they are also symmetric around a horizontal plane $\HH^2 \times \{t_0\}$,
so in particular, if we translate so that $t_0 = 0$, then $u^-(z) = - u^+(z)$.  

Here and later,  we use the  Poincar\'e disc model $\{|z| < 1\}$ of $\HH^2$ with metric $g_0 = 4 |dz|^2/(1-|z|^2)^2$, 
so the product metric on $\HH^2 \times \RR$ is $g = g_0 + dt^2$, and also write $z = re^{i\theta}$, $r < 1$.  
To be clear, we regard $z$ as a fixed global coordinate chart on $\HH^2$. 
More generally, we seek minimal annuli $A \subset \HH^2 \times \RR$ for which the asymptotic boundary $\del_\infty A$ is a 
union of two curves $\gamma^\pm$ which can be represented as graphs $t = \gamma^\pm(\theta)$, 
$\theta \in \sph^1 = \del_\infty \HH^2$. We say that such ends are {\bf horizontal}. One of the main results of 
this paper consists of proving that any properly embedded, annular horizontal end can be written (outside a compact set) 
as the graph of a smooth function $\HH^2 \to \RR$ (see Section \ref{sec:graphs}).
For the vertical catenoids described above, the boundary curves are constant graphs, $\gamma^\pm(\theta)
\equiv a^\pm$. The general question is to determine which pairs $\Gamma = (\gamma^\pm)$ (initially with $\gamma^-(\theta) 
\leq \gamma^+(\theta)$, $\theta \in \sph^1$) bound a properly embedded minimal annulus with horizontal 
ends.  

Taking a broader perspective, the asymptotic Plateau problem in $\HH^2 \times \RR$ asks for a characterization 
of those curves (or closed subsets) in the asymptotic boundary of $\HH^2 \times \RR$ which bound complete 
minimal surfaces. Implicit in this question is a choice of compactification of this space. This question is 
discussed in some generality in \cite{KM}; in the present paper we consider only the product compactification 
$(\overline{\HH^2 \times   \RR})^\times =\overline{\HH^2} \times \overline{\RR}$, which is the product of a 
closed disk and a closed interval, and only consider boundary curves lying in the vertical part of the boundary
$\overline{\HH^2} \times \RR$. The paper \cite{KM} describes a number of different families of examples 
of `admissible' (connected) boundary curves and notes various obstructions for such curves to be asymptotic 
boundaries. 

As above, a curve $\gamma$ is called horizontal if it lies in the vertical boundary $(\partial_\infty \HH^2) 
\times \RR$ of this product compactification and is a graph $t = \gamma(\theta)$, $\theta \in \sph^1$. 
The simplest problem is to determine whether any connected horizontal curve bounds a minimal surface,
and this was settled by Nelli and Rosenberg \cite{n-r}. They proved that if $\gamma(\theta) \in \calC^0(\sph^1)$, 
then there exists a unique function $u$ defined on the disk $\{|z|<1\}$, with $u = \gamma$ at $r=1$, such that
the graph of $u$ is minimal in $\HH^2 \times \RR$. Moreover, this solution is unique, so any complete 
embedded minimal surface with connected horizontal boundary must be a vertical graph.  We refer to 
\cite{st1, KM, Con} for a list of various general existence and non-existence results for other classes of connected boundary curves. 

The existence result for pairs of horizontal boundary curves, $\gamma^\pm$, one lying above the other, is more 
complicated. As above, we consider only minimal annuli, though certain facts hold even for higher genus surfaces.  
First, not every pair $\gamma^\pm$ is fillable by minimal annuli.  For example, these curves cannot be too far apart.
In Theorem \ref{th:nonexistence} we prove that if $\gamma^+(\theta) -  
 \gamma^-(\theta) > \pi$ for all $\theta$, then no such minimal annulus exists. 

Define
\[
\Cb = \{ (\gamma^+, \gamma^-):  \gamma^\pm \in \calC^{2,\alpha}(\sph^1), \quad \gamma^+(\theta) > 
\gamma^-(\theta)\ \mbox{for all}\ \theta\}.
\]
The restriction on existence above suggests that we focus on the open subset 
\[
\Cb^\pi = \{ (\gamma^+, \gamma^-) \in \Cb:  \sup_{\theta \in \sph^1} \left( \gamma^+(\theta) - \gamma^-(\theta)\right) < \pi\}.
\]
We also define $\Ab$ to be the space of properly embedded minimal annuli with $\del A \in \Cb$.
Our main results will be phrased in terms of properties of the natural projection map 
$$
\Pi: \Ab \longrightarrow \Cb, \qquad \Pi(A):= \del A. 
$$

The first result is the easiest one to state. Consider the subspaces $\fC_m \subset \fC^\pi$ and $\fA_m$ of boundary curves 
and minimal annuli which are invariant under the discrete group of isometries generated by the rotation $R_m$ by angle $2\pi/m$ 
about the axis $\{o\}\times \RR$.  Imposing symmetry eliminates a degeneracy in the problem. 
\begin{theorem} For any $m \geq 2$, 
\[
\left. \, \Pi \, \right|_{\fA_m}: \fA_m \to \fC_m  
\]
is surjective. 
\end{theorem}

It is not the case that the full map $\Pi: \Ab \to \Cb^\pi$ is surjective, and indeed, we present below a simple and large 
family of examples of pairs of curves which do not bound minimal annuli.  Thus we prove a slightly weaker existence result.
\begin{theorem} 
Given any $(\gamma^+,\gamma^-) \in \calC^{2,\alpha}(\sph^1)^2$, there exist constants $a_0$, $a_1$, $a_2$ so that the pair 
$(\gamma^+ + a_0+ a_1\cos \theta + a_2 \sin \theta, \gamma^-)$ bounds a properly Alexandrov-embedded,  minimal annulus. 
\end{theorem}
\begin{remark}
There is a very important difference between this result and how we have tried to formulate the result previously. 
First, we are not specifying the boundary curves completely, but allowing a three-dimensional freedom in the
top curve.  Second, and of fundamental importance, we pass from the space of properly embedded to (properly) Alexandrov-embedded 
minimal annuli with embedded ends.  We denote this space by $\fA^*$.  It is most likely impossible to characterize the precise set of pairs of boundary curves 
for which the minimal annuli provided by this theorem are actually embedded, but if we allow Alexandrov-embeddedness, there is a satisfactory 
global existence theorem.   For the subclasses $\fA_m$ and $\fC_m$ however, it is possible to remain within the class of embedded surfaces. 
\end{remark}

The strategy to prove both of these theorems uses degree theory in a familiar way. The main step is to show that $\Pi$
is a proper Fredholm map.  This is true for the restriction of $\Pi$ to $\Ab_m$, but unfortunately may not be the case 
on all of $\Ab$, so instead we consider a finite dimensional extension of $\Pi$ which is proper, but which leads to the 
need to introduce the extra flexibility in the top boundary curve.  

After setting forth some notation and basic analytic and geometric facts in the next section, \S \ref{sec:graphs} contains an extension
of a theorem of Collin, Hauswirth and Rosenberg \cite{chr} and proves that the ends of elements of $\fA$ are indeed vertical graphs.
\begin{proposition} If $A \in \fA$, then there is a compact set $K \subset A$ such that $A \setminus K = E^+ \sqcup E^-$,
where each $E^\pm$ is a vertical graph of a function $u^\pm$ over some region $\{r_0 < |z| < 1\}$. 
\end{proposition}
This result is followed by the calculation of fluxes on horizontal ends in \S \ref{subsec:flux}.  
Next, we present the nonexistence
theorems in \S \ref{sec:nonexistence}.

By an observation in \cite{KM} (see the proof of Theorem \ref{manifoldthm}), $u^\pm$ extends to a $\calC^{2,\alpha}$ function up to $|z| = 1$, or equivalently, $\overline{A}$ 
is a $\calC^{2,\alpha}$ surface with boundary.   We prove in \S 6 that the space $\fA$ is a Banach manifold and study the space of 
Jacobi fields on a minimal annulus $A$. This leads to the definition, in \S 7, of the extended boundary map $\widetilde{\Pi}$, and 
an exploration of its infinitesimal properties. The more difficult fact that $\widetilde{\Pi}$ is proper occupies \S 8. We finally prove 
the two main theorems in \S 9 and \S 10. 

\vskip 3mm

\noindent {\bf Acknowledgements.} We thank B. White, J. Pérez and A. Ros for valuable conversations and suggestions. L. Ferrer 
and F. Martín are very grateful to the Mathematics Department of Stanford University for its hospitality 
during part of the time the research and preparation of this article were conducted.  R. Mazzeo is also very grateful to the 
Institute of Mathematics at the University of Granada for sponsoring a visit where this work started. 
Finally, we would like to thank the anonymous referee for many helpful suggestions.
\section{vertical catenoids} \label{sec:2}
In this section we recall the salient geometric and analytic properties of the vertical minimal catenoids.

\medskip

As in the introduction, we use the Poincar\'e disk model for $\HH^2$, with Cartesian coordinates $z \in  \{|z| < 1\}$ and 
polar coordinates $(r,\theta)$. We shall also use the notation $D(z_0,R)$ and $D_{\HH^2}(z_0, R)$ to denote the Euclidean and hyperbolic
disks with center $z_0$ and (Euclidean or hyperbolic) radius $R$.  When $z_0 = o$ is the origin (in this coordinates), we sometimes omit it from the notation.

\subsection{Geometric properties}
The family of vertical minimal catenoids was introduced and studied by Nelli and Rosenberg \cite{n-r} as the unique
family of (non flat) minimal surfaces invariant under rotations around a fixed vertical axis. Indeed, parametrizing a surface of rotation by 
\[
[0,2\pi] \times (a,b) \ni (\theta,t) \mapsto X(\theta,t) := ( r(t) e^{\ir \theta },  t), 
\]
then minimality is equivalent to the equation 
\[
4r r'' - 4 (r')^2 - (1- r^4) = 0.
\]
Integrating this gives that 
\begin{equation} \label{eq:kappa}
\displaystyle \kappa^2= \frac{(1-r^2)^2}{4 r^2}+\left( \frac{r'}{r} \right)^2
\end{equation}
for some constant $\kappa^2 > 0$. It can then be deduced that solutions exist on some interval $(a,b)$ with $b -a = 2h < \pi$, 
and furthermore that the correspondence $(0,\pi/2) \ni h \to \kappa^2 \in \RR^+$ is bijective.   We denote the corresponding
surface by $C_h$, usually with the normalization that $a=-h$, $b=h$, hence $r(-h) = r(h) = 1$.  Note that the first-order equation for $r$ 
implies that
\[
\sqrt{1 + \kappa^2} - \kappa \leq r(t) < 1,
\]
and this lower bound is the minimum of the corresponding solution $r(t)$; this provides a correspondence between $\kappa$ 
and the minimum value $r(0)$. 

We now calculate that 
\[
X^* \left( \left. g \right|_{C_h} \right) = \frac{4r^2}{(1-r^2)^2} (  d\theta^2  + \kappa^2 dt^2). 
\]
This is nearly conformal and the extra constant factor (which could obviously be scaled away) does not cause any problems. 

\begin{figure}[htpb]
\begin{center}
\includegraphics[height=.39\textheight]{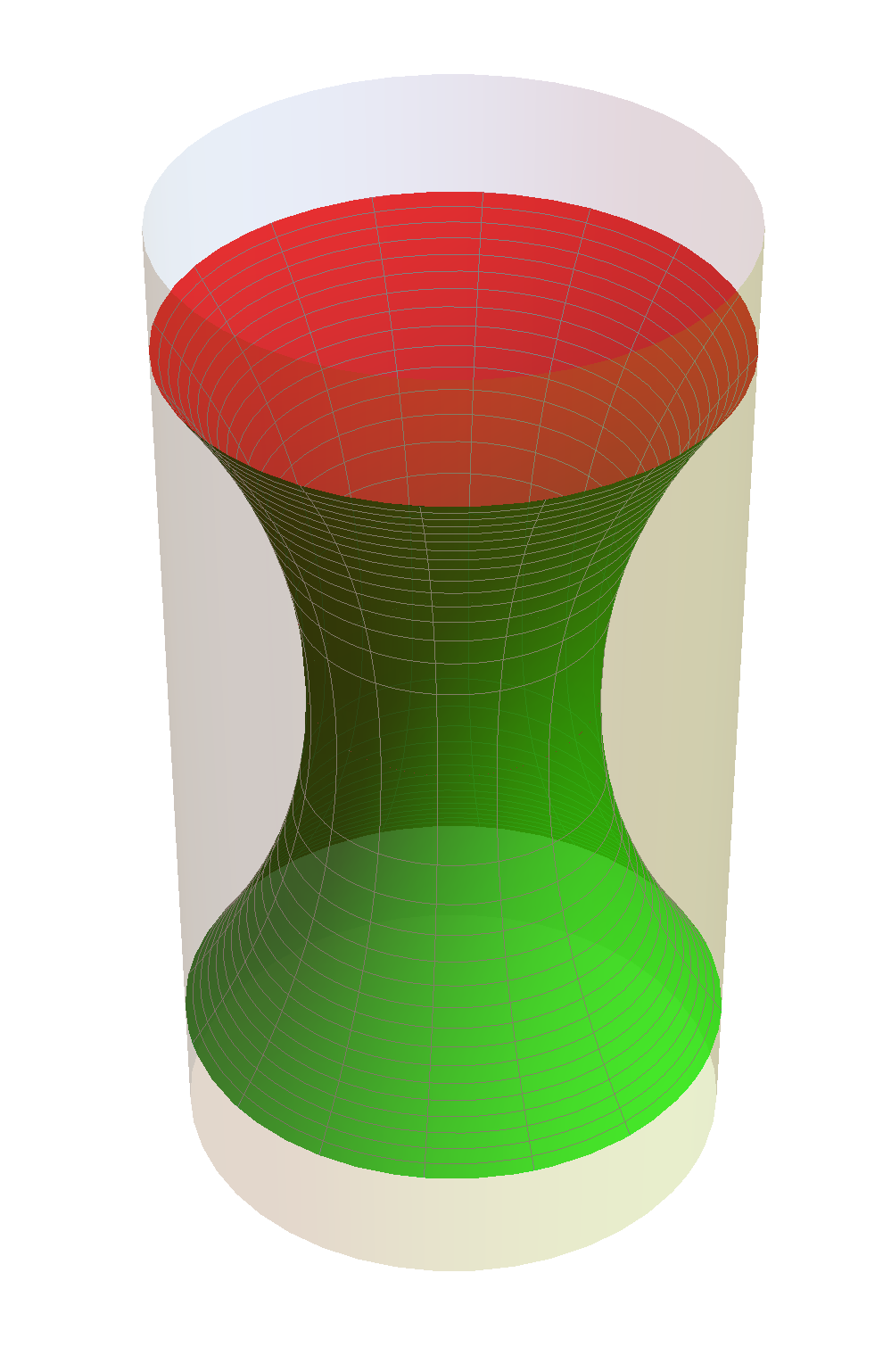} \hspace{3cm} \includegraphics[height=.39\textheight]{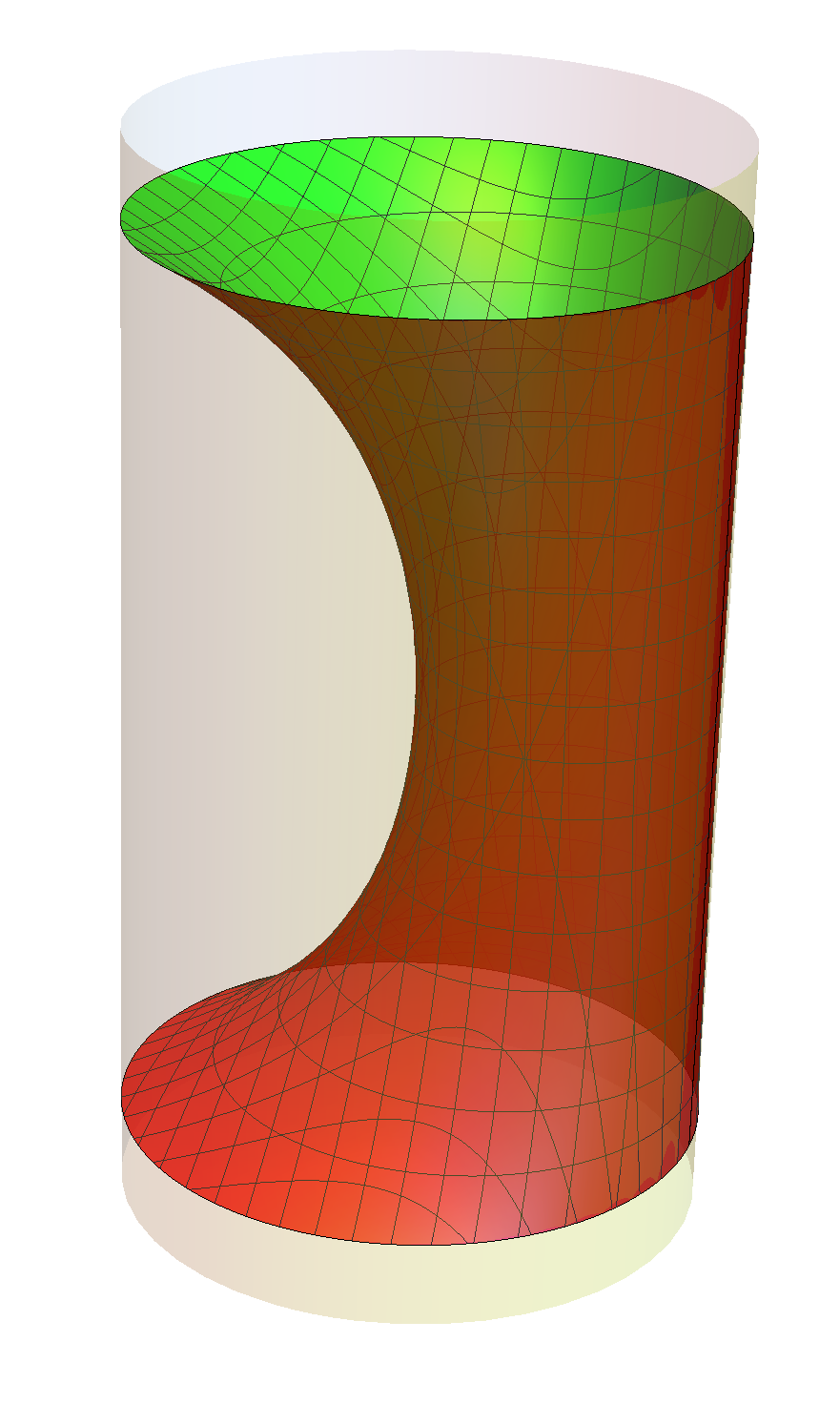}
\end{center}
\caption{Vertical catenoid and parabolic generalized catenoid.} \label{catenoid}
\end{figure}

These surfaces, called catenoids, have the following properties: 
\begin{enumerate}
\item $C_h$ is a bigraph with respect to the horizontal plane $\h^2 \times \{0\}$.
\item As $h \searrow 0$, $C_h$ converges to $\h^2 \times \{0\}$, branched at the origin, with multiplicity $2$.
\item As $h \nearrow \pi/2$, $C_h$ diverges to  $\partial_\infty \HH^2 \times [-\pi/2,\pi/2]$.
\end{enumerate}
\begin{remark} \label{re:catenoids} Let $T_{z_0}$ denote the horizontal dilation that maps $(0,t)$ into $(z_0,t) \in \h^2 \times \RR$, $$T_{z_0}(z,t):=\left( \frac{z+z_0}{\overline{z_0} z+1},t\right),$$
and define $C_{h,z_0}:=T_{z_0}(C_h)$. When $z_0$ is the origin we simply write $C_h$. Then the family
$\calM:=\{C_{h,z_0} \; : \; (z_0,h) \in \h^2 \times \RR\}$ forms a $3$-dimensional submanifold of the Banach manifold of all annuli.
\end{remark}

\subsection{Parabolic generalized catenoids.} \label{subsec:2} Although $C_h$ diverges as $h \to \pi/2$, one can obtain a nontrivial limit of this family as follows. For each $h$ apply a hyperbolic 
isometry $T_h$ of $\HH^2$, acting trivially on the $\RR$ factor, which translates along a fixed geodesic passing through the origin and ending at a point $q$ in $\partial_\infty \h^2$.  We
fix $T_h$ completely by demanding that $o \in T_h(C_h)$;  the tangent plane at that point is then necessarily vertical.  There exists
a nontrivial limit $\calD = \calD_q$ of the $T_h(C_h)$, as $h \to \pi/2$, discovered originally by Hauswirth \cite{h} and Daniel \cite{d}. Its asymptotic boundary consists of the two
circles $\sph^1 \times \{\pm \pi/2\}$ together with a vertical segment $\{q\} \times [-\pi/2, \pi/2]$.  It is foliated by horocycles 
$H_t = \calD \cap (\h^2 \times \{t\})$ based at the point $(q,t)$.  Note that applying other horizontal dilations along the same geodesic
produces a family of minimal surfaces with the same asymptotic boundary which foliate the slab $\h^2 \times (-\pi/2, \pi/2)$; one limit
of this family is the two disks $\h^2 \times \{ \pm \pi/2\}$. We shall refer them as {\em parabolic generalized catenoids}.
\medskip

These families of surfaces enjoy the following uniqueness properties: 
\begin{enumerate}
\item (Nelli, Sa Earp \& Toubiana \cite{n-s-t}): A minimal annulus bounded by $(\sph^1 \times \{\pm h\})$, for any $h \in (0, \pi/2)$, must 
equal $C_{h,z_0}$ for some $z_0 \in \h^2$. 
\item (Daniel, Hauswirth \cite{d, h}):  If $X$ is any (possibly incomplete) minimal surface in $\h^2 \times \R$, the intersection of which
with any horizontal plane $\h^2 \times \{t\}$ is a piece of a horocycle, then it must be a piece of some parabolic generalized catenoid. 
\end{enumerate}

These are interesting model examples and provide very useful barriers. 

\subsection{The Jacobi operator}
We now derive an explicit expression for the Jacobi operator $L$ on a catenoid $C$ and determine the space of decaying Jacobi fields.

We first recall the Jacobi operator
\[
L_C = \Delta_C + |S_C|^2 + \mathrm{Ric}(\nu, \nu),
\]
where $S_C$ is the shape operator (or second fundamental form) of $C$ and $\nu$ is the unit normal vector field to $C$. Rather than computing the last two terms explicitly, we use the coordinates and explicit form of the metric above to note that
\[
L_C= \frac{ (1-r^2)^2}{4r^2} ( \kappa^{-2} \del_t^2 + \del_\theta^2  +q(t)),
\]
for some function $q$, where $\kappa$ is the parameter given by \eqref{eq:kappa}. We can determine $q$ by plugging in a known solution of $L_C w = 0$, and we shall use the Jacobi field $w$ 
arising from vertical translations.

To calculate this Jacobi field, we first compute the unit normal to $C$,
\[
\nu = \left( \frac{ (1-r^2)^2}{4\kappa r} \cos \theta, \frac{ (1-r^2)^2}{4\kappa r} \sin \theta, - \frac{r'}{\kappa r} \right). 
\]
The Killing field generated by vertical translation is $(0,0,1)$, hence its projection onto $\nu$ is simply $-1/\kappa$ times the
function $r'/r$.  In other words,  $L_C (r'/r) = 0$. A short calculation then gives that
\[
q = \frac{1}{2\kappa^2} ( r^{-2} + r^2), 
\]
so altogether,
\begin{equation} \label{eq:jacob}
L_C = \frac{ (1-r^2)^2}{4\kappa^2 r^2} \left( \del_t^2 + \kappa^2 \del_\theta^2  + \frac12 (\frac{1}{r^2} + r^2) \right).
\end{equation}

Now set 
\[
\fJ(C) = \{\psi \in L^\infty(C): L_C \psi = 0\}, \qquad \fJ^0(C)= \fJ(C) \cap L^2(C);
\]
we shall actually consider only the subclass of Jacobi fields which extend to be $\calC^{2,\alpha}$ on $\overline{C}$,
but this will be discussed in detail in \S 6, along with many further properties of the Jacobi operator, both at $C$ and
at any other $A \in \fA$. The space $\fJ(C)$ is infinite dimensional and is (almost) parametrized by its 
asymptotic boundary values. We note one special fact that if $\phi \in \fJ^0(C)$, then $\phi$ is automatically smooth 
up to $\del C$ as a function of $(r,\theta)$, and the $L^2$ condition means that it vanishes like $1-r$. 

Expanding any function $u$ on $C$ as 
\[
u(t,\theta)=\sum_{n=0}^\infty ( \alpha_n(t) \cos (n \theta)+\beta_n(t) \sin (n \theta) ), \qquad \qquad (\mbox{with}\ \beta_0 \equiv 0),
\]
then
\[
L_C u = \sum_{n=0}^\infty (  (L_n \alpha_n) \cos (n \theta) + (L_n \beta_n) \sin (n \theta) ),
\]
where 
\[
L_n = \frac{ (1-r^2)^2}{4\kappa^2 r^2} \left( \del_t^2 - \kappa^2 n^2  + \frac12 (\frac{1}{r^2} + r^2) \right).
\]

\begin{proposition} \label{th:jacobin}
The space of decaying Jacobi fields $\fJ^0(C_h)$ is spanned by $\varphi_1 = \phi(r) \cos \theta $ and 
$\varphi_2 = \phi(r) \sin \theta$, where 
$$
\displaystyle \phi(r) := \frac{1}{r} - r. 
$$ 
These are the Jacobi fields generated by horizontal dilations. 
\end{proposition}
\begin{proof}
We must determine all solutions to $L_n u = 0$ with $u(\pm h) = 0$.   First observe that $L_1 \phi = 0$ where $\phi$ is
given in the statement of the theorem. The Sturm-Picone comparison theorem then gives that any solution of the 
Dirichlet problem for $L_n$, $n \geq 1$, must be proportional to $\phi$, which is impossible for $n \geq 2$. 

There is at most a two dimensional space of solutions to $L_0 u = 0$, and a basis for this space is given by the Jacobi fields 
generated by vertical translations and by varying the parameter $h$. We have already computed the first of  these, which 
is the function $r'/r$, which does not vanish at $\pm h$. It is easy to see that the second one cannot vanish at $\pm h$ either.
\end{proof}


\section{Graphical parametrization of horizontal ends}
\label{sec:graphs}
In this section we extend and sharpen a result of Collin, Hauswirth
and Rosenberg~\cite{chr} and prove that any properly
Alexandrov-embedded minimal annulus $A$ with embedded ends can be
written as a vertical graph near infinity.

One key tool in the argument below is the family of `tall rectangles'
obtained in~\cite{d,h,st} (see also~\cite{mrr2,st1}), which we now
recall. Let $\sigma$ be any connected arc in $\sph^1 = \partial_\infty \HH^2$ and
denote by $\eta$ the geodesic in $\HH^2$ with the same endpoints as
$\sigma$. Also fix $a, b \in \RR$ with $b-a > \pi$. Then there is a
minimal disk $R(\sigma, a, b)$ with asymptotic boundary the rectangle
$(\sigma \times \{a, b\}) \cup (\del \sigma \times [a,b])$, and such
that for any $t_0 \in (a,b)$, the projection onto $\HH^2$ of the
intersection $R(\sigma, a, b) \cap (\HH^2 \times \{t_0\})$ is a curve
equidistant from the geodesic $\eta$. Furthermore, $R(\sigma, a, b)$
is symmetric with respect to the horizontal slice at height $(a+b)/2$,
and is a vertical bigraph. If we denote by $\eta_0$ the projection of
this central slice, all other horizontal slices project
to curves `outside' $\eta_0$, i.e., in the component of
$\HH^2 \setminus \eta_0$ not containing $\eta$. The distance
between $\eta$ and $\eta_0$  tends to $0$ 
as $b-a \nearrow \infty$ so it makes sense to define $R(\sigma, -\infty, \infty)$ to 
be the vertical plane $\eta \times \RR$; \label{tallrectangles}
$R(\sigma, -\infty, b)$ and $R(\sigma, a, \infty)$ to be the vertical graphs over the 
domain bounded by $\eta \cup \sigma$ with the corresponding boundary values, 
called `semi-infinite tall rectangles'. The distance from 
the central slice $\eta_0$ to $\eta$ tends to infinity as $b-a \searrow \pi$, and in fact if we 
simultaneously let $\sigma$ 
converges to the entire circle and $b-a \searrow \pi$ then $R(\sigma, a, b)$ converges 
to a parabolic generalized catenoid, which is foliated by horocycles (see Section \ref{subsec:2}).  

\begin{figure}[htbp]
\begin{center}
\includegraphics[height=.36\textheight]{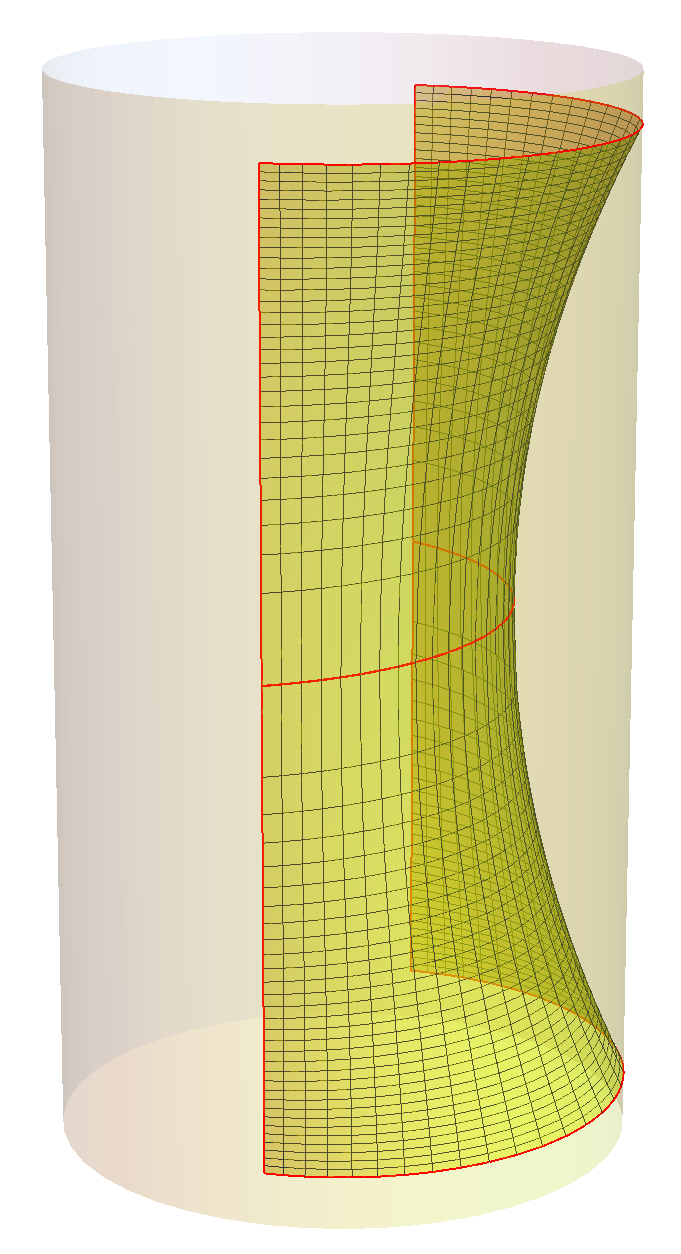}
\caption{The tall rectangle for $\sigma$ equals to half the
  circumference at infinity and $b-a=\frac95 \pi$.}
\label{graph-1}
\end{center}
\end{figure}

The other tool in this proof is the so-called `Dragging Lemma' (which
we state in a slightly simplified version adapted to our purposes),
inspired in Colding and Minicozzi ideas.

\begin{lemma}[\cite{chr}] \label{lem:drag} Let $E$ be a properly
  embedded annular end in $\h^2\times\R$ so that one component of its
  boundary, $\partiali E$, is a closed loop in the interior of this
  space and its virtual boundary at infinity,
  $\partial_\infty E = \gamma$, is a vertical graph of a continuous
  function $\gamma(\theta)$ over $\partial_\infty \HH^2$.  Suppose that $Y_s$ is a
  one-parameter family of compact minimal surfaces with boundary in
  $\HH^2 \times \RR$ such that $Y_s \cap \partial E = \varnothing$ and
  $\del Y_s \cap E = \varnothing$ for each $s\in[0,1]$.  Suppose that
  $p_0 \in Y_0 \cap E$.  Then there exists a continuous curve
  $s \mapsto p_s$ such that $p_s \in E \cap Y_s$ for each $s$ which
  equals $p_0$ at $s=0$.
\end{lemma}

We now turn to the main result of this section:

\begin{proposition}\label{prop:graph}
  Let $E$ be a properly embedded annular end as in the previous
  lemma. Then for a sufficiently large $R > 0$, there exists a
  function $u : \overline{\h^2} \setminus D_{\HH^2}(o,R) \to \RR$ with
  $u(1,\theta) = \gamma(\theta)$ such that the graph of $u$ equals
  $E \cap (( \overline{\h^2} \setminus D_{\HH^2}(o,R)) \times \RR)$.
\end{proposition}

\begin{remark} \label{re:extension}
  The graph function $u$ is smooth in the interior and as regular at
  $\del_\infty E$ as the function $\gamma(\theta)$, see the proof of Theorem \ref{manifoldthm} in Section
  \S \ref{sec:manifold}.
\end{remark}

\begin{proof}
  In the following we fix Euclidean coordinates $z = re^{i\theta}$ on
  $\overline{\HH^2}$, and when we refer to the length of an arc on
  $\partial_\infty \HH^2$, we mean with respect to the Euclidean metric and these
  coordinates.

  For each $\epsilon > 0$, there exists $\delta > 0$ so that if
  $\sigma$ is an arc in $\sph^1$ of length less than $\delta$, then
  $\mathrm{osc}_\sigma \gamma \ (= \sup_\sigma \gamma - \inf_\sigma
  \gamma) < \epsilon$.
  If the length of $\sigma$ is sufficiently small, then the
  semi-infinite tall rectangles
  $R(\sigma, -\infty, \inf_\sigma \gamma)$ and
  $R(\sigma, \sup_\sigma \gamma, \infty)$ do not intersect
  $\partiali E$, and by the maximum principle, neither of these
  intersect $E$ in its interior.

  Fixing $\epsilon > 0$ and a large constant
  $C_1<\frac{\pi}{4\epsilon}$, there exists a curve $\tilde{\eta}$
  equidistant from the geodesic $\eta$ associated to $\sigma$ so that
  in the lens-shaped region $D_\sigma$ between $\sigma$ and
  $\tilde{\eta}$ the vertical distance between these upper and lower
  semi-infinite tall rectangles is less than $C_1 \epsilon$ (see
  Fig. \ref{fig:Dsigma}.)  We then cover $\sph^1$ by finitely many
  such arcs $\sigma$; the union of the corresponding lenses $D_\sigma$
  covers an outer annular region $\HH^2 \setminus D_{\HH^2}(o, R_0)$,
  and in this region, the difference between the maximum and minimum
  height of $E \cap (\{p\} \times \RR)$ is less than $C_1\epsilon$.
  Clearly $E \cap ((\HH^2 \setminus D_{\HH^2}(o, R_0)) \times \RR)$ is
  trapped in the region between the union of the upper and of the
  lower semi-infinite tall rectangles.

  \begin{figure}[htpb]
    \begin{center}
      \includegraphics[height=.2\textheight]{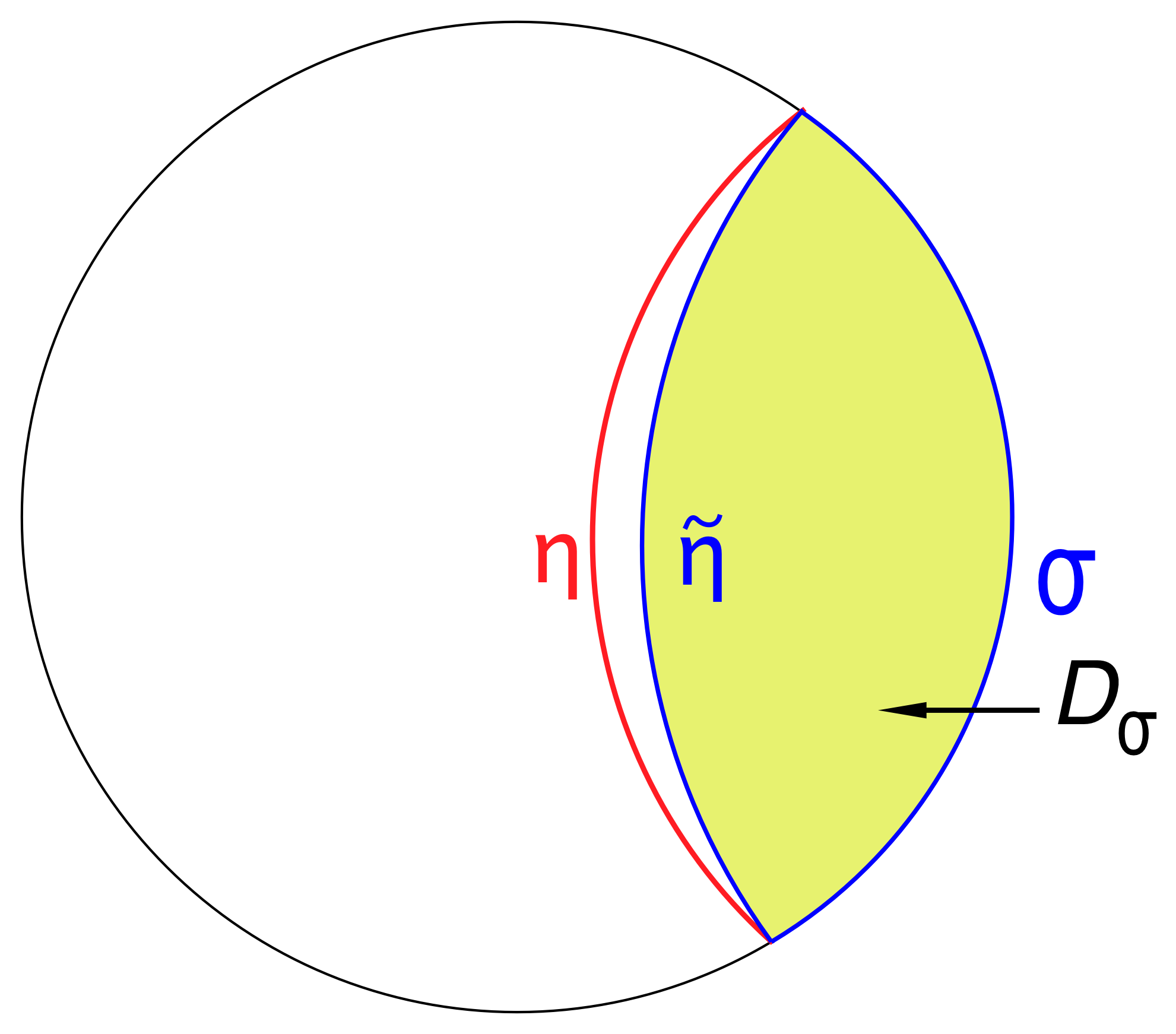}
    \end{center}
    \caption{ The region $D_\sigma$. }
    \label{fig:Dsigma}
  \end{figure}

  Next fix a truncated vertical catenoid $C$, i.e., the intersection
  of the catenoid centered on the axis $\{o\} \times \RR$ of height
  $4C_1\epsilon$ with the cylinder $D_{\HH^2}(o, \rho) \times \RR$.
  We choose $\rho$ sufficiently large so that the vertical separation
  between the upper and lower boundaries of $C$ is greater than
  $3C_1\epsilon$.  Denote by $C(q,\tau)$ the translate of this
  truncated catenoid by isometries of $\HH^2 \times \RR$ so that it is
  centered at $(q,\tau) \in \HH^2 \times \RR$.  By the construction
  above, we may choose a radius $R_1 \gg R_0$ so that
  $\partiali E \subset D_{\HH^2}(o,R_1) \times \R$, and a continuous
  function $q \mapsto \tau(q)$,
  $q \in \HH^2 \setminus D_{\HH^2}(o, R_1)$, satisfying 
\begin{equation} \label{eq:dobleonada} C(q, \tau(q)) \cap \partiali
    E = \varnothing, \mbox{ and
      $\del C(q, \tau(q)) \cap E = \varnothing$}
  \end{equation}
  for every $q$ in this exterior region
  $\HH^2 \setminus D_{\HH^2}(o, R_1)$.

  Denote by
  $E_R := E \cap (({\HH^2} \setminus D_{\HH^2}(o,R)) \times \RR)$ the
  region in $E$ outside a large cylinder. To prove
  Proposition~\ref{prop:graph}, we must verify the following two
  assertions for $R \gg R_1$:
  \begin{enumerate}
  \item[i)] the projection $E_{R} \to \HH^2 \setminus D_{\HH^2}(o,R)$
    is bijective and
  \item[ii)] there are no points $(q,t) \in E_R$ such that
    $T_{(q,t)} E$ is vertical. 
  \end{enumerate}

  First note that i) is a consequence of ii).  Indeed, if the
  projection of $E_R$ does not contain a full neighborhood of
  infinity, then there exists a sequence of points $q_j\in\HH^2$ which
  tend to infinity and which do not lie in this image. Because $E$ is
  properly embedded, its projection on $\HH^2$ is closed, so for each
  $j$ there exists $\kappa_j > 0$ such that $D_{\HH^2}(q_j, \kappa_j)$
  is also disjoint from this image. Next, by translating the center
  $q_j$ in the component of the complement of the image of the
  projection of $E$, we may arrange that
  $\del D_{\HH^2}(q_j, \kappa_j) \times \RR$ is tangent to $E$ at some
  point, and clearly the tangent plane of $E$ must be vertical there.
  This proves the surjectivity of this projection. Furthermore, if ii)
  holds, then the projection is a covering map, and properness of the
  embedding prevents there from being more than one sheet. Hence it
  must be bijective, and therefore a diffeomorphism.

  \medskip
  
  We therefore turn to assertion ii). Choose $R_2>R_1+2\rho$ and let
  $K_0=(\overline{D_{\h^2}(o,R_2)}\times \R)\cap E$ be the
  corresponding compact portion of $E$. Properness of $E$ guarantees
  that $K_0$ has a finite number of connected components and also that
  there exists a connected compact subset $K \subset E$ such that
  $K_0 \subset K$. In particular, any two points in $K_0$ can be
  joined by a path contained in $K$. 

  Now suppose there exists a point $(q,t) \in E \setminus K$ such that
  $T_{(q,t)} E$ is the vertical plane $\wt{\eta} \times \R$, where
  $\wt{\eta}$ is a geodesic in $\h^2$.  Transform the whole ensemble
  by an isometry $T$ carrying $\wt{\eta}$ to the geodesic
  $\eta=\{z \in \h^2, \Re(z)= 0 \}$, with $T(q,t)= (o, 0)$; we also
  write $T(o,0)=(q_0,t_0)$ for the image point, which lies in
  $\{z \in \h^2, \Re(z)\geq 0 \} \times \R$.  Note that there exists
  $C_2 > 0$ so that $T(E)$ lies in the horizontal slab
  $\h^2 \times [-C_2,C_2]$.

  Fix $\delta' > 0$ and consider the geodesic
  $\wh{\eta} \subset \{z \in \h^2, \Re(z)<0 \}$ which is orthogonal to
  $\{\Im (z) = 0\}$ and at hyperbolic distance $\delta'$ from $\eta$
  (thus $o$ is the closest point on $\eta$ to $\wh{\eta}$). Also let
  $\wh{\sigma}\subset\partial_\infty \HH^2$ be the boundary arc connecting the
  endpoints of $\wh{\eta}$ and containing the endpoints of $\eta$. We
  denote by $\wh{\eta}(\delta')$ the curve equidistant from
  $\wh{\eta}$ at distance $\delta'$ and on the same side of
  $\h^2 \setminus \hat{\eta}$ as $\eta$  (see
  Fig. \ref{fig:lunette}.) There is a unique tall
  rectangle $R(\hat{\sigma}, -h', h')$ which meets
  $\HH^2 \times \{0\}$ at $\hat{\eta}(\delta')$, and is tangent to
  $\eta \times \RR$, and hence $T(E)$, at $(o,0)$.  Note that
  $\delta'$ and $h'$ are determined by one another, and
  $h' \to \infty$ as $\delta' \to 0$.
 \begin{figure}[htpb]
    \begin{center}
      \includegraphics[height=.25\textheight]{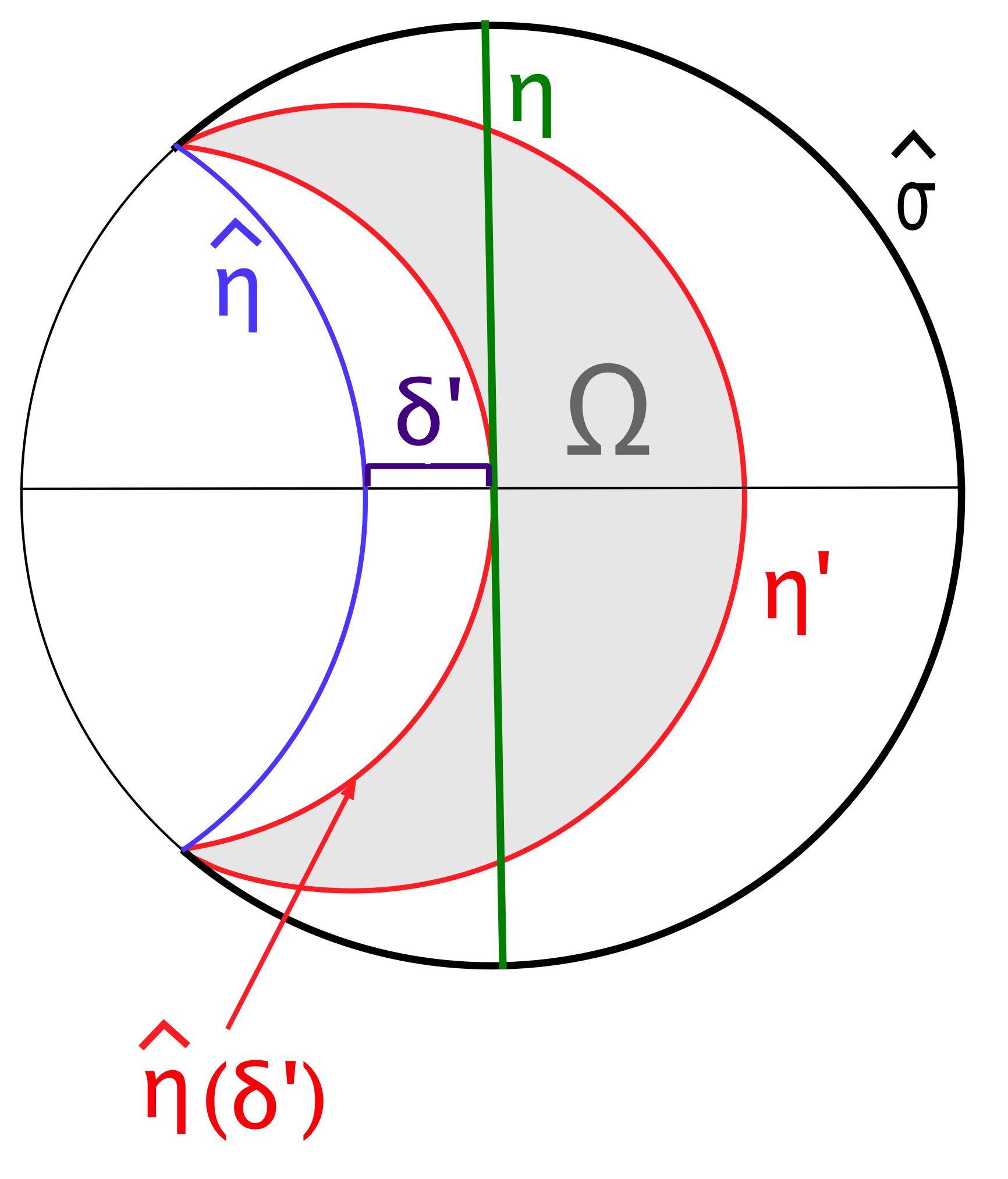}
    \end{center}
    \caption{ The lunette region $\Omega$ between the core curve
      $\hat{\eta}(\delta')=R(\hat{\sigma}, -h', h') \cap \{t=0\}$ and
      the equidistant curve $\eta'$, which is the projection to
      $\HH^2$ of $R(\hat{\sigma}, -h', h') \cap \{t = \pm 2C_2\}$. }
    \label{fig:lunette}
  \end{figure}
  \begin{claim} \label{claim:yes}
    For $h'$ and $\dist_{\h^2}(o,q_0) (=\dist_{\h^2}(q,o))$
    sufficiently large, $\mathcal{R}=R(\hat{\sigma}, -h', h')$ does
    not intersect $T(K)$.
  \end{claim}
  Indeed, suppose that $h'> 3 C_2$ and $\eta'$ be the curve
  equidistant from $\hat{\eta}$ which is the projection to $\h^2$ of
  $\mathcal{R}\cap \{t=\pm 2 C_2\}$.  Let $\Omega$ be the region
  between $\hat{\eta}(\delta')$ and $\eta'$ (see
  Fig. \ref{fig:lunette}.) We know that
  $T(E)\cap\mathcal{R}\subset \Omega \times \R$, as
  $T(E)\subset\h^2 \times [-C_2,C_2]$.  Furthermore,
  if $\dist_{\h^2}(o,q_0)$ is large enough,
  $(\Omega \times \R) \cap T(K)=\varnothing$, and Claim
  \ref{claim:yes} follows.
  
  Let $\mathcal{R}^\pm$ be the connected components of
  $(\h^2 \times \R) \setminus \mathcal{R}$, and assume that
  $T(K) \subset \mathcal{R}^+$.  Locally around $(o,0)$,
  $T(E) \cap \mathcal{R}$ is the union of $n$ smooth curves meeting at
  equal angles.  There are many foliations of $\h^2 \times (-h',h')$ by
  families of tall rectangles of which $\mathcal{R}$ is one
  element. These may be used to sweep out either
  $\mathcal R^\pm\times (-h',h')$, and together with the maximum
  principle show that $T(E) \cap \mathcal{R}$ cannot bound a disk
  lying either in $\mathcal{R}^+$ or $\mathcal{R}^-$. In particular,
  $T(E) \cap \mathcal{R}^+$ has at least two distinct connected
  components $\Sigma_1$, $\Sigma_2$.
 
  If $\lambda>0$ is arbitrary, denote by
  $\mathcal{R}(\lambda) \subset \mathcal{R}^+$ the hyperbolic
  translation of $\mathcal{R}$ along $\{\Im (z) = 0\}$ by the distance
  $\lambda$. We also write
  $\mu=\mathcal{R}(\lambda)\cap (\h^2 \times \{0\})$ and
  $\mathcal{R}(\lambda)^+$ for the connected component of
  $(\h^2\times\R)\setminus \mathcal{R}(\lambda)$ in $\mathcal{R}^+$.
  When $\lambda$ is sufficiently small, $\mathcal{R}(\lambda)$ meets
  both of the $\Sigma_j$, and in addition,
  $T(K) \subset \mathcal{R}(\lambda)^+$.

  Let $\mathcal{U}$ be a connected component of
  $\Sigma_i \cap \mathcal{R}(\lambda)^+$, $i = 1$ or $2$. If
  $\overline{\mathcal{U}}$ is compact, then the same argument
   as above shows that the case where
  $\partial \mathcal{U} \subset \mathcal{R}(\lambda)$ is
  impossible. We thus turn to the remaining case
  $\partiali T(E) \subset \partial \mathcal{U}$. We can then consider
  a continuous path in $\mathcal{U}\subset T(E) \cap \mathcal{R}^+$
  from a point $x_i\in\Sigma_i \cap \mathcal{R}(\lambda)$ to a point
  $y_i\in \partiali T(E)\subset T(K_0)$.

  Suppose now that $\overline{\mathcal{U}}$ is non-compact.  We are
  also going to construct a continuous path in
  $T(E) \cap \mathcal{R}^+$ from a point $x_i\in\Sigma_i$ to a point
  in $y_i\in T(K_0)$. Adjoining this to a continuous path in $T(K)$
  connecting $y_1$ and $y_2$ gives a continuous path in
  $T(E)\cap\mathcal{R}^+$ between  $x_1 \in \Sigma_1$ and
  $x_2\in \Sigma_2$. This contradicts that $\Sigma_1$ and
  $\Sigma_2$ are different components of $E \cap \mathcal{R}^+$. Hence
  there would not exist a point $(q,t) \in E \setminus K$ whose
  tangent plane $T_{(q,t)} E$ is the vertical, if $\dist_{\h^2}(q,o)$
  is large enough, and assertion ii) would follow.  Then in order to
  complete the proof of assertion ii) is suffices to construct such a
  path in the case $\overline{\mathcal{U}}$ is non-compact.

  Since $\overline{\mathcal{U}}$ is non-compact, there exist a point
  $x_i\in\Sigma_i\cap \mathcal{R}(\lambda)^+$ sufficiently far from both
  $\mathcal{R}$ and $T(D_{\HH^2}(o,R_1))\times \R$ so that a truncated
  catenoid $C(\hat x _i, \tau(\hat x _i))$ passes through $x_i$, for
  some point $\hat x _i\in \HH^2 \setminus T(D_{\HH^2}(o, R_1))$.  Let
  $\xi_i(s)$ be a curve in $\h^2\times \mathcal{R(\lambda)}^+$ joining
  $\hat x_i$ to a point in $T(D_{\HH^2}(o,R_1+\rho))\times \R)$.  We
  can assume that $\xi_i$ is at a distance bigger than $2\rho$ from
  $\mathcal{R}$ at any point.  From \eqref{eq:dobleonada}, the
  translated catenoids $C(\xi_i(s), \tau(\xi_i(s)))$ satisfy
  \[
  C(\xi_i(s), \tau(\xi_i(s))) \cap \partiali T(E)= \varnothing \quad {\rm and} \quad \partial
  C(\xi_i(s), \tau(\xi_i(s))) \cap T(E) = \varnothing.
  \]
  Lemma \ref{lem:drag} shows that there exists a continuous curve
  $s \mapsto p_s$ such that $p_s \in E\cap C(\xi_i(s), \tau(\xi_i(s)))$ for each $s$. This
  gives a continuous path in $E \cap \mathcal{R}^+$ from a point $x_i$
  in $\Sigma_i$ to a point $y_i \in T(K_0)\subset T(K)$, as desired.
\end{proof}

\begin{remark} We observe that the graphical behaviour of the asymptotic boundary of the annulus has only been used
to ensure the existence of truncated catenoids satisfying \eqref{eq:dobleonada}.  This conclusion may also be obtained
in the following setting where somewhat less is known. Consider two functions $\alpha^\pm  \in \calC^{2 , \alpha}(\sph^1)$ 
such that $0 < \alpha^+(\theta)- \alpha^-(\theta)<\pi$ for all $\theta \in \sph^1$. Suppose also that $E$ is a properly 
embedded minimal annulus with one compact boundary component and $\del_\infty E$ lies between the graphs of $\alpha^\pm$, 
i.e., $\del_\infty E \subset \{ (\theta,t) \in \sph^1 \times\R \; : \; \alpha^-(\theta) < t <\alpha^+(\theta)\}$.  
Then $E$ is graphical in some region $\{|z| > 1-\epsilon\}$. In particular, necessarily $\del_\infty E$ is
a vertical graph. 

If we remove the hypothesis of embeddedness in Proposition~\ref{prop:graph}, then the assertion is not
longer necessarily true.  However,  the proof above still shows that for small enough $\epsilon > 0$, 
there is no point in $E \cap (\{ |z| > 1-\epsilon\} \times \R)$
where the tangent plane to $E$ is vertical. This means that near infinity, $E$ is a multi-graph. 
\end{remark}

We conclude this section with a closely related result about the shape of the set of points on $A$ where the tangent plane is vertical.
\begin{proposition} \label{prop:V}
Let $A$ be a properly Alexandrov-embedded, minimal annulus with embedded ends such that $\Pi(A)=(\gamma^+,\gamma^-)$ consists of two 
$\calC^{2, \alpha}$ graphs over $\sph^1$. Then the set $V$ of all points on $A$ where the tangent plane is vertical (or equivalently, 
where the normal has no vertical component) is a regular curve which generates $H_1(A,\z)$. Moreover, the Gauss map of $A$ 
restricted to $V$, $\nu|_V$, is a diffeomorphism from $V$ to the equator of the unit sphere $\sph^2$.
\footnote{The Gauss map is well defined in any Cartan-Hadamard manifold since there is a natural identification of the unit sphere bundle 
at a point with the sphere at infinity. In the present setting the horizontal equator is also well-defined since it corresponds to the unit normals
which have no $\RR$ component.}
\end{proposition}
\begin{proof}
Suppose $H_s$ is a smooth `sweepout' of $\HH^2\times \RR$ by vertical planes. In other words, the $H_s$ are leaves of a smooth 
foliation. Assuming that the parameter $s$ varies over $\RR$, we define a height function $K: \HH^2 \times \RR \to \RR$ by 
setting $H_s = \{K= s\}$. Let $K_A$ denote the restriction of $K$ to $A$.  We claim that $K_A$ is a Morse function with precisely 
two critical points, each of index $1$. 
\begin{figure}[htbp]
    \begin{center}
        \includegraphics[width=.55\textwidth]{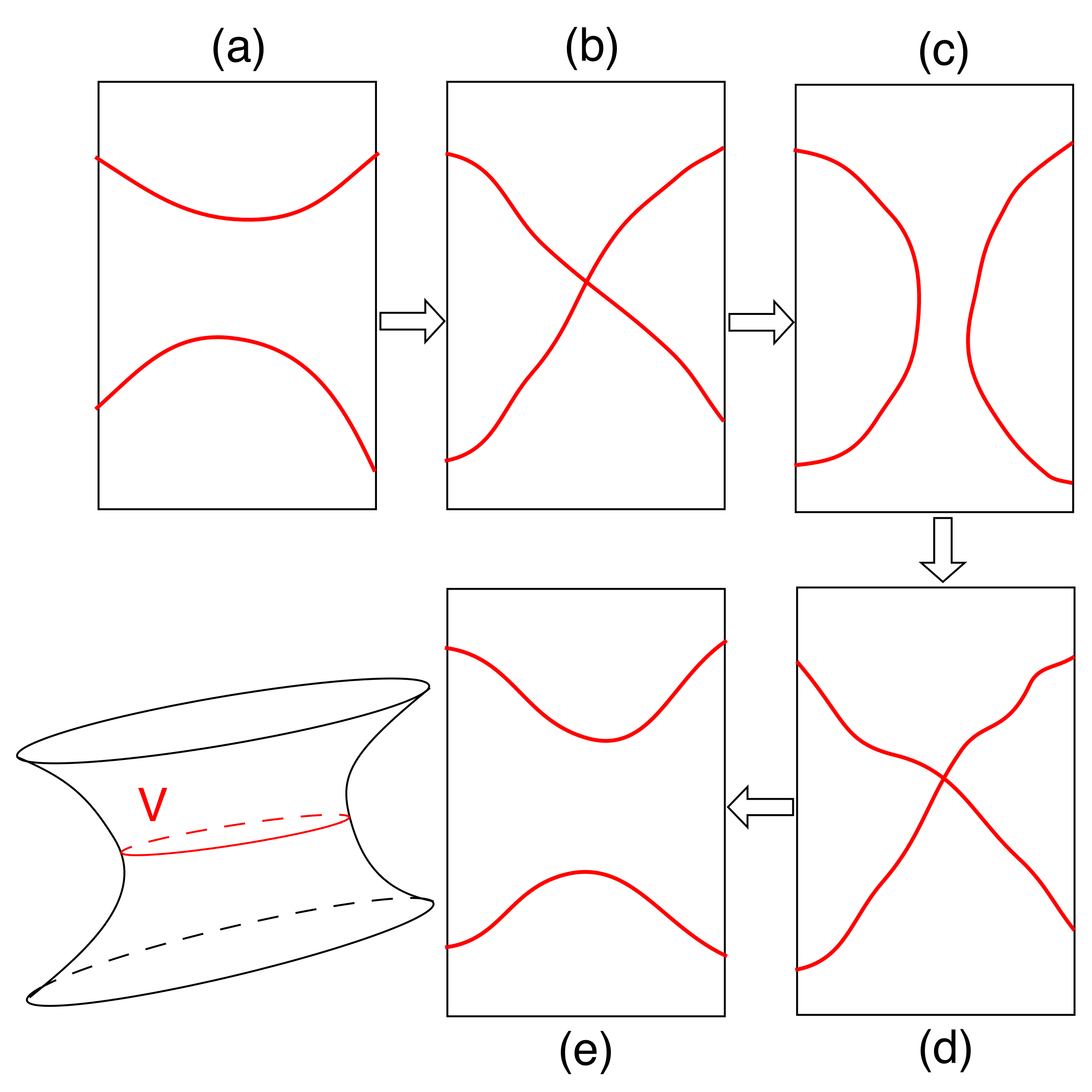}       
   \end{center}
   \caption{Evolution of $H_{s}\cap A$ when there are no loops. } \label{fig:good}
\end{figure} 

The graphical representation theorem proved in this section shows that for $s$ very negative, $H_s \cap A$ is a union of two arcs, 
each one lying in an end of  $A$. In fact,  for any value of $s$, $H_s \cap A$ intersects a neighborhood of $\partial_\infty \HH^2 \times \RR$
in four simple arcs, two arriving in $\gamma^+$ and two arriving in $\gamma^-$.  Now as $s$ increases from $-\infty$ there is a 
point of first tangency with $A$, say at $s = s_1$, which occurs at a point $p_1 \in V$.   Locally around $p_1$, $H_{s_1} \cap A$ is 
a union of $\ell$ curves intersecting at equal angles for some $\ell \geq 2$. 

If there is not any closed loop then the shape of $H_{s_1} \cap A$ is as in Figure \ref{fig:good}-(b). In particular $\ell=2$. If there is 
$\sigma$ a closed loop in $H_{s_1}$, then (by the maximum principle) $\sigma$ cannot bound a compact disk in $A$. So $\sigma$ 
generates $H_1(A,\z)$. By the maximum principle again, there is only one such loop. In particular, $\ell \leq 3$. 
If $\ell=3$ then $\sigma$ separates $\gamma^+$ and $\gamma^-$ and so it separates the diverging arcs in $H_{s_1} \cap A$, two on 
each region of $H_{s_1} \setminus \sigma$. But this is absurd because one of the components of $H_{s_1} \setminus \sigma$ is compact. 
Hence $\ell=2$ and  the shape of $H_{s_1} \cap A$ is either as in Figure \ref{fig:bad}-(b) or as in Figure \ref{fig:bad-1}-(c).
\begin{figure}[htbp]
    \begin{center}
        \includegraphics[width=.55\textwidth]{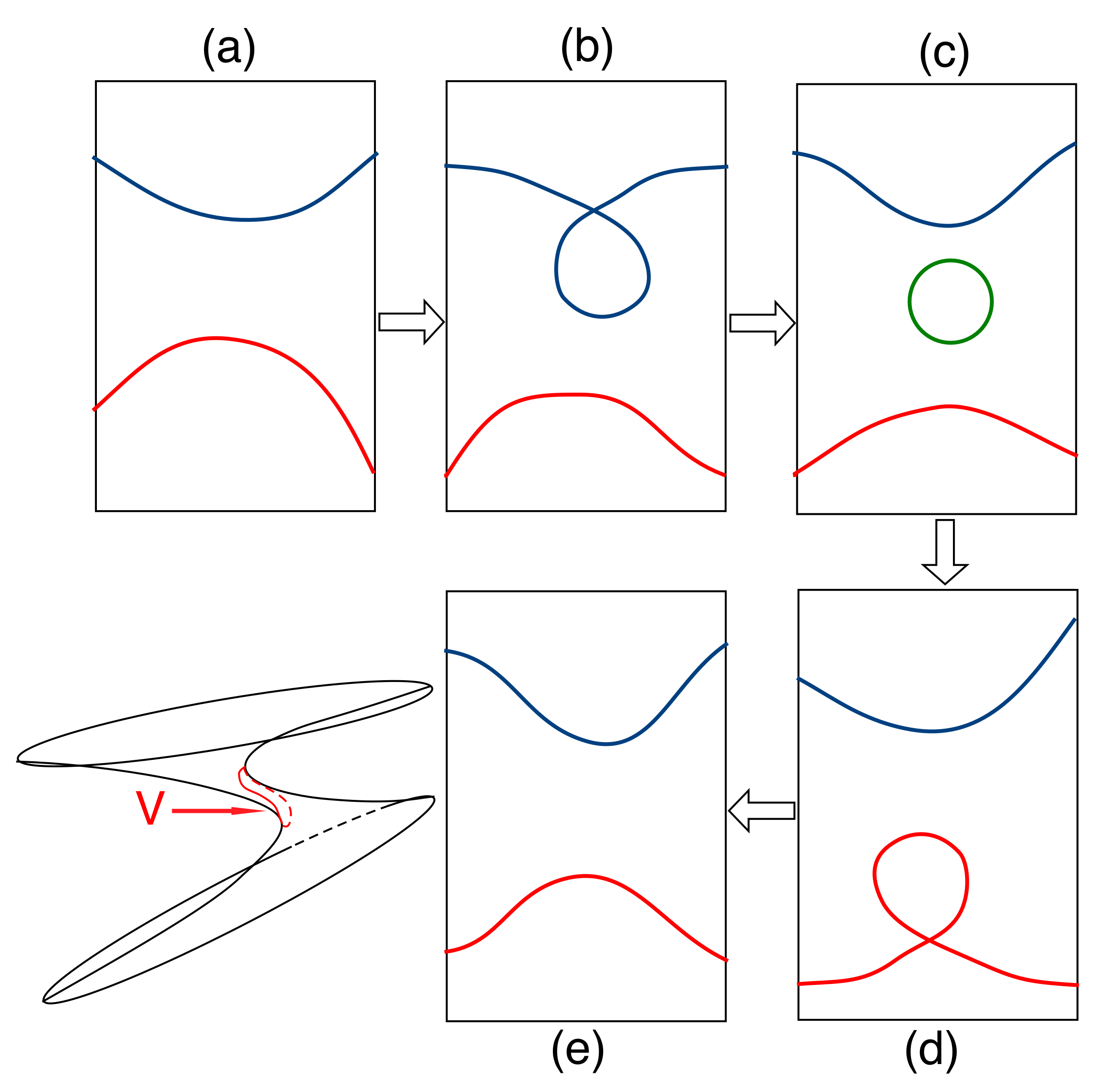}     
     \end{center}
   \caption{Possible evolution of $H_{s}\cap A$  and the curve $\sigma$, when $\sigma$ is a generator
   of $H_1(A,\z)$.} \label{fig:bad}
\end{figure} 
\begin{figure}[htbp]
    \begin{center}
        
              \includegraphics[width=.55\textwidth]{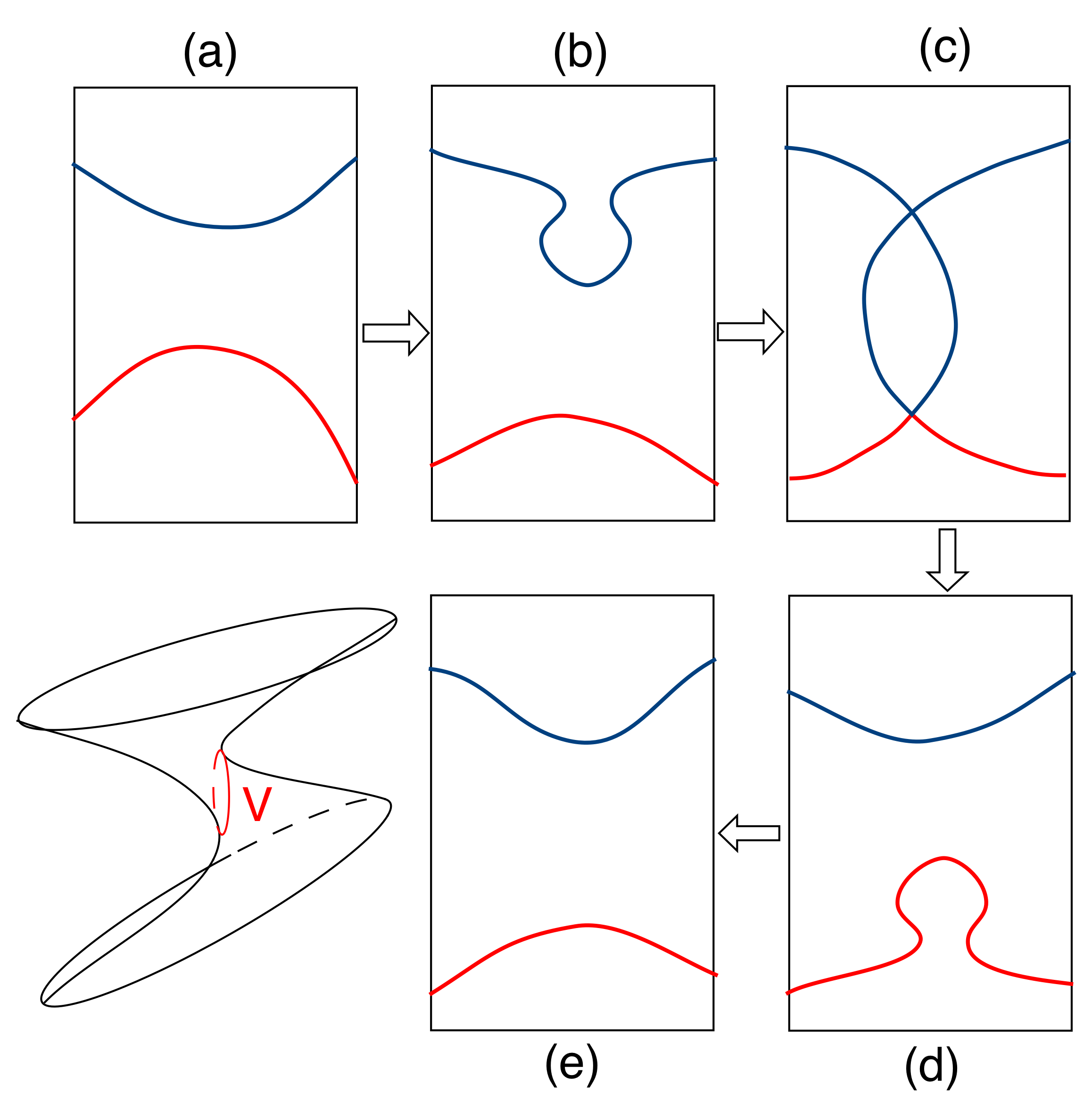}   
  
     \end{center}
   \caption{The other possible evolution of $H_{s}\cap A$  and the curve $\sigma$, when it contains a generator
   of $H_1(A,\z)$.} \label{fig:bad-1}
\end{figure} 

In any case $\ell = 2$, and hence this is a simple tangency, or 
in other words, the function $K_A$ has a nondegenerate critical point of index $1$ at $s_1$.  Letting $s$ increase further, we encounter some number of other
critical points $p_2, \ldots, p_r$, at the values $s = s_2, \ldots, s_r$, each one of which corresponds to another nondegenerate critical point of index $1$ 
of $K_A$. 

To conclude, observe that we can apply the standard Morse-theoretic arguments to see how these critical points correspond to a decomposition
of $A$ into a union of cells. 

In the first case (Figure \ref{fig:good}), for $s$ very negative, $A \cap \{K_A \leq s\}$ is a union of two disks. The transition between the sublevel $K_A \leq s_1 - \epsilon$
and $K_A \leq s_1 + \epsilon$ corresponds to attaching a two-cell which connects these two disks, resulting in another (topological) disk.  Crossing
the next critical point, another two-cell is added, which changes the topology again. Each of the remaining critical points add further handles. 
However, since $A$ is an annulus, and in particular has genus $0$, we must have $r = 2$. Hence there are precisely two critical points.

In the second case (Figure \ref{fig:bad}), again we have that for $s$ very negative, $A \cap \{K_A \leq s\}$ is a union of two disks. Now, the transition between the sublevel $K_A \leq s_1 - \epsilon$
and $K_A \leq s_1 + \epsilon$ corresponds to attaching a two-cell which connects one of the disks
with itself, resulting in a topological annulus. Crossing
the next critical point, another two-cell is added, which connects the annulus with the other disk. Again we deduce that $r=2$. Using similar arguments we deduce that $r=2$ in the last
case (Figure \ref{fig:bad-1}).

Note that since both these critical points are nondegenerate, the set of points $p\in V$ near either $p_1$ or $p_2$ constitute a regular curve.

We may of course do this for any sweepout of $\HH^2 \times \RR$ by vertical planes.  This shows that in any direction, there are precisely two critical
points and hence $V$ is a regular curve which generates $H_1(A,\z)$. \end{proof}


\section{Fluxes of minimal annuli} \label{subsec:flux}
Let $A$ be a minimal surface lying in an ambient space which has continuous families of isometries, and 
$\gamma$ any closed curve in $A$.   For any Killing field $Z$ on the ambient space, the flux of $A$ across $\gamma$, 
$\mbox{Flux}\,(A,\gamma, Z)$, depends only on the homology class of $\gamma$ in $A$. This flux is defined 
by integrating $\langle Z, \nu \rangle$ around $\gamma$, where $\nu$ is the unit normal to $\gamma$ in $A$.
We now compute these flux integrals for minimal annuli in $\h^2 \times \RR$ with horizontal ends, where 
$\gamma$ is the generating loop for the homology $H^1(A,\z)$. These invariants play an important role later
in this paper. 

Fix $A \in \fA$, $\Gamma= (\gamma^+,\gamma^-) = \del A$.  As proved in \S \ref{sec:graphs} (Proposition 
\ref{prop:graph}
and Remark \ref{re:extension}), each end of $A$ is 
a vertical graph over some region $\{z: R \leq |z| \leq  1\}$, with graph functions $u^\pm$, so we use the graphical representations
\[
X^\pm(r,\theta)=(r \cos \theta, r \sin \theta, u^\pm(r,\theta)), \quad R \leq r \leq 1,\ \ \theta \in \sph^1. 
\]

Consider the restriction of $X^\pm$ to $\beta^\pm:=X^\pm \cap \{r = \mathrm{const.}\}$ and the orthonormal frame 
\[
E_1=(\sqrt{F},0,0) ,\quad E_2=(0,\sqrt{F},0) ,\quad E_3=(0,0,1),  
\]
where $F=\frac14 (1-r^2)^2$.  Evaluating all functions at this fixed value of $r$, 
\[
 \beta^\pm_\theta(\theta)=X_\theta^\pm (r,\theta)= - \frac{r \, \sin \theta }{\sqrt{F}}\, E_1+ \frac{r \,  \cos \theta }{\sqrt{F}}\, E_2 + u^\pm_\theta\, E_3,
\]
hence the unit tangent to $\theta \mapsto \beta^\pm$ is
\[
T^\pm(\theta)= \frac{\beta^\pm_\theta}{ \| \beta_{\theta}^{\pm} \|_g }, \quad \left\| \beta^\pm_\theta
\right\|_g= \frac{ \sqrt{r^2+ F \, (u^\pm_\theta)^2}}{\sqrt{F}}.
\]
Similarly, the normal to this curve in $A$ equals 
\[
\nu^\pm(\theta) = \pm Q \left( \left(\frac{r\cos\theta}{\sqrt{F}}+\sqrt{F} u^\pm_\theta u^\pm_{x_2}\right) E_1+ 
\left(\frac{r\sin\theta}{\sqrt{F}}-\sqrt{F} u^\pm_\theta u^\pm_{x_1}\right) E_2 + r u^\pm_r E_3\right),
\]
where 
$
Q = \left(\|\beta^\pm_\theta\|_g \sqrt{1+F \|\nabla u^\pm \|^2}\right)^{-1}.
$

We now compute the fluxes with respect to $E_3$ and the horizontal Killing fields generated by rotations and hyperbolic dilations. 

The first is the simplest: 
\begin{multline*}
\text{ Flux} (A, \beta^\pm,E_3)= \int_{\beta^\pm} \langle \nu^\pm, E_3 \rangle \, d \sigma= \\
\pm  \int_0^{2 \pi} \frac{ r u^\pm_r}{\sqrt{1+F \|\nabla u^\pm \|^2}}\, d \theta= 
\pm \int_0^{2 \pi} \frac{r u^\pm_r}{\sqrt{1+\frac14 (1-r^2)^2\left((u^\pm_r)^2+ r^{-2} (u^\pm_\theta)^2\right) }}\, d \theta.
\end{multline*}
This is constant in $r$, and the limit as $r \nearrow 1$ equals 
\[
\text{ Flux} (A,\gamma^\pm,E_3)= \pm \int_0^{2 \pi} u^\pm_r(1,\theta) \, d\theta .
\]

Next, the Killing field generated by rotations around the vertical axis $\{0\} \times \RR$ is
\[
Z=-\frac{r \sin\theta}{\sqrt{F}}\, E_1+ \frac{r\cos\theta}{\sqrt{F}}\, E_2, 
\]
and we have
\[
\text{ Flux} (A,\beta^\pm,Z)= \mp  \int_0^{2 \pi}\frac{r\, u^\pm_\theta\, u^\pm_r}{\sqrt{1+F \|\nabla u^\pm  \|^2}}\, d \theta.
\]
Letting $r \nearrow 1$ as before gives that
\begin{equation}
\text{ Flux} (A,\gamma^\pm,Z)=\mp \int_0^{2 \pi} u^\pm_r(1,\theta) \; u^\pm_\theta(1,\theta) \, d \theta.
\end{equation}

Finally, consider the Killing fields 
\[
H_a= \frac{r^2 \cos (a-2 \theta)-\cos (a)}{2\sqrt{F}}\,  E_1-
\frac{r^2\sin (a-2 \theta)+\sin (a)}{2\sqrt{F}}\,  E_2,
\]
$a \in [0,2\pi)$, corresponding to the horizontal dilations along the geodesic joining $\text{e}^{{\rm i} a} $ and $-1$.  We calculate
\begin{multline*}
\text{ Flux} (A,\beta^\pm,H_a)= \int_{\beta^\pm}  \langle \nu^\pm, H_a \rangle d \sigma\\
\ =\pm \int_0^{2 \pi} \frac{1}{2\sqrt{1+F \|\nabla u^\pm \|^2}}  \left(
\left( \frac{4r}{r^2-1}+\frac{(r^2-1)(u^\pm_\theta)^2}{r}\right) \cos(a-\theta) \right.\\
+(r^2+1)u^\pm_ru^\pm_\theta\sin(a-\theta)\bigg) \, d \theta .
\end{multline*}

Unlike the previous cases we cannot take limits directly since the first term appears to diverge. However, 
\[
\frac{1}{2\sqrt{1+F \|\nabla u^\pm
    \|^2}}\left(\frac{4r}{r^2-1}+\frac{(r^2-1)(u^\pm_\theta)^2}{r}\right)=
\frac{1}{r-1}+\frac{1}{2}+\calO (1-r) .
\]
When we multiply by $\cos (\theta -a )$ and integrate in $\theta$, the first two terms on the right vanish, while the third
vanishes in the limit as $r \nearrow 1$. Therefore only the final term remains and we obtain that
\[
\text{ Flux} (A,\gamma^\pm,H_a)= 
\pm \int_0^{2 \pi} u^\pm_r(1,\theta) \; u^\pm_\theta(1,\theta) \sin (\theta-a) \, d \theta.
\]

These computations prove the following
\begin{lemma}  Let $A$ be a complete properly embedded minimal annulus with horizontal ends, parametrized as above. Then
\begin{eqnarray}
& \int_0^{2 \pi} u^+_r(1,\theta) d \theta + \int_0^{2 \pi}
  u^-_r(1,\theta) d \theta = 0, \\ & \int_0^{2 \pi} u^+_r(1,\theta) \;
  u^+_\theta(1,\theta) \, d \theta + \int_0^{2 \pi} u^-_r(1,\theta) \;
  u^-_\theta(1,\theta) \, d \theta = 0, 
\end{eqnarray}
and for every $a \in [0,2\pi)$,
\begin{equation} 
\int_0^{2 \pi}   u^+_r(1,\theta) u^+_\theta(1,\theta) \sin (\theta-a) \, d \theta +
  \int_0^{2 \pi} u^-_r(1,\theta) u^-_\theta(1,\theta) \sin (\theta-a)
\, d \theta = 0,
\end{equation}
\end{lemma}

\section{Nonexistence} \label{sec:nonexistence}
We present here two separate results which limit the types of pairs of curves which can arise as boundaries of minimal annuli.
The first proof is based on a standard barrier argument and the second on Alexandrov reflection principle. 
Among other things, we are going to prove in this section that there are no properly embedded minimal surfaces with two consecutive horizontal annular ends when these ends are more than $\pi$ apart. Observe that if there exists a horizontal slab of height greater than $\pi$ separating the boundaries, then the result can be easily proved by using the maximum principle and the family of catenoids described in \S \ref{sec:2}. However, the general case is more involved. 

In order to prove this, we need  to show that two tall rectangles of the same height with the vertical segments
in common are not area-minimizing\footnote{A noncompact surface is called area minimizing if any compact domain minimizes the area among all the surfaces with the same boundary.} (see Appendix \ref{sec:tall}.)

We would like to remind 
that we are denoting the finite boundary of a surface $\Sigma$ as $\partiali \Sigma$.
\begin{theorem} \label{th:nonexistence}
Consider two curves $\alpha^+$ and $\alpha^-$ in $\mathcal{C}^{2 , \alpha}(\sph^1)$
satisfying that $\alpha^+(\theta)- \alpha^-(\theta)>\pi$, for all $\theta \in \sph^1$, and $D^+$ and $D^-$ the minimal disks so that 
$\partial_{\infty} D^\pm=\alpha^\pm$, respectively. We label $\Omega(\alpha^+,\alpha^-)$ the domain in $\h^2 \times \R$ bounded by $D^+$ and $D^-$. Then
there is not a connected, properly embedded, minimal surface $\Sigma$ with boundary (possibly empty) satisfying:
\begin{enumerate}[i)]
\item $\Sigma$ intersects both connected components of $(\h^2 \times \R) -\Omega(\alpha^+,\alpha^-)$.
\item $(\partial_{\infty} \Sigma \cup \partiali \Sigma)\cap \overline{\Omega(\alpha^+,\alpha^-)}=\varnothing.$
\end{enumerate}
\end{theorem}
\begin{proof}
We proceed by contradiction, assume that there exists a connected surface $\Sigma$ verifying the hypotheses
of the theorem.

Given an element $b \in [0, \pi]$, we label $\alpha^\pm_b := \alpha^\pm|_{[e^{-{\rm i} b},e^{{\rm i} b}]}$
the corresponding arcs in $\alpha^+$ and $\alpha^-$, respectively,  between $e^{-{\rm i} b}$ and $e^{{\rm i} b}$.
Similarly, we define $s_x:= \{ e^{{\rm i} x} \} \times [\alpha^-(e^{{\rm i} x}), \alpha^+(e^{{\rm i} x})].$
Then we consider in $\partial_\infty (\h^2 \times \R)$ the Jordan curve  $\Gamma_b:=\alpha_b^+ \cup \alpha_b^- \cup s_b \cup s_{-b}$.
According to Coskunuzer's results \cite[Theorem 2.13]{cos-2}, we know that there exists a complete, area-minimizing  disk $\mathcal{T}_b$ which
spans $\Gamma_b$.

Using the surface $\Sigma$ as a barrier, we can prove that there exists a limit of the family
$\mathcal{T}_b$, as $b \to \pi$, which is different from $D^+ \cup D^-$. Let $\mathcal{T}$ denote the limit of this family that is a properly embedded,
area-minimizing disk whose ideal boundary consists of $\alpha^+\cup \alpha^- \cup s_\pi.$

As $\alpha^+(\theta)-\alpha^-(\theta)>\pi$, we can place two tall rectangles, $R$ and $R'$, 
placed at both sides of $s_\pi$ and satisfying $R\cap \mathcal{T}=R'\cap \mathcal{T}=\varnothing.$ 
If $R$ and $R'$ are
close enough, then we can move the rectangles toward $s_\pi$ until the boundaries of the three
surfaces; $\mathcal{T}$, $R$ and $R'$ touch along a common vertical segment.  This implies that the 
angle in which $\mathcal{T}$ meets $\partial_\infty (\h^2 \times \R)$ is bigger than the angle in which $R$
and $R'$ meets $\partial_\infty (\h^2 \times \R)$. In particular, $\mathcal{T}$ would not extend smoothly
to $s_\pi$. 

We consider $T_n:\h^2 \times \R \to \h^2 \times \R $ the horizontal dilation along the geodesic 
$\beta$ connecting $(-1,0)$ and $(1,0)$ by ratio $n \in \n$. Let us denote $\mathcal{E}_n:= T_n(\mathcal{T})$.
As $\{\mathcal{E}_n\}_{n \in \n}$ is a sequence of area-minimizing surfaces, then there is a sub-sequential
limit $\mathcal{E}_0$, which is not empty because the angle in which $\mathcal{T}$ meets $s_\pi$ is not zero.
Moreover, from the method in which we have obtained $\mathcal{E}_0$ we know that:
\begin{itemize}
\item $\mathcal{E}_0$ is area-minimizing.
\item $\mathcal{E}_0$ is foliated by equidistant curves in horizontal slices.
\item $\partial_\infty(\mathcal{E}_0)$ consists of two horizontal cycles at height
$\alpha^+(-1)$ and $\alpha^-(-1)$, respectively, and two vertical segments: 
$s_\pi$ and $\{1\} \times [\alpha^-(-1),\alpha^+(-1)].$
\end{itemize}
Hence, using the characterization of minimal surfaces invariant under hyperbolic
dilations (see \cite{h,st1}), we deduce that $\mathcal{E}_0$ consists 
of two tall rectangles of height $\alpha^+(-1)-\alpha^-(-1)$ symmetric with respect
the totally geodesic plane $\beta\times \R.$

But this contradicts Lemma \ref{lem:areamin} in Appendix \ref{sec:tall}, because $\mathcal{E}_0$ would be area-minimizing.
This contradiction proves that it is impossible to construct a connected, properly embedded,
minimal surfaces satisfying items {\em i)} and {\em ii)}. 
\end{proof}
As a consequence, we have the following corollary:
\begin{corollary} \label{co:non}
Let $\gamma^+$ and $\gamma^-$ be two curves in $\mathcal{C}^{2,\alpha}(\sph^1)$
satisfying $\gamma^+(\theta)-\gamma^-(\theta) >\pi$, $\forall \theta \in \sph^1$. Then 
there is not any properly embedded minimal annulus $A$ in $\h^2 \times \R$ with
$\partial A=\gamma^+ \cup \gamma^-$.
\end{corollary}
\begin{proof}
We apply the previous theorem to the curves  $\alpha^{+}:= \gamma^{+}- \varepsilon$ and
$\alpha^{-}:= \gamma^{-}+ \varepsilon$ for a
small enough $\varepsilon>0$.
\end{proof}

Finally, as a consequence of the proof of Theorem \ref{th:nonexistence} 
we have
\begin{corollary} 
Let $\gamma^+$ and $\gamma^-$ be two curves as in the previous corollary, and consider 
the vertical segment $s_\theta:= \{\theta\} \times [\gamma^-(\theta), \gamma^+(\theta)]$. Then 
there is not any properly embedded minimal surface $S$ in $\h^2 \times \R$ with
$\partial S=\gamma^+ \cup \gamma^- \cup s_\theta$.
\end{corollary}

\begin{remark}
Notice that in the process of proving Theorem \ref{th:nonexistence} we have obtained that
the limit of the family of disks $\{ T_b , \; \; b \in (0,\pi) \},$ as $b \to \pi$, is the union of the
two disks $D^+$, $D^-$ and the vertical segment $s_\pi$.
\end{remark}

\begin{remark}
We would like to point out that Theorem \ref{th:nonexistence} is more general than the corollaries that
we are going to use in this paper. We are not imposing any restriction neither about the genus nor
 the number of ends. For instance it shows that there are no Costa-Hoffman-Meeks type surfaces
 (like the ones constructed by Morabito \cite{morabito} and their possible perturbations) so that the vertical distance between 
 two consecutive ends is more than $\pi$.
\end{remark}

For the next result we impose a monotonicity condition on $\Gamma =
(\gamma^+, \gamma^-)$, which we normalize by centering around $\theta
= 0, \pi$.  Thus suppose that $\gamma^+$ is monotone decreasing
on $[0, \pi]$ and monotone increasing on $[\pi, 2\pi]$, while
$\gamma^-$ is monotone increasing on $[0,\pi]$ and monotone
decreasing on $[\pi, 2\pi]$. In other words, the two curves are tilted
away from each other. Notice that we allow non-strict
  monotonicity in each interval. 
 \begin{figure}[htpb]
 \begin{center}
 \includegraphics[height=.3\textheight]{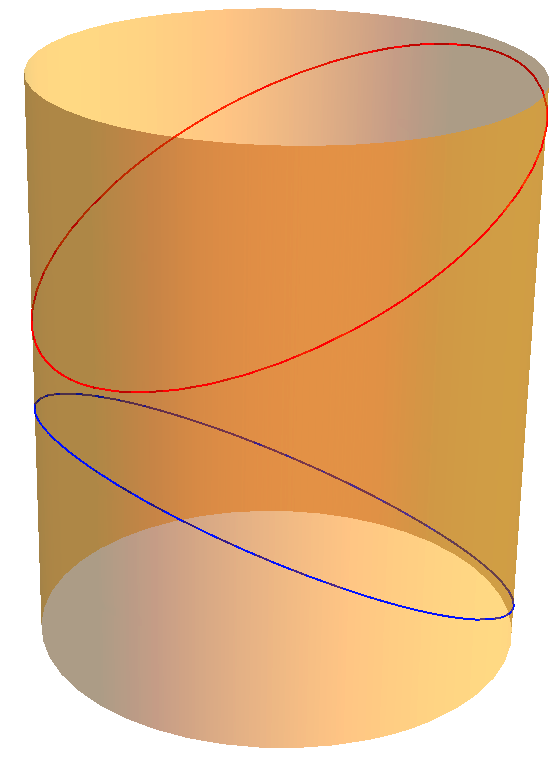}
 \end{center}
 \caption{ An example of curve $\Gamma$ in Proposition \ref{cl:alexandrov}.} 
 \label{fig:curves-1}
 \end{figure}
 
\begin{proposition}\label{cl:alexandrov}
Under the conditions above, there is no $A \in \fA$ with $\del_\infty A = \Gamma$ unless
$\gamma^\pm$ are constant (in which case $A$ is a catenoid). 
\end{proposition}
\begin{proof}
Let $\eta(s)$ denote the geodesic in $\h^2$ which connects $(1,0)$
(where $\theta = 0$) to $(-1,0)$, with $\eta(0) = (0,0)$
and $\eta \to (1,0)$ as $s \to -\infty$.  For each
$s$ denote by $\lambda_s$ the geodesic orthogonal to $\eta$ and
meeting it at $\eta(s)$. The vertical plane $P_s = \lambda_s \times
\RR$ separates $\h^2 \times \RR$ into two components, $\calU_s$ and
$\calV_s$, and we assume that $(1,0,0) \in \overline{\calU_s}$.

Write $A_s' = A \cap \overline{\calU_s}$, $A_s'' = A \cap
\overline{\calV_s}$, and denote by $A_s^*$ the reflection of $A_s'$
into $\overline{\calV_s}$.  By
  Proposition~\ref{prop:graph}, each end of $A$ is a vertical graph,
$\HH^2 \setminus D(o,R) \ni z \mapsto (z,
u^{\pm}(z))$. Then for $s \ll 0$, $A_s'$ consists of
  two connected components, each of them a vertical graph. Since
$u^\pm(1,\theta) = \gamma^\pm(\theta)$, the monotonicity hypotheses
imply that the boundary curves $(\gamma_s^*)^\pm$ of
$A_s^*$ satisfy
\[
(\gamma_s^*)^-(\theta)\leq \gamma^-(\theta) <
  \gamma^+(\theta) \leq (\gamma_s^*)^+(\theta)
\]
for all $\theta$ with $e^{\ir \theta} \times \R \subset
\overline{\calV_s}$. In addition, $\del A_s^* \cap P_s = A \cap P_s$.
By using the maximum principle with vertical
  translations of each connected component of $A_s^*$, we deduce that
  when $s$ is very negative, $A_s^*$ does not make contact with
$A_s''$ except at the boundary.  We then let $s$ increase until the
first point of interior contact, which shows that $A_s^* = A_s''$ for
some $s$. Therefore $\gamma^\pm$ are constant and  $A$
  is a rotationally invariant catenoid by~\cite[Theorem 2.1]{n-s-t}.
\end{proof}
\begin{remark} \label{re:alex}
The results of this section are still true if the annulus $A$ is
Alexandrov-embedded with embedded ends.
\end{remark}

\section{The manifold of minimal annuli} \label{sec:manifold}
In this section we prove the basic structural result about the space of minimal annuli. 
\begin{theorem}
The space $\fA'$ of properly Alexandrov-embedded minimal annuli with embedded ends and $\calC^{2,\alpha}$ boundary curves 
is a Banach submanifold in the space of complete properly immersed surfaces of class $\calC^{2,\alpha}$ in $\HH^2 \times \RR$.
\label{manifoldthm}
\end{theorem}
\begin{proof}
Fix any $A \in \fA'$.  We first note that $A$ is $\calC^{2,\alpha}$ up to the boundary. Indeed, following \cite{KM},
by the results of \S 3, there exists a neighborhood of infinity in $A$ which is the graph of a function $w: \HH^2 
\setminus D_{\HH^2} (o,R) \to \RR$ for some $R \gg 1$, and this function extends (at least) continuously to $\del \overline{\HH^2}$. 
Near any point of the asymptotic boundary we may as well use the 
upper half-space representation of $\HH^2$, and in such coordinates $(x,y)$ with $x \geq 0$, the minimal
surface equation is
\begin{equation}
w_{xx}( 1 + x^2 w_y^2) -2x^2 w_x w_y w_{xy} + w_{yy}(1 + x^2 w_x^2) - xw_x (w_x^2 + w_y^2) = 0.
\label{ndjo}
\end{equation}
The remarkable and fortuitous fact is that although one might expect factors which are powers of $x$ coming from the
corresponding factors in the hyperbolic metric in these coordinates, there is an overall factor of $x^2$ in this
equation, so it becomes nondegenerate. In any case, from this expression it is standard that if the boundary at infinity is
a $\calC^{2,\alpha}$ graph, $\gamma = \gamma(y)$, then $w(x,y)$ is $\calC^{2,\alpha}$ up to $x=0$. 

The standard method is to parametrize surfaces $\calC^1$ near to $A$ as normal graphs over $A$, i.e., as 
$\{ \exp_p( u(p) \nu(p) ):  p \in A \}$, where $\nu$ is the unit normal vector field to $A$. It is more useful here, however,
to alter this, replacing $\nu$ by a vector field $n$ for which $n = \nu$ in $A \cap (D_{\HH^2} (o,R) \times \RR)$ and 
$n = \del_t$ in $A \cap ((\HH^2 \setminus D_{\HH^2} (o,R')) \times \RR)$ for some $R' > R$.  Thus, if $u$ is any small $\calC^{2,\alpha}$ 
function, write
\[
A_u = \{ \exp_p ( u(p) n(p) ): p \in A\}.
\]
The surface $A_u$ is minimal if $\mathcal N(u) = 0$, where $\mathcal N$ is some degenerate second order quasilinear elliptic differential
operator. Note that $\calN$ is not the operator \eqref{ndjo}, but at least locally near $x=0$ of the form $x^2$ times that operator.
In the ball model, let $\calN = (1-r^2) N$, so $N$ is a nondegenerate quasilinear elliptic operator.
The linearization of $\calN$ at $u=0$ is the Jacobi operator $\calL$ relative to the vector field $n$.   For {\it normal}
graphs (i.e., using $\nu$ instead of $n$), this Jacobi operator is
\[
\widehat\calL \phi = (\Delta_A + |S|^2 + \mbox{Ric}(\nu,\nu)) \phi,
\]
where $S$ is the shape operator (or second fundamental form) of $A$. With this slight change of parametrization, it is 
shown in \cite{MP}[appendix] that 
\[
\calL = \wh{\calL} J, \quad J = n \cdot \nu.
\]
Since $\nu = (-x^2 w_x, -x^2 w_y, 1)/ \sqrt{ 1 + x^2 (w_x^2 + w_y^2)}$,  $J = 1/\sqrt{1 + x^2 |\nabla w|^2}$,
so $J$ is globally close to $1$, and moreover $J = 1$ and $J_x = 0$ at $x=0$. 

We make one further reduction, noting that in the ball model, $\calN$ is $(1-r^2)$ times a nondegenerate operator,
so its linearization can be divided by the same vanishing factor to obtain the linear modified Jacobi operator $L = (1-r^2)^{-2} 
\wh{\calL} J$, which is simply the linearization of $N$, and this operator is now a standard nondegenerate elliptic operator.  
The net effect is that we can use  standard elliptic theory (rather than the uniformly degenerate elliptic theory from 
\cite{Mazzeo-edge} needed to study $\calL$). 

For a given $A\in \fA'$, let us define
\[
\fJ(A) = \{\psi \in L^\infty(A): \widehat{\mathcal{L}}_A \psi = 0\}, \qquad \fJ^0(A)= \fJ(A) \cap L^2(A).
\]

The main results about $N$ and $L$ are as follows.
\begin{proposition} Let $A \in \fA'$ and suppose that $\del A = \Gamma$ consists of a pair $\gamma^\pm$ of 
$\calC^{2,\alpha}$ horizontal curves.  Then
\begin{enumerate}[i)]
\item The graph function $u$ (relative to the parametrization using the vector field $n$) lies in $\calC^{2,\alpha}(\overline{A})$;
\item If $\phi$ is a solution to $L\phi = 0$ with boundary values $\phi_0^\pm \in \calC^{2,\alpha}(\sph^1)$, 
then $\phi \in \calC^{2,\alpha}(\overline{A})$;
\item The operator $L: \calC^{2,\alpha}_D(A) \longrightarrow \calC^{0,\alpha}(A)$ is Fredholm of index $0$, where
$\calC^{2,\alpha}_D$ is the space of $\calC^{2,\alpha}$ functions which vanish at $\del A$. Its kernel $\{\phi:  L\phi = 0\}$
is identified with $\fJ^0(A)$ via $L \phi = 0$ if and only if $\phi = J \varphi$, $\wh \calL \varphi = 0$. (For simplicity,
we often refer to this nullspace as $\fJ^0(A)$, recalling this identification when appropriate.) This same finite dimensional 
space is a complement for its range; 
\item If $\gamma^\pm \in \calC^\infty$ and if $\phi \in \fJ^0(A)$, i.e., $\phi_0^\pm = 0$, then $\phi \in \calC^{\infty}(\overline{A})$. 
\end{enumerate}
\label{Lmapprop}
\end{proposition}
The only point that requires comment is the third; the fact that it is Fredholm is of course standard, and its
index vanishes since it is deformable amongst elliptic Fredholm operators to a self-adjoint operator.

We say that $A$ is nondegenerate if $\fJ^0(A) = \{0\}$.  Define a continuous extension operator
$e: \calC^{2,\alpha}(\sph^1)^2 \to \calC^{2,\alpha}(A)$, and now consider the map
\begin{equation} \label{eq:ya}
\calC^{2,\alpha}(\sph^1)^2 \times \calC^{2,\alpha}_D(A) \ni ((\phi_0^+, \phi_0^-) , w) \longrightarrow
N( e( \phi_0^+, \phi_0^-) + w) \in \calC^{0,\alpha}(A). 
\end{equation}

\begin{proposition}
If $A \in \fA'$ is nondegenerate, then there exists a neighborhood $\calU$ of $0$ in $\calC^{2,\alpha}(\sph^1)^2$ and
a smooth map $G: \calU \to \calC^{2,\alpha}_D(A)$ such that $N( e(\phi_0^+, \phi_0^-) + G (\phi_0^+, \phi_0^-)) = 0$,
and all solutions $u$ to $N(u) = 0$ sufficiently close to $0$ are of this form. 
\end{proposition}
We reduce this to the implicit function theorem as follows. 
By hypothesis, the linearization of the map given by \eqref{eq:ya}, $L$, is surjective, and there is a bijective correspondence between the nullspace of
$L$ and pairs $\phi_0^\pm$. This last statement is a restatement of the fact that the linear Poisson problem is
well-posed: there exists a unique homogeneous solution of $Lw = 0$ with $w = \phi_0^\pm$ on $\gamma^\pm$. 

\medskip

To prove that $\fA'$ is a Banach manifold even around degenerate annuli, we must characterize those pairs 
$(\phi^+_0, \phi^-_0)$ which occur as leading coefficients of elements of $\fJ(A)$, which denotes the space of Jacobi fields of $A$.  In the following, for
$\phi \in \fJ(A)$, we write $\phi_1$ for its normal derivative at the boundary (computed with respect to
the fixed chart $z$). 
\begin{proposition} \label{bvalues}
Let $(\phi_0^+, \phi_0^-)$ be a pair of functions in $ (\mathcal{C}^{2,\alpha}(\sph^1) )^2$. Then, there exists a Jacobi field 
$\phi \in \fJ(A)$ satisfying $\phi |_{\partial A^\pm }=\phi_0^\pm$ if and only if 
$$ 
\int_{\sph^1} (\phi^+_0 \psi^+_1 + \phi^-_0 \psi^-_1) = 0 
$$
for every $\psi \in \fJ^0(A).$
\end{proposition}
\begin{proof}  First note that if $\wh \phi, \wh \psi \in \fJ^0(A)$, then
\begin{equation*}
\begin{aligned}
0 & =\int_A (\wh \calL\wh \phi) \wh \psi - \wh \phi (\wh \calL\wh \psi) = \int_{\sph^1} (\wh \phi^+_0 \wh \psi^+_1 - 
\wh \phi^+_1 \wh \psi^+_0)  +  (\wh \phi^-_0 \wh \psi^-_1 - \wh \phi^-_1 \wh \psi^-_0) \\ 
& = \int_{\sph^1} (\wh \phi^+_0 \wh \psi^+_1 + \wh \phi^-_0 \wh \psi^-_1).
\end{aligned}
\end{equation*}
(The integrals at the  two boundary components appear with the same sign because we are using the outward pointing
normal derivative at each of these.)   We may transfer this to an identify involving functions $\phi, \psi$ in the nullspace
of $L$ by setting $\phi  = J \wh \phi$, $\psi = J \wh \psi$ and also multiplying the area form of $A$ by $J^{-1}$; 
note also that $\phi_0^\pm = \wh \phi^\pm_0$,  $\psi_0^\pm = \wh \psi_0^\pm$ because $J = 1$ and its normal derivative 
vanishes at $\del A$.  Thus we also have
\[
0 = \int_{\sph^1} (\phi^+_0 \psi^+_1 + \phi^-_0 \psi^-_1).
\]
In the following we perform the same integration by parts a few more times; each time we invoke the self-adjointness
for $\wh \calL$ with respect to the geometric area form and then conjugate to obtain the analogous formula for 
the nondegenerate operator $L$ with respect to a new area form.  However, for simplicity, we do not spell this out carefully again.

This necessary condition is also sufficient. Indeed, fix any $\phi_0^\pm$ satisfying this orthogonality condition, 
and set $u = e(\phi_0^+, \phi_0^-)$.  Then $L u = f \in \calC^{0,\alpha}$. By part iii) of Proposition~\ref{Lmapprop}, 
there exists $\psi \in \fJ^0(A)$ and $v \in \calC^{2,\alpha}_D(A)$ such that $Lv = f + \psi$.  Thus writing $\phi = v-u$, 
then $ L \phi = \psi$. We now show that this is impossible unless $\psi = 0$.  Indeed, since $v$ vanishes at $\del A$, 
$u^\pm = \phi_0^\pm$. Now we compute that
\begin{equation} \label{eq:comp}
\int_A |\psi|^2 = \int_A (L\phi)\psi - \phi (L\psi) =  \int_{\del_+ A} \phi_0^+ \psi_1^+ + \int_{\del_- A} \phi_0^- \psi_1^- = 0,
\end{equation}
hence $\psi = 0$.  This completes the proof. \end{proof}

This result proves that the set of pairs $(\phi^+_0, \phi^-_0)$ which can occur as leading coefficients of Jacobi fields $\phi \in \fJ(A)$ 
has finite codimension in $\calC^{2,\alpha}(\sph^1)^2$, and that a good choice of complementary subspace for it is the space
\[
W = \{ e(\phi_1^+, \phi_1^-):  \phi \in \fJ^0(A)\}
\]
of normal derivatives of all elements of $\fJ^0(A)$. 

\begin{proposition} \label{surj}  The map 
\begin{equation}
L: W \oplus \calC^{2,\alpha}_D(A) \longrightarrow \calC^{0,\alpha}(A)
\label{bestsurj}
\end{equation}
is surjective, with nullspace $\fJ^0(A)$.   
\end{proposition}
\begin{proof}
We have already noted that the range of $L$ on $\calC^{2,\alpha}_D$ is a finite codimensional space in 
$\calC^{0,\alpha}$ complementary to $\fJ^0(A)$.  Suppose then that $\gamma \in \fJ^0(A)$ and
\[
\int_A \gamma  (L( e(\phi_1^+,\phi_1^-) + u ) = 0\ \ \ \mbox{for every}\ \phi \in \fJ^0(A)\  \mbox{and}\ u \in \calC^{2,\alpha}_D(A).
\]
Taking $\phi = 0$ and integrating by parts simply confirms that $L\gamma = 0$.  Next, using that $\gamma$ vanishes at the boundary, 
let $u = 0$ and integrate by parts again to obtain
\[
0 = \int_A \gamma (L e(\phi_1^+,\phi_1^-)) - (L\gamma) e(\phi_1^+,\phi_1^-) = \int_{\sph^1} (\gamma_1^+ \phi_1^+ + \gamma_1^- \phi_1^-).
\]
Letting $\phi = \gamma$ shows that $\gamma_1^\pm = 0$, and hence that $\gamma = 0$. 

To finish the proof, note that if $\psi = e(\phi_1^+, \phi_1^-) + u \in \fJ(A)$, then $\psi_0^\pm = \phi_1^\pm$ for
some $\phi \in \fJ^0(A)$, which we showed above is impossible unless $\phi = 0$.  This proves that 
the null space of \eqref{bestsurj} equals $\fJ^0(A)$.  
\end{proof}

We may now complete the proof of Theorem~\ref{manifoldthm}.  The case when $A$ is nondegenerate has already been handled, 
so suppose that $\fJ^0(A) \neq \{0\}$.  Choose subspaces $\fJ^0(A)^\perp \subset \fJ(A)$ and $\calX_0 \subset \calC^{2,\alpha}_D(A)$,
each complementary to $\fJ^0(A)$ in the respective larger ambient spaces. Immediately from Proposition \ref{surj}, 
\[
L:  \fJ^0(A)^\perp  \oplus W \oplus \calX_0 \longrightarrow \calC^{0,\alpha}(A)
\]
is surjective, with nullspace $\fJ^0(A)^\perp$.  In addition,
\[
N:  \fJ^0(A)^\perp \oplus W \oplus \calX_0 \longrightarrow \calC^{0,\alpha}(A)
\]
is well-defined and smooth.  The implicit function theorem implies, as before, the existence of a map 
\[
G: \fJ^0(A)^\perp \longrightarrow W \oplus \calX_0
\]
and a neighbourhood $\calW$ of $0$ in $\fJ^0(A)^\perp$ such that
\[
\calW \ni \phi \mapsto N( \phi + G(\phi) ) \equiv 0,
\]
and all solutions of $N$ near to $0$ are of this form.  

Once again, this is a chart for $\fA'$ near $A$, which proves that $\fA'$ is a Banach submanifold even
around degenerate points. 
\end{proof}

\begin{remark} \label{re:leading}  Since it will be important later, we recall that we have already
given explicit expressions for the decaying Jacobi fields $\fJ^0(A)$ associated to the rotationally 
invariant catenoid $A_0 = C_h$, see Proposition 2.2. We calculate from these that the space of 
normal derivatives of elements of $\fJ^0(A_0)$ is spanned by 
$( \sin \theta, \sin \theta)$ and $(\cos \theta, \cos \theta)$.  
\end{remark}

\section{The extended boundary parametrization}

Let $A$ be a proper, Alexandrov-embedded, minimal annulus with embedded ends such that $\Pi(A)=(\gamma^+,\gamma^-)$ consists of two 
$\calC^{2, \alpha}$ graphs over $\sph^1$. The bottom  boundary curve $\gamma-$ bounds a unique minimal disk $D^-$; this is 
the vertical graph of a function $v^-$. Let $u^-$ denote the function parametrizing the bottom end of $A$. We shall 
consider the space of  minimal annuli $\fA^* \subset \fA'$ which satisfy 
\begin{equation}\label{eq:open}
u_r^-(1,\theta)-v_r^-(1,\theta)<0, \quad \forall \theta \in \sph^1. 
\end{equation} 
Clearly $\fA^*$ is an open subset of $\fA'$, and hence its tangent space $T_A \fA^*$ at any point equals $\fJ(A)$. In addition,
it is trivial that $\fA \subset \fA^*$. Consider the map
\[
\Pi: \fA^* \to \calC^{2,\alpha}(\sph^1)^2,
\]
which takes any $A \in \fA^*$ to its pair of boundary curves $(\gamma^+ ,\gamma^-) \in \calC^{2,\alpha}(\sph^1)^2$.

The perhaps naive hope is that this map can be used to parametrize $\fA^*$ by some subset of $\calC^{2,\alpha}(\sph^1)^2$.  To 
understand whether this is feasible, the first step is to compute its index. 
\begin{theorem}
The map $\Pi: \fA^* \to \calC^{2,\alpha}(\sph^1)\times  \calC^{2,\alpha}(\sph^1)$ is Fredholm of index zero.
\end{theorem}
\begin{proof}
The assertion is that the linear map $D\Pi|_A$ is Fredholm of index $0$ for every $A$.  However, $D\Pi|_A (\phi)  = \phi_0$, the leading 
coefficient of the Jacobi field $\phi$ at $\del A$, so we must show that $\fJ(A) \ni \phi \mapsto \phi_0 \in \calC^{2,\alpha}(\sph^1)^2$ is 
Fredholm of index $0$. This follows immediately from Proposition \ref{bestsurj}. 
\end{proof}

We have already seen that $D\Pi$ is not invertible at the catenoid $C_h$. It has a two-dimensional nullspace there, and the implicit 
function theorem shows that the range of $\Pi$ is contained locally near $\Pi(C_h)$ around a codimension $2$ submanifold.  
We prove later that this image has nontrivial interior, but by Proposition~\ref{cl:alexandrov}, $\Pi(C_h)$ is not an interior point
of this image.  In fact, we do not have a precise characterization of $\Pi( \fA^*)$. 

It is useful to define a slightly different boundary correspondence via the extended boundary map
\begin{equation}
\begin{gathered}
\widetilde{\Pi}: \fA^* \times \RR \times \CC \longrightarrow \calC^{2,\alpha}(\sph^1)^2 \times \RR \times \CC, \\
\widetilde{\Pi}( A, a, \eta) = ( \Pi_-(A),  \Pi_+(A) + a + \Re \left( \eta e^{i\theta} \right), G(A) ).
\end{gathered}
\end{equation}
Here $\Pi_\pm(A) = \gamma\pm$ and the components $(G_0, G_1, G_2)$ of $G$ are defined as follows.  
The bottom boundary curve $\gamma-$ bounds a unique minimal disk $D^-$; this is the vertical graph of a
function $v^-$. Letting $u^-$ denote the function parametrizing the bottom end of $A$, we write
\begin{equation*}
\begin{aligned}
f_0(A) & = \mbox{Flux}(A, \del_- A, E_3) = \int_{\sph^1} u_r^-(1,\theta)\, d\theta, \\
f_1(A) + \imag f_2(A) & = \int_{\sph^1} e^{i\theta} (u_r^-(1,\theta) - v_r^-(1,\theta))\, d \theta,
\end{aligned}
\end{equation*}
and in terms of these, define
\[
G_0(A) = f_0(A) - f_0(A)^{-1}, \ \  G_1(A) + \imag \; G_2(A) = (f_1(A) + \imag f_2(A))/f_0(A). 
\]
\begin{definition}
We refer to $\mathbf{C}(A) := G_1(A) + \imag \; G_2(A)$ as the {\bf center} of $A \in \fA^*$. 
\end{definition}
Note that 
\begin{equation}{\bf C}(\widetilde R_\zeta (A))=R_\zeta ({\bf C}(A)), \label{eq:rota} \end{equation}
where $R_\zeta: \RR \times \CC \to \RR \times \CC$ is the rotation $(t, z) \mapsto (t, {\rm e}^{{\rm i} \zeta} z)$.

Since the flux of the disk $D^-$ is zero, we can also write 
\begin{equation}\label{eq:f0}
f_0(A) = \int_{\sph^1} \left(u_r^-(1,\theta)-v_r^-(1, \theta)\right)\, d\theta, 
\end{equation}
and  certainly $u_r^-(1,\theta)-v_r^-(1, \theta) <0$ for all $\theta \in \sph^1$. We thus have that 
$|{\bf C}(A)| <1$ for all $A \in \fA^*$, i.e., $\mathbf{C}(A) \in \h^2$. 
Now define 
\begin{equation}
\widetilde{\fA^*}: = \widetilde{\Pi}^{-1} \left( \calC^{2,\alpha}(\sph^1)^2 \times \RR \times \mathbb D \right), 
\label{fatstar}
\end{equation}
where $\D$ denotes the unit disk in $\CC$.
\begin{remark} \label{re:cta} 
Evaluating $f_1 + i f_2$ on the $3$-dimensional family of catenoids $\calM$ (see Remark \ref{re:catenoids}), then 
$v_r^- \equiv 0$, and hence  
$$
f_1(C_{h,z_0}) + if_2(C_{h,z_0})  = \int_{\sph^1} e^{i\theta} u_r^-(1,\theta) \, d \theta.
$$
It is then straightforward to check that $z_0= G_1(C_{h,z_0}) + \imag \; G_2(C_{h,z_0}).$
\end{remark}

This remark shows that the center of the neck of the catenoid $C_{h,z_0}$ in the obvious geometric sense 
equals $\mathbf{C}(C_{h,z_0})$.  We show in Section \ref{sec:compact} that $\mathbf{C}(A)$ behaves 
like this center more generally in the following sense. If $A_n$ is a sequence of annuli for which 
the sequence of vertical fluxes is bounded and $\mathbf{C}(A_n)$ diverges in $\h^2$, then $A_n$ 
converges to two disjoint minimal disks, and  the necks of these annuli disappear at infinity.

The motivation for introducing this enhanced boundary map $\widetilde \Pi$ is that the catenoids are no longer degenerate points.
\begin{theorem}
The extended boundary correspondence $\widetilde{\Pi}$ is a proper Fredholm map of index $0.$
It is locally invertible near any one of the catenoids $C_h$. 
\label{pfi0}
\end{theorem}
Theorem \ref{pfi0} will be proved in a series of steps. In the remainder of this section we verify that $\widetilde{\Pi}$ is Fredholm 
of index $0$ and check that its differential at $C_h$ is invertible.  The properness assertion is more difficult and its proof occupies 
the next section. 

The main result of this paper is an essentially global existence theorem which is proved using degree theory. This relies on the fact 
that, under mild hypotheses, a proper Fredholm map between Banach manifolds has a $\mathbb Z$-valued degree. The importance of Theorem~\ref{pfi0} is 
that it implies that this degree equals $1$. This will be discussed carefully below. 
\begin{proposition}
The map $\widetilde{\Pi}$ is Fredholm of index $0$.
\end{proposition}
\begin{proof}
This map is Fredholm because its domain and range are finite dimensional extensions of those for $\Pi$; since 
these extensions have the same dimension, the index remains $0$. 
\end{proof}

\begin{proposition}
Let $C$ be any catenoid.  Then $D\widetilde{\Pi}|_{C}$ is invertible.
\end{proposition}
\begin{proof}
We may as well suppose that $C = C_h$ is centered. Since the index of this differential vanishes, it suffices to show that its nullspace is 
trivial.  Suppose then that $D\widetilde{\Pi}|_{C} (\phi, \alpha, \mu ) = (0,0)$ for some $(\phi, \alpha, \mu) \in \fJ(A) \times \RR 
\times \CC$. This corresponds to the set of equations 
\[
\phi_0^- = 0, \qquad \phi_0^+ + \alpha + \Re(\mu e^{i\theta}) = 0, \quad \mbox{and}\qquad  DG_C(\phi) = 0.
\]

The first equation states that the Jacobi field $\phi$ vanishes at $\del_-A$, while the second condition shows that its restriction
to $\del_+ A$ lies in the span of $\{1, \cos \theta, \sin \theta\}$.  On the other hand, by our knowledge of the elements of $\fJ^0(C_h)$, 
$(\phi_0^+, \phi_0^-) = (a_0^+, a_0^-) +  a_1 (\cos \theta, \cos \theta) + a_2 (\sin \theta, \sin \theta)$. Comparing these two 
expressions gives $a_1 = a_2 = 0$, and hence $\phi$ is the Jacobi field corresponding to the variation $\epsilon \mapsto C_{h+\epsilon}$
(where we assume that the bottom boundary of $C_{h+\epsilon}$ remains fixed while the height of the top circle varies).  Combining this with
the two-dimensional nullspace $\fJ^0(C)$ of the map $\phi \mapsto \phi_0$,  we see that the nullspace of the differential of the first two
components of $D \widetilde{\Pi}|_C$ is three-dimensional.

Now examine the equation $DG|_C (\phi) = 0$. First consider the Jacobi field $\phi$ with constant boundary values 
$\phi_0^- = 0$, $\phi_0^+ = a_0^+$. We compute that
\[
DG_0|_C (\phi) = (1 + f_0(C)^{-2} )  Df_0|_C (\phi) = c \int_{\sph^1}  \phi^+_1(\theta) \, d\theta,
\]
with $c \neq 0$. We claim that this expression is nonzero. Indeed, $\phi_1^+(\theta) = \del_r \phi(1,\theta)$ at $\del_+ A$,
and by the maximum principle, this normal derivative is nonnegative. Thus this whole expression vanishes if and only if 
$\del_r \phi \equiv 0$ at this top boundary. Hence $DG_0|_C(\phi) = 0$ implies $a_0 = 0$. 

Next compute that
\[
D G_j|_C =  \frac{ f_0(C) Df_j|_C - f_j(C) D f_0|_C}{ f_0(C)^2}, \quad j = 1, 2. 
\]
Since $f_j(C) = 0$, it suffices to check that the two-by-two matrix which is the restriction of the Jacobian of $(f_1, f_2)$ to
$\fJ^0(C)$ is nonzero. However, this is clear from Remark \ref{re:cta} and the formul\ae
\[
D(f_1 + \imag \; f_2)|_C ( \phi) = \int_{\sph^1}  e^{i\theta} \, \phi_1^-(\theta)\, d\theta, 
\]
since $\phi_1^- = c_1 \cos \theta + c_2 \sin \theta$ if $\phi \in \fJ^0(C)$. This proves that $D \widetilde{\Pi}|_C$ is invertible. 
\end{proof}

\begin{proposition} \label{prop:uno}
For any $h \in (0,\pi/2)$, $\widetilde{\Pi}^{-1} \left( \widetilde \Pi \left( C_h,0,0,0\right)\right)=\{ (C_h,0,0,0)\}$.
\end{proposition}
\begin{proof}  Observe that $ \widetilde \Pi \left( C_h,0,0,0\right)= ( -h,h, G_0(C_h),0,0).$

If $(A,x_0, x_1+{\rm i} \,x_2) \in \widetilde{\Pi}^{-1} \left( -h,h, G_0(C_h),0,0\right)$, then by definition, 
$\Pi_- (A)= \sph^1 \times \{-h\}$ and $\Pi_+( A)= \{ (\theta, h -x_0 -x_1 \cos \theta+x_2 \sin \theta) \; : \; \theta \in \sph^1\}. $ Using Proposition \ref{cl:alexandrov} and facts that $h \mapsto G_0(C_h)$ is bijective and $\left(G_1(A),G_2(A) \right)=(0,0)$,
we deduce that $x_0=x_1=x_2=0$ and $A = C_h$.
\end{proof}

\section{Compactness} \label{sec:compact}
Our goal in this section is to prove that the map $\widetilde{\Pi}: \widetilde{\fA^*} \to \calC^{2, \alpha}(\sph^1)^2 \times \RR \times \D$ is proper.  In other words, we 
show that if $\widetilde{\Pi}(A_n, x^{(n)})  = (\gamma_n^-,\gamma_n^+, z^{(n)})$ converges in 
$\calC^{2, \alpha}(\sph^1)^2 \times \RR \times \D$, then some subsequence
of $(A_n, x^{(n)})$ converges in $\widetilde{\fA^*}$.   

First, we study  the modes of divergence of sequences of minimal annuli in $\fA$.
This will allow us
to obtain the properness of $\widetilde \Pi$, and the properness of $\Pi$ when we restrict it to certain submanifolds and open regions of $\fA$.

\subsection{Diverging sequences in ${\fA}$}
Suppose that $\{A_n\}$ is any sequence in $\fA$, whose sequence of boundary curves $\{\Pi(A_n)\}$ converges to $\Gamma=(\gamma^+,\gamma^-) \in \fC.$  By Proposition~\ref{prop:graph}, for each $n$ there exists a solid cylinder $D_{\HH^2}(q_n,R_n) \times \R$ 
such that $A_n\setminus (D_{\HH^2}(q_n,R_n)\times \R) = E_n^\pm$ is the union of two vertical graphs $\h^2 \setminus D_{\HH^2}(q_n,R_n) \ni z \to (z, u_n^\pm(z))$, each
one embedded. Up to a subsequence (which we assume without further comment), there are three possible behaviors: 
\begin{description} \label{cases}
\item[Case I]  Both the centers $q_n$ and radii $R_n$ can be chosen independent of $n$, hence $E_n^\pm$ are graphs over a fixed annulus 
$\h^2 \setminus D_{\HH^2}(q,R)$; 
\item[Case II]  The radii $R_n$ are independent of $n$, but the centers $q_n$ diverge;
\item[Case III]  The sequence of radii $R_n$ diverges. Notice that we can re-arrange the 
sequence of solid cylinders so that $q_n=o$, for all $n$.
\end{description}

Our analysis relies on the following two results: 
\begin{theorem}[White \cite{white12}]   \label{th:Z}
Let $(\Omega,g)$ be a Riemannian $3$-manifold and $M_{n} \subset \Omega$ a sequence of properly embedded minimal 
surfaces with boundary such that 
\[
\limsup_{n \to \infty}\,  {\operatorname{length}}\{\partial M_{n}\cap K\}< \infty,
\]
for any relatively compact subset $K$ of $\Omega$. Define the area blowup set 
\[
Z := \{p \in \Omega: \lim \sup \, \mbox{\rm area}(M_n \cap B(p, r )) = \infty \quad \mbox{\rm for every $ r > 0$} \},
\]
and suppose that $Z$ lies in a closed region $N \subset \Omega$ with smooth connected mean-convex boundary $\partial N$,
i.e., $g\big(H_{\partial N},\xi\big)\ge 0$ on $\del N$, where $H_{\partial N}$ is the mean curvature vector 
and $\xi$ is the inward-pointing unit normal to $\partial N$. 
Then $Z$ is a closed set and if $Z \cap \del N \neq \varnothing$, then $Z \supset \del N$. 
\end{theorem}
\begin{theorem}[White \cite{white13}] \label{th:ZZ} Let $\Omega$ be an open subset in a Riemannian $3$-manifold
and $g_n$ a sequence of smooth Riemannian metrics on $\Omega$ which converge smoothly to a metric $g$.
Suppose that $M_n \subset \Omega$ is a sequence of properly embedded surfaces such that $M_n$ is minimal with respect
to $g_n$, and that the area and the genus of $M_n$ are bounded independently of $n$. Then after passing to a subsequence,
$M_n$ converges to a smooth, properly embedded $g$-minimal surface $M'$. For each connected component $\Sigma$ of $M'$, either
\begin{enumerate}
\item the convergence to $\Sigma$ is smooth with multiplicity one,  or
\item the convergence is smooth (with some multiplicity greater than $1$) away from a discrete set $S.$
\end{enumerate}
In the second case, if $\Sigma$ is two-sided, then it must be stable.
\end{theorem}
We may now proceed. 
\begin{theorem} \label{th:annulus} In Case I, the sequence $A_n$ converge smoothly to $A \in {\fA}$, and $\Pi(A)=\Gamma \in \fC$. 
\end{theorem}
\begin{proof}
Using classical elliptic estimates and the Arzel\`{a}-Ascoli theorem, some subsequence of the $u_n^\pm$ converge smoothly to functions 
$u^\pm$ on $\h^2 \setminus D_{\HH^2}(q,R)$, the graphs of which are minimal.  This means that the truncated minimal surfaces
$M_n:=A_n\cap (D_{\HH^2}(q,R) \times \R)$ are annuli with smoothly converging  boundaries $\del M_n = \widehat{\gamma}_n^\pm$,
so in particular, the lengths of $\del M_n$ are uniformly bounded. 

The blowup set $Z$ of the sequence $A_n$ lies in the interior of the fixed solid cylinder. For $h$ close to $\pi$, the catenoids 
$C_h$ do not intersect 
$D_{\HH^2}(q,R)  \times \R$, hence if the blowup set $Z$ for this sequence were nonempty, then by decreasing $h$ there would exist a 
point of first contact of some $C_h$ with $Z$, which contradicts Theorem~\ref{th:Z}. Thus $Z = \varnothing$. Theorem~\ref{th:ZZ} then 
implies that the $A_n$ converge smoothly everywhere. 
\end{proof}
\begin{remark} \label{re:contact-I}
By an easy modification of this proof, Theorem \ref{th:annulus} remains valid even if the limit curves $(\gamma^-, \gamma^+)$ satisfy
$\gamma^-(\theta) \leq \gamma^+(\theta)$ for all $\theta$ but $\gamma^- \not\equiv \gamma^+$. 
\end{remark}

Case II is more complicated.  By~\cite[Theorem 4]{n-r}, the limiting boundary curves $\gamma^\pm$ each span unique properly embedded minimal disks 
$Y^\pm$ which are vertical graphs of functions $v^\pm$ over all of $\h^2$. Notice that, by compactness, up to a subsequence we can assume that $\{q_n\}$ converges to a point $q_\infty \in \partial_\infty \h^2$. Without lost of generality (up to applying suitable rotations around $o$) we can assume that 
$q_\infty=(1,0)$ and that $q_n$ lies on the real line, for all $n \in \n$.   In the following, we choose a sequence of horizontal dilations $T_n$ 
such that $T_n(q_n)$ is the origin $o \in \h^2$.  We also denote by $T_n$ the usual extension of these dilations to isometries of $\h^2 \times \RR$.

\begin{theorem} \label{th:catenoid} In Case II, $A_n$ converges smoothly on compact sets of $\h^2 \times \RR$ to the union of the minimal 
disks $Y^\pm$. The closures $\overline{A_n}$ converge as subsets of $\overline{\h^2}\times \R$ to the union of the vertical line 
segment $\{q_\infty\}\times[t^-,t^+]$ and the closures of $Y^\pm$. Here $(q_\infty, t^\pm) = ( (1,0), t^\pm)$ are points in $\gamma^\pm$,
hence $t^\pm = \gamma^\pm(0)$, and necessarily in this case, $(\gamma^\pm)'(0) = 0$. 

Moreover, choosing $T_n$ as above, the sequence $\overline{T_n(A_n)}$ converges to a vertical catenoid $C_h$ smoothly in the interior and in
 $\calC^{2,\alpha}$ on compact sets of $(\overline{\h^2} \setminus \{(-1,0)\}) \times \RR$.  Since $h < \pi$, this convergence to
a catenoid implies that $t^+-t^-<\pi$. 
\end{theorem}
\begin{remark}
In particular, if there is no pair of points $(q_\infty, t^\pm) \in Y^\pm$, one above the other, where the tangents are horizontal, then Case II cannot occur. 
\end{remark}
\begin{proof}
First note that, similarly to Case I, the ends $E_n^\pm$ converge as minimal graphs, smoothly in the interior and in $\calC^{2,\alpha}$ on compact sets of 
$(\overline{\h^2}\setminus\{q_\infty\}) \times \RR$, to the minimal disks $Y^\pm$.    On the other hand, the dilated boundary curves
$\sigma_n^\pm = T_n(\gamma_n^\pm)$ converge in $\calC^{2,\alpha}$ to the constant maps $\sigma^\pm(\theta) = t^\pm$ 
away from $\theta = \pi$, i.e., away from the point $(-1,0)$. Hence $\overline{T_n(A_n)}$ converges in $\calC^{2,\alpha}$ away
from $\{(-1,0)\} \times \RR$. 

We deduce from this that $T_n(A_n)$ converges to an embedded minimal annulus $\Sigma$. At first we only know 
that $\del \Sigma$ consists of the two circles $\sigma^+ \sqcup \sigma^-$ and two (possibly overlapping) line segments 
$\{(-1,0)\} \times J^\pm$.  We claim that in fact $\del \Sigma$ consists only of the two circles, so that by \cite[Theorem 2.1]{n-s-t}, 
$\Sigma$ equals a rotationally invariant catenoid $C_h$.  To prove this, write $J^\pm = [a^\pm, b^\pm]$, and suppose that $b^+ > t^+$. 
Given $s \in (-1,1)$ let $\gamma_s$ denote the geodesic orthogonal to real axis 
and passing through $(s,0)$. Let $c_s$ be the arc in $\sph^1$ determined by the ends
points of $\gamma_s$ and containing the point $(1,0)$. Let $\Omega_s$ be the region
in $\h^2$ determined by the geodesic $\gamma_s$ and the arc $c_s$. We know
that there exists a minimal graph over $\Omega_s$ (we called it $R(c_s,t^+, \infty)$ in page \pageref{tallrectangles}) whose Dirichlet boundary
values are the constant $t^+$ along $c_s$ and $+\infty$ along $\gamma_s$.
It is important
to notice that this family $\{R(c_s,t^+, \infty) \; : \; s \in (-1,1)\}$ foliates the region 
$\h^2 \times (t^+, +\infty)$. Since $R(c_s,t^+, \infty)$ 
is disjoint from $\Sigma$, for $s$ sufficiently close
to $1$, then we have that $\Sigma \cap R(c_s,t^+, \infty)=\varnothing$, for all $s \in (-1,1) $. 
In particular 
$b^+=t^+$. Similarly, we can prove that $a^-=t^-$.

At this stage, $\Sigma$ is a minimal surface and its boundary at infinity consists of  
$\sigma^+ \sqcup \sigma^- \sqcup \beta$,   where $\beta  \subseteq \{(-1,0) \} \times [t^-,t^+]$. 
We want to prove that $\beta=\varnothing$. 
Now consider the  geodesic $\gamma_s$ described above, $-1<s<1$. 
We denote by $P_s$ the vertical plane $\gamma_s \times \R$. Let $\calU_s$ and $\calV_s$ be the two connected components 
of $(\h^2\times \R) \setminus P_s$, with $(1,0,0) \in \overline{\calU_s}.$ We write $\Sigma'_s:= \Sigma \cap \overline{\calU_s}$,  
$\Sigma''_s:= \Sigma \cap \overline{\calV_s}$ and $\Sigma^*_s$ the reflection of $\Sigma'_s$ with respect to $P_s$. Reasoning as 
in Proposition \ref{cl:alexandrov}, we deduce, when $s$ is 
very close to 1,  $\Sigma_s^*$ does not intersect $\Sigma_s'$ except at the boundary.  We claim that this is always the case.

If  $\Sigma_s^*$ does not intersect $\Sigma_s'$ except at the boundary, for all $s$, then $\Sigma$ 
would be simply connected, which is absurd. Then, there is a  first point of interior contact, so that $\Sigma_s^* = \Sigma_s''$ for some $s$. By the maximum principle, $\Sigma$ is symmetric with respect to
$P_s$. In particular $\beta=\varnothing.$

These arguments show that the boundary at infinity of the limit of $T_n(A_n)$ is the pair of parallel circles $\sigma^+ \sqcup 
\sigma^-$, and hence $T_n(A_n)$ converges to a catenoid $C_h$ with axis $\{o\} \times \RR$. 

The limit of the translated catenoids $T_n^{-1}(C_h)$ contains the entire line segment $\{(1,0)\} \times [t^-, t^+]$, hence the same
must be true for the limit of the $A_n$. 

It remains to show that the tangent lines to (the undilated curves) $\gamma^\pm$ at $(q_\infty,t^\pm)$ are horizontal, i.e., that 
$(\gamma^\pm)'(0) = 0$.  This relies on a flux calculation. We recall from \S \ref{subsec:flux} that if $1/\kappa$ is the normal 
derivative of the graph function $C_h$ at $r=1$ and $Z$ is any horizontal Killing field, then 
\begin{equation}
 \flux(C_h,\eta,E_3) = \frac{2 \pi}{\kappa}, \quad \flux(C_h,\eta,Z)  = 0. 
\label{catenoidfluxes} 
\end{equation}

\begin{claim} \label{cl:fluxes} Parametrizing the (undilated) ends $E_n^\pm$ by graph functions $u_n^\pm$, then as $n \to \infty$, 
\begin{eqnarray} 
& \int_0^{2 \pi}(u_n^+)_r(1, \theta) \, d \theta  \to 2 \pi/\kappa, \label{eq:vertical-1} \\
& \int_0^{2 \pi}(u_n^+)_r(1, \theta) (u_n^+)_\theta(1, \theta) \, d \theta  \to 0. \label{eq:horizontal-1}
\end{eqnarray}
\end{claim}
\medskip
Indeed, since $T_n(A_n) \to C_h$ smoothly on compact sets, there is a connected component $\lambda_n$ of $T_n(A_n) \cap 
\{t=0\}$ which generates the homology $H_1(T_n(A_n))$. Using the smooth convergence of $T_n(A_n)$ and the fact that fluxes are
independent of the representative of homology class and invariant under isometries, then for each $\eps > 0$, 
\[
\left| \flux(T_n(A_n),\lambda_n,E_3) - \frac{2 \pi}{\kappa}\right| <\epsilon , \quad  \left|\flux(T_n(A_n),\lambda_n,T_n(Z)) \right| <\epsilon,
\]
for $n$ sufficiently large. The claim follows. 

\medskip

Now focus just on the top curve, and for simplicity, drop the $+$ superscript. As noted earlier, $\gamma_n$ bounds 
a unique minimal disk $Y_n$ which is the vertical graph of a function $v_n(r,\theta) \in \calC^{2,\alpha}(\overline{\HH^2})$, 
cf.\ \cite[Proposition 3.1]{KM}.  As $n \to \infty$, $Y_n$ converges to the minimal disk $Y$ with graph function $u$.
The limit $u$ of the functions $u_n$ is attained in $\calC^{2,\alpha}$ on $\overline{\h^2} \setminus \{ (1,0) \}$. 
Using $Y_n$ as a barrier, we have 
\begin{equation} \label{eq:primi}
(v_n)_r(1,\theta) \leq (u_n)_r(1,\theta)\quad \forall\, \theta,
\end{equation}
and since the $v_n$ converge in $\calC^{2,\alpha}(\overline{\HH^2})$ to $u$, then we obtain that
\begin{equation}
\label{eq:secon} 
(u_n)_r(1, \theta) \geq -a > -\infty,
\end{equation}
uniformly in $\theta$ and for all $n$.   Moreover, since the pullbacks $T_n^* u_n = u_n \circ T_n$ converge along with their
derivatives when $\theta \neq \pm \pi$, the radial derivatives of these functions are strictly positive for $|\theta| \leq \pi - \epsilon$
and $n$ large. Since $T_n$ is conformal, the radial derivatives of the original functions $u_n$ are positive on $T_n^{-1}( \sph^1 \setminus I)$
where $I$ is any small interval around $(-1,0)$.   

\begin{claim}\label{cl:gamma} 
For each $\zeta > 0$ there exists a sequence of decreasing open arcs $\varUpsilon_n \subset \sph^1$ converging to $(1,0)$ such that 
\[
\int_{\Upsilon_n} (u_n)_r \geq \frac{2\pi }{\kappa} - \zeta. 
\]
\end{claim}
\medskip

To prove this, note that if this were to fail, then for some $\zeta > 0$ and any neighborhood $\varUpsilon$ of $(1,0)$ in $\sph^1$, some subsequence, 
still labeled $u_n$, would satisfy 
\[
\int_\varUpsilon(u_n)_r <  \frac{2\pi }{\kappa} - \zeta. 
\]

Now, $Y$ is simply connected, hence $\int_{\sph^1} u_r= 0$, so given any $\zeta' > 0$, there exists a neighborhood 
$\varUpsilon$ of $(1,0)$ in $\sph^1$ such that 
\[
\left| \int_{\sph^1\setminus\varUpsilon} u_r \right| < \zeta'. 
\]
But $(u_n)_r \to u_r$ uniformly on $\sph^1 \setminus\varUpsilon$, so 
\[
\left| \int_{\sph^1 \setminus \varUpsilon} (u_n)_r \right|  <  \zeta', \ \ \ n \gg 1,  
\]
which implies that 
\[
\int_{\sph^1} (u_n)_r =\int_{\sph^1 \setminus \varUpsilon} (u_n)_r+ \int_{\varUpsilon} (u_n)_r< \zeta' + 2\pi/\kappa - \zeta
\]
for arbitrarily large $n$. But we can choose $\zeta' < \zeta$, which then contradicts \eqref{eq:vertical-1}. Hence the decreasing sequence of
intervals $\varUpsilon_n$ with the stated properties exists. 

Finally, since the integral of $u_r$ over the entire circle vanishes, its integral over the complement $\calU$ of any sufficiently small neighborhood of 
$(1,0)$ can be made arbitrarily small.  Since $(u_n)_r  \to u_r$ on $\calU$, we may assume that the intersection of all the $\Upsilon_n$
equals the single point $(1,0)$. This finishes the proof of the claim.

\medskip

We now show that $u_\theta(1,0) =0$. If this were not the case, then it is either positive or negative, and to be definite we assume that it is positive. 
Take an arc $\sigma \subset \sph^1$ centered at $(1,0)$ and positive constants $c_1 < c_2$ such that 
$$
0 < c_1  < u_\theta |_{\sigma} < c_2;
$$
since $(u_n)_\theta \to u_\theta$ on $\sigma$, then for $n$ large, 
\begin{equation}\label{eq:beta}
0 < c_1  < (u_n)_\theta |_{\sigma} < c_2.
\end{equation}

Now consider the horizontal flux
\[
\flux(A_n,\lambda_n,Z)= \int_{\sph^1} (u_n)_r(1,\theta) (u_n)_\theta(1,\theta) \, d \theta,
\]
where $Z$ is the horizontal Killing field $Z$ induced by rotations. Fixing $\zeta > 0$, \eqref{eq:horizontal-1} implies that 
$\left| \flux(A_n,\lambda_n,Z) \right| < \zeta$ when $n$ is large.  In addition, 
\[
\int_{\sph^1}u_r(1,\theta) u_\theta(1,\theta) \, d \theta=0,
\]
so there is an arc $\beta$ of length less than $\zeta$, with $(1,0) \in \beta \subset \sigma$, and satisfying
\begin{equation}
\left| \int_{\sph^1 \setminus \beta} u_r(1,\theta) u_\theta(1,\theta) \, d \theta \right| < \zeta.
\label{eq:joa-1}
\end{equation}

Since $(u_n)_r (u_n)_\theta \to u_r u_\theta$ on $\sph^1 \setminus \beta$, we see from \eqref{eq:joa-1} that for $n$ large, 
\begin{equation} \label{eq:roabastos} 
\left| \int_{\sph^1 \setminus \beta} (u_n)_r(1,\theta) (u_n)_\theta(1,\theta) \,d \theta\right| < \zeta. 
\end{equation}
We may as well assume that $\Upsilon_n \subset \beta$. Then, recalling that $(u_n)_r > 0$ on $\Upsilon_n$, and using \eqref{eq:secon}, \eqref{eq:beta} 
and \eqref{eq:roabastos}, we obtain 
\begin{multline*}
\zeta > \flux(A_n,\lambda_n,Z) =\int_{\Upsilon_n}  (u_n)_r (u_n)_\theta+ \int_{\beta \setminus \Upsilon_n} (u_n)_r (u_n)_\theta+ 
\int_{\sph^1 \setminus \beta} (u_n)_r (u_n)_\theta  \\
\hfill 
\geq   c_1 \int_{\Upsilon_n} (u_n)_r- c_2\, a\,   |\beta \setminus \Upsilon_n|  -\zeta, \hfill 
\end{multline*}
where $a$ is given by \eqref{eq:secon}.

Finally, by Claim~\ref{cl:gamma} and the fact that $|\beta| < \zeta$, we see finally that 
$$
\zeta > \flux(A_n,\lambda_n,Z)\geq  c_2 (2\pi/\kappa - \zeta) + c_3 \zeta,
$$
which is a contradiction when $\zeta$ is small.  This shows that $u_\theta(0) = 0$ and completes the proof of the theorem. 
\end{proof}

Of course, Case II is not vacuous since for example $A_n = T_n(C_h)$ is a sequence which diverges in this manner. The family 
of catenoids also provides an example in Case III, as $h \to \pi.$

\medskip

Before proceeding to Case III we consider the limit set $\Lambda_\infty:= \lim(\overline{A_n}) \cap (\del_\infty\h^2 \times \R)$ of the sequence $A_n$. Certainly $\Gamma^\pm \subset \Lambda_\infty$,
and we set $\ell_\infty = \overline{ \Lambda_\infty \setminus (\Gamma^+ \cup \Gamma^-)}$. 
We have shown that in Case I, $\ell_\infty = \varnothing$ and in Case II, $\ell_\infty$
is a single vertical segment of length less than $\pi$.  We now study what can happen in the remaining case.

First, we state the following lemma.
\begin{lemma} \label{lem:case3}
Let $A_n$ be a sequence satisfying Case III. Consider $p_0 \in \partial_\infty \h^2$ and $t^\pm$ such that
  $(p_0,t^\pm) \in \Gamma^\pm$. 
  
\begin{enumerate}[i)] 
\item If $t^+-t^- > \pi$, then  $(\{p_0\}\times (t^-,t^+) ) \cap \ell_\infty=\varnothing.$
\item If $(p_0,t_0)\in  \ell_\infty$ with $t^-<t_0<t^+$, then $(\{p_0\}\times (t^-,t^+) ) \subset \ell_\infty$.
\end{enumerate}
\end{lemma}
\begin{proof}
First we prove assertion {\em i)}. Since $t^+-t^- > \pi$, then we can consider a tall rectangle $\mathcal{R}$, whose asymptotic boundary is 
arbitrarily close to the vertical segment $\{p_0\}\times (t^-+\epsilon,t^+-\epsilon)$, with $\epsilon >0$ sufficiently small so that 
$(t^++\epsilon)-(t^+-\epsilon)>\pi$. 
Then we can use $\mathcal{R}$ as barrier to prevent that $\{p_0\}\times (t^-,t^+)$ were in the limit set
of our annuli.

Next, we verify assertion {\em ii)}.  Suppose  that  there exists $(p_0,\tilde{t}) \in (\{p_0\}\times (t^-,t^+) )\setminus  \ell_\infty$. 
Then, there exists $\rho>0$ such that $B((p_0,\tilde{t}),\rho) \cap \ell_\infty=\varnothing$. 
We can assume that $\tilde{t}<t_0$ (the case $\tilde{t}>t_0$ is similar). Hence we can find an arc $\sigma$, 
$p_0 \in \sigma \subset (\sph^1 \cap D(p_0, \rho))$
and a sufficiently small $\epsilon>0$ so that the tall rectangle $R(\sigma, t_0-\pi, t_0+\epsilon)$
verifies $(R(\sigma ,t_0-\pi, t_0+\epsilon) \cap (\h^2 \times \{\tilde{t}\})) \subset B((p_0,\tilde{t}),\rho)$.  Then we could use
the piece of the tall rectangle $R(\sigma,t_0-\pi, t_0+\epsilon )\cap  (\h^2 \times [\tilde{t},t_0+\epsilon])$ as a barrier to prove that
$(p_0,t_0) \not \in \ell_\infty$  which contradicts our hypothesis. \end{proof}

We can finally proceed with the analysis of sequences of surfaces satisfying the conditions of Case III.

 \begin{theorem} \label{th:daniel} Let $A_n$ be a sequence satisfying Case III. Then $\ell_\infty$ is a (non-empty) union of 
vertical segments in $\partial_\infty\h^2\times\R$ of length at most $\pi$ joining $\Gamma^+$ and $\Gamma^-$ .
For each one of these vertical segments $\{p_\infty\}\times(t^-,t^+) \subset 
\ell_\infty$, there exists a sequence of points $(p_n,t_n) \in A_n$ where the unit normal to $A_n$ is horizontal,
with $(p_n,t_n)$ converging to a point in $\{p_\infty\} \times (t^-, t^+)$.  

Moreover, there exists at least one segment in $\ell_\infty$ such that if $T_n$ is a sequence of horizontal isometries mapping $(p_n,t_n)$ to $(o, t_n)$, then
$T_n(A_n)$ converges smoothly on compact sets of $\h^2\times \R$ to a parabolic generalized catenoid. In this particular case, $t^+-t^-=\pi$.
 \end{theorem}

\begin{proof}

By hypothesis, we can assume there exists, for any $n\in\n$, a point $(p_n,t_n)$ in $A_n\cap(\partial D_{\h^2}(o,R_n)\times\R)$ with horizontal
normal vector such that, after passing to a subsequence,   $\{(p_n,t_n)\}_{n\in\n}$ diverges to a point in $\{p_\infty\}\times\R$,
for some $p_\infty\in\partial_\infty\h^2$.  Consider $(p_\infty,t^\pm)=(\{p_\infty\}\times\R) \cap\Gamma^\pm$.
Let $T_n$ be a sequence of horizontal isometries mapping $(p_n,t_n)$ to $(o, t_n)$ and denote $\Sigma_n=T_n(A_n)$. We observe that the geodesics passing through $p_n$ and $o$ converge to the geodesic passing through $o$ with end point
$p_\infty$. Thus, the disks $\mathcal{H}_n=T_n(D_{\h^2}(o,R_n))$ converge to the horodisk $\mathcal{H}_\infty$ at $-p_\infty$
passing through the origin, since $R_n\to\infty$.  We can consider a larger horodisk $\widetilde{\mathcal{H}}_\infty$ at $-p_\infty$ containing
$\mathcal{H}_\infty$. It is clear that $\widetilde{\mathcal{H}}_\infty$ also contains any $\mathcal{H}_n$, for $n$ big enough.

We call $\sigma^\pm=\partial_\infty\h^2\times\{t^\pm\}$. Let $u_n^\pm$ be the smooth function defined on $\Lambda:=\h^2 \setminus\widetilde{\mathcal{H}}_\infty$ whose graph 
represents the ends of $\Sigma_n$ around $\sigma^\pm$.  By the Arzel\`{a}-Ascoli theorem, a subsequence of $\{u_n^\pm\}_{n
\in \n}$ converges uniformly on compact subsets of $\Lambda$ to a minimal graph $u_\infty^\pm$, with $u_\infty^\pm=t^\pm$ 
on $\partial_\infty\h^2\setminus\{-p_\infty\}$.

Now set $M_n:=\Sigma_n\cap(\widetilde{\mathcal{H}}_\infty\times\R)$; each $M_n$ is an annulus bounded by two curves 
$\varrho_n^\pm$ in $\partial\widetilde{\mathcal{H}}_\infty\times\R$. By possibly enlarging $\widetilde{\mathcal{H}}_\infty$, we can 
assume that $\{\varrho_n^\pm\}_{n\in \n}$ converges uniformly on compact sets to the graph of $u_\infty^\pm \mid_{\partial\widetilde{\mathcal{H}}_\infty}$.  
Hence it is easy to see that the boundary measures of the annuli $M_n$ are uniformly bounded on compact sets.
Thus, by Theorem~\ref{th:Z}, the area blowup set $Z$ of the minimal annuli $M_n$ (or $\Sigma_n$), which lies in the solid 
truncated horocylinder, obeys the same maximum principles that hold for properly embedded minimal surfaces without boundary. Assume $Z$ is not empty.
We consider the foliation of $\h^2 \times \R$ by vertical geodesics planes orthogonal to the geodesic 
joining $p_\infty$ and $-p_\infty.$
By Theorem \ref{th:Z} we get that one of these  vertical planes  is included in  
the area blowup set~$Z$, which is absurd.

Since in $\Lambda\times\R$ the convergence of the annuli $\overline \Sigma_n$ is 
smooth with multiplicity one, then we can apply Theorem
\ref{th:ZZ} to deduce that we have the same convergence inside of the solid horocylinder. Therefore, the minimal annuli $\Sigma_n$
converge to a complete, embedded minimal surface $\Sigma_\infty$ with asymptotic boundary $\sigma^+ \cup \sigma^-$ and 
possibly some points in $\{-p_\infty\}\times \R$. 
Furthermore, reasoning as in the proof of Theorem \ref{th:catenoid}, it is not hard
to see that $\Sigma_\infty \subset \h^2 \times [t^-,t^+].$ 

We are going to prove that $\Sigma_\infty$ coincides with an isometric copy of
$\mathcal{D}$. Firstly, let us prove that their asymptotic boundaries have the same behavior.
\begin{claim}\label{cl:boundary}
The sequence $\{(p_n,t_n)\}_{n \in \n}$ cannot converge to a point in $\Gamma^+\cup\Gamma^-$. In other words, if we denote $\overline t=\lim t_n$,
then $t^- < \overline t<t^+.$ Moreover, $\partial_\infty\Sigma_\infty=\sigma^+ \cup \sigma^-\cup
(\{-p_\infty\}\times (t^-,t^+) )$.
\end{claim}
We proceed by contradiction. Assume that $\lim t_n =t^+$ (the other case is similar). 
Then $T_n(p_n,t_n)$ converges to $(o,t^+) \in \Sigma_\infty.$ By the maximum principle, 
we have that  $\h^2 \times \{t^+\} \subset \Sigma_\infty.$ But the normal vector
to $\Sigma_\infty$ at $(o,t^+)$ is horizontal, which is absurd. This contradiction proves the first assertion in the claim. 
The second part of this claim is a direct consequence of item \emph{ii)} in Lemma \ref{lem:case3}.

Up to a vertical translation, we can assume $t^-=-t^+$.
\begin{claim}\label{cl:daniel}
$\Sigma_\infty=\mathcal{D}$, up to an isometry and $t^+=\pi/2$.
\end{claim}


Fix any point $q \in \partial_{\infty}\h^2 \setminus \{ p_\infty \}$ and let $z$ be a point in one of the arcs in
$\sph^1$ between $q$ and $p_\infty$. Let $P_z$ be the vertical plane over the geodesic in $\h^2$ with endpoints
at $q$ and $z$.  We now reflect $\Sigma_\infty$ with respect to this family of vertical planes, letting $z$ move
from $q$ toward $p_\infty$. This straightforward application of the Alexandrov reflection principle shows that
$\Sigma_\infty$ is symmetric with respect to $P_{p_\infty}$.  Since $q$ is arbitrary, we see that $\Sigma_\infty$
is symmetric with respect to any vertical plane with $\{p_\infty\} \times \R$ as one of its asymptotic boundaries.
Using a well-known theorem in hyperbolic geometry, we see that each slice $\h^2 \times \mbox{const.}$
is a horocycle. Hence $\Sigma_\infty$ is foliated by horocycles, and therefore $\Sigma_\infty=\mathcal{D}$.

\vskip 0.3cm

To finish the proof of Theorem~\ref{th:daniel}, it remains to prove that, for any vertical 
segment in $\ell_\infty$, there exists $\{p_n'\}_{n\in\n}$ converging to a point in this vertical segment,
where $p_n'\in A_n$ is a point with horizontal normal vector.  Up to a rotation we can assume that the
vertical segment is contained in $\{(1,0)\} \times \R$. Suppose on the contrary that there 
not exists such a sequence. Then we can find a geodesic $\eta$ perpendicular to the curve
$\{z \in \h^2, \Im(z)= 0\}$ such that $V_n \cap (D \times \R)= \varnothing$, for all $n \in \n$,
where $D$  is the lens-shaped region between 
  $\eta$  and the arc of $\sph^1$ delimited by $\eta$ that contains $\{(1,0)\}$ and
  $V_n \subset A_n$  is the set of all points where the tangent plane is vertical. 
Notice that   $A_n \cap (D \times \R)$ are graphs over the domain $D$. 
  
  Consider
  $\{(q_n,s_n)\} \in A_n \cap (D \times \R)$ a sequence converging to a 
  point in $\{(1,0)\} \times (s^-,s^+)$ where $s^\pm=\Gamma^\pm \cap (\{(1,0)\} \times \R)$.
  Let $S_n$ be a sequence of horizontal isometries mapping $(q_n,s_n)$ to $(o, s_n)$ and denote $\widetilde{A}_n=S_n(A_n)$. 
  We observe that the geodesics passing through $q_n$ and $o$ converge to the geodesic passing through $o$ with end point
$(-1,0)$.  Hence the limit of $\widetilde{A}_n$, $\widetilde{A}$, consists of the union of two disks $\h^2 \times \{s^\pm\}$ and the
vertical segment $\{(-1,0)\} \times [s^-,s^+]$. But the sequence $\{S_n(q_n,s_n)\}$ converges to
a point in $\{o\} \times (s^-,s^+)$ which is not contained in $\widetilde{A}$. \end{proof}

\begin{corollary} \label{co:cuca}
Let $\{A_n\}_{n \in \n}$ be a sequence of annuli in Case III. Then the sequence of vertical fluxes 
satisfies
$$\lim_{n \to \infty} f_0(A_n) =-\infty.$$
\end{corollary}
\begin{proof}
Let $T_n$, $\Sigma_n$, $p_{\infty}$, $u_n^-$  and $u^-$ be as  
described at the beginning of the proof of Theorem \ref{th:daniel}. By \eqref{eq:f0} we know that
$$ f_0(\Sigma_n) = \int_{\sph^1} \left((u_n)_r^-(1,\theta)-(v_n)_r^-(1, \theta)\right)\, d\theta \;,  $$
where $v_n^-$ denotes the smooth function defined on $\h^2$ whose graph represents the unique minimal
disk bounded by $\Pi^-(\Sigma_n)$. We also denote by $v^-$ the constant function obtain as the limit of $v_n^-$. 
As $\lim_{z \to -p_\infty}u_r^-(z)=-\infty$ and $v_r^-\equiv 0$, for any $\varepsilon>0$, we obtain
$$\int_{\sph^1\setminus \beta(\varepsilon)} \left(u_r^-(1,\theta)-v_r^-(1, \theta)\right)\, d\theta =-\infty\;,  $$
where $\beta(\varepsilon)$ is an arc in $\sph^1$ centered at $-p_\infty$ of length $\varepsilon$. So
given $N \in \n$, there exists $\varepsilon_N>0$ and $n_0$ such that 
$$\int_{\sph^1\setminus \beta(\varepsilon_N)} \left((u_n)_r^-(1,\theta)-(v_n)_r^-(1, \theta)\right)\, d\theta <-N\;,  $$
for all $n \geq n_0$. Then
\begin{multline*}
f_0(A_n)=f_0(\Sigma_n)=\int_{\sph^1\setminus \beta(\varepsilon_N)} \left((u_n)_r^-(1,\theta)-(v_n)_r^-(1, \theta)\right)\, d\theta \\
+
\underbrace{\int_{\beta(\varepsilon_N)} \left((u_n)_r^-(1,\theta)-(v_n)_r^-(1, \theta)\right)\, d\theta }_{<0}<-N\;.  
\end{multline*}
Then we have $\displaystyle \lim_{n \to \infty} f_0(A_n) =-\infty.$
\end{proof}
\begin{remark} \label{re:cuca}
Notice that the proof of the previous corollary works if the sequence $\{A_n\}$
is contained in $\fA^*$. Actually we only use that the ends are embedded graphs
and that the annuli satisfy \eqref{eq:open}.
\end{remark}
\vskip 3mm

We would like to point out that Theorem \ref{th:daniel} has a useful application 
for rotationally invariant annuli.
\begin{theorem} \label{th:last}
Let $m \in \n$, $m\geq 2$ and consider a sequence of minimal annuli
$\{A_n\}_{in\in \n}$ in $\fA_m$. Assume that the sequence of boundary curves
$\{\Gamma_n:=\Pi(A_n)\}_{n \in \n}$ (which is a sequence of curves in $\fC_m$) 
satisfies that 
$\{\Gamma_n\}_{n \in \n} \to \Gamma_0 \in \fC_m$ in the $\calC^{2,\alpha}$
topology. Then, up to a subsequence, $\{A_n\}_{n \in \n}$ converges (smoothly
on compact sets) to a properly embedded minimal annulus 
$A_0 \in \fA_m$ such that $\Pi(A_0)=\Gamma_0.$
\end{theorem}

\begin{proof}
Since the annulus $A_n$ is in $\fA_m$, we deduce the existence of a radius $R_n>0$ such that 
$$
A_n \setminus \D_{\h^2}(R_n) \times \R
$$
is the union of two vertical graphs. Theorem \ref{th:daniel} says us that the sequence
$\{R_n \}_{n \in \n}$ is bounded; otherwise the limit curve $\Gamma_0$ cannot belong
to $\fC_m$, because there must be points whose vertical distance is precisely $\pi.$

We can now reason as in the proof of Theorem \ref{th:annulus} to deduce the existence of the limit  
annulus $A_0$. As $A_n$ is $\mathcal{R}_m$-invariant, for all $n \in \n$, then the limit is
also $\mathcal{R}_m$-invariant.
\end{proof}
\begin{corollary} \label{co:m-proper} Given $m \in \n$, $m \geq 2$, then the projection $\Pi: \fA_m \to \fC_m$
is proper.
\end{corollary}

We finish this section with an observation similar to Remark \ref{re:contact-I}.

\begin{remark} \label{re:contact-II}
If the limit curve $\Gamma_0\equiv(\gamma_0^+, \gamma_0^-)$ in Theorem \ref{th:last} satisfies 
$\gamma_0^+(\theta) \geq \gamma_0^-(\theta)$, for all $\theta \in \sph^1$, but $\gamma_0^+ \neq \gamma_0^-$, then the 
statement of the theorem remains true.
\end{remark}

\subsection{ The properness of the map $\widetilde \Pi$}

We conclude this section with the proof of the properness of $\widetilde{\Pi} : \widetilde{\fA^*}  \longrightarrow 
\mathcal{C}^{2,\alpha}(\sph^1)^2 \times \RR \times \D$, where, as introduced earlier,
$$
\widetilde{\fA^*} = \widetilde{\Pi}^{-1} (\mathcal{C}^{2,\alpha}(\sph^1)^2  \times \RR \times \D).
$$ 

As we did in the previous subsection, given $\{A_n\}_{n \in \n}$ a sequence in $\fA^*$, the ends of $A_n$
can be written as $(z,u_n^\pm(z))$ where $u_n^\pm: \h^2\setminus D_{\h^2}(q_n,R_n) \rightarrow \R$ is a 
$\mathcal{C}^{\infty}$ function that extends $\mathcal{C}^{2,\alpha}$ up to the boundary.

\begin{lemma} \label{lem:primo}
Let $\{A_n\}_{n \in \n}$ be a sequence in $\fA^*$ such that:
\begin{itemize}
\item $\{f_0(A_n)\} \to h_0<0,$
\item $\{\Pi_-(A_n)\} \to \gamma_0^-$ in $\mathcal{C}^{2, \alpha}(\sph^1),$ 
\item $\overline{\h^2}\setminus D_{\h^2}(q_n,R_n)$ converge to $\overline{\h^2} \setminus \{q_\infty\}$,  $q_\infty \in \partial_\infty \h^2$.
\item $\{u_n^-\}$ converge to the function $u_0^-: \overline{\h^2} \to \R$, smoothly on compact sets of $\h^2$ and 
$\mathcal{C}^{2, \alpha}$ on compact sets of $\overline{\h^2} \setminus \{q_\infty\},$ where the graph of $u_0^-$
is the (only) minimal disk bounded by $\gamma_0^-$.
\end{itemize}
Then the sequence of centers $\mathbf{C}(A_n)$ diverges in $\h^2$, i.e., $|\mathbf{C}(A_n)| \to 1$.
\end{lemma}
\begin{proof} Let us write $q_\infty:=e^{\imag \theta_0}$.
Rotating, we can assume that $0<\theta_0 <\pi/2.$

From our hypotheses, we know that $|(u_n^-)_r-(v_n^-)_r| < 1/n$, in an arc $L_n$ of 
$\sph^1 \setminus \{e^{\imag \theta_0}\}$, with $\cup_n L_n=\sph^1 \setminus \{e^{\imag \theta_0}\}.$
Label $L_n'=\sph^1-L_n$. If $n$ is large enough, then there are angles 
$0<\theta_1^n <\theta_0<\theta_2^n<\pi/2$, such that $L_n'=\{e^{\imag \theta} : \theta_1^n <\theta<\theta_2^n\}.$
Hence
\begin{multline*}
G_1(A_n)= \frac{\int_{\sph^1} \cos \theta ((u_n^-)_r-(v_n^-)_r) d\theta}{\int_{\sph^1} ((u_n^-)_r-(v_n^-)_r) d\theta} = \\
 \frac{\int_{L_n'} \cos \theta ((u_n^-)_r-(v_n^-)_r) d\theta+\int_{L_n} \cos \theta ((u_n^-)_r-(v_n^-)_r) d\theta}{\int_{L_n'} ((u_n^-)_r-(v_n^-)_r) d\theta+\int_{L_n} ((u_n^-)_r-(v_n^-)_r) d\theta}
\end{multline*}
and so
\begin{multline*}
\frac{\cos (\theta_1^n) \int_{L_n'} ((u_n^-)_r-(v_n^-)_r) d\theta+\int_{L_n} \cos \theta ((u_n^-)_r-(v_n^-)_r) d\theta}{\int_{L_n'} ((u_n^-)_r-(v_n^-)_r) d\theta+\int_{L_n} ((u_n^-)_r-(v_n^-)_r) d\theta}  \geq 
G_1(A_n) \geq \\
\frac{\cos (\theta_2^n) \int_{L_n'} ((u_n^-)_r-(v_n^-)_r) d\theta+\int_{L_n} \cos \theta ((u_n^-)_r-(v_n^-)_r) d\theta}{\int_{L_n'} ((u_n^-)_r-(v_n^-)_r) d\theta+\int_{L_n} ((u_n^-)_r-(v_n^-)_r) d\theta}.
\end{multline*}
Since the integrals over $L_n$ converge to $0$ and 
$$
\left\{\int_{L_n'} ((u_n^-)_r-(v_n^-)_r) d\theta \right\} \to h_0,
$$
we deduce that $\{G_1(A_n)\} \to \cos \theta_0.$ Similarly, $\{G_2(A_n)\} \to \sin \theta_0.$ 
\end{proof}

\begin{lemma} \label{lem:secondo}
Let $A_n \in \fA^*$ be a sequence such that $\Pi(A_n)$ converges to a pair of curves $\Gamma_0=(\gamma_0^+,
\gamma_0^-) \in \mathcal{C}^{2,\alpha}(\sph^1)^2$, $\{G_0(A_n)\}$ is bounded in $\R$, and finally
that the curves $V_n= \{ p \in A_n : \, \langle \nu_n(p), E_3\rangle =0\}$ remain in
a compact region $K$ of $\h^2 \times \R.$ Then, up to a subsequence, $A_n$ converges to
a minimal annulus $A_0 \in \fA^*$ with $\Pi(A_0)=\Gamma_0$.  The convergence is smooth 
on the interior and $\calC^{2,\alpha}$ up to the boundary. 
\end{lemma}
\begin{proof}

We know that each $A_n \setminus (A_n \cap K)$ 
is a union of two graphs. Using classical elliptic estimates and the Arzel\`{a}-Ascoli theorem, 
these two sequences of graphs converge smoothly to minimal graphs
for which the boundaries at infinity are $\gamma_0^\pm$. 

Note that $\gamma_0^+\neq \gamma_0^-$. Indeed, if this were the case, the sequence of minimal annuli $A_n$
would converge to the minimal disk $D_0$ spanned by $\gamma_0^+= \gamma_0^-$. This would force the vertical 
flux $f_0(A_n)$ to converge to $0$, and hence $G_0(A_n) \to \infty$, contrary to assumption. 

Next, we claim that $A_n \cap K$ must converge smoothly to a regular annulus with boundary inside $K$. 
To prove this, we use again that the vertical flux $f_0(A_n)$ is bounded away from $0$. The key point is that
it is impossible for $A_n \cap K$ to `pinch'.  Suppose that this were to occur. Then there would exist points $p_n \in 
K \cap A_n$ at which the shape operator $S_n$ of $A_n$ satisfies
$$
\lambda_n:= |S_n(p_n)|=\max\{ |S_n(p)| \; : \; p \in K \cap A_n\} \longrightarrow +\infty.
$$
Then the rescaled surfaces $\frac{1}{\lambda_n} (A_n-p_n)$ would converge to a complete minimal surface  $A_\infty$ in $\R^3$
which passes through the origin and with $|S_\infty(0)|=1.$ From Proposition \ref{prop:V} we have that
the Gauss map $\nu_\infty : A_\infty \to\sph^2$ takes each value in the equator at most once. Moreover, $A_\infty$ has the topology of either a disk or an annulus. 
Using a result by Mo and Osserman \cite{mo-osserman},  $A_\infty$ has finite total curvature $-4 \pi$. Then $A_\infty$ is either a catenoid or a copy of
Enneper's surface. However, Enneper's surface is not Alexandrov-embedded, so it must be a catenoid. Thus at each point where $|S_n|$ blows up,  a 
catenoidal neck is forming. This cannot happen at more than one point, since if this were to occur at two distinct points, there would be an enclosed annular 
region for which both boundary curves are very short, and this violates the isoperimetric inequality.  Denoting this point of curvature blowup by $p_0$, then
away from $p_0$, $A_n$ converges to the union of two disks $D_0^+$ and $D_0^-$, each of which is a graph over $\h^2$.  Hence the sequence
of curves $V_n$ must converge to $p_0$, and they must have bounded length. However, the length of $V_n$
equals the vertical flux $\mbox{Flux}(A_n,V_n,E_3)$, and this convergence would force this flux to tend to $0$, which we have assumed is
not the case. 

We have now proved that $A_n$ must converge to an Alexandrov-embedded minimal annulus with embedded ends.  It is clear from the convergence outside $K$ that 
$\Pi(A_0)=\Gamma_0$.
\end{proof}

\begin{lemma} \label{lem:tertio}
Let $\{A_n\}$ be a sequence in $\widetilde{\fA^*}$ satisfying that:
\begin{itemize}
\item $\Pi(A_n)$ converges to a pair of curves $\Gamma_0=(\gamma_0^+,\gamma_0^-) \in 
\mathcal{C}^{2,\alpha}(\sph^1)^2.$
\item $\{G_0(A_n)\}$ is bounded in $\R.$ 
\item The curves $V_n$ have bounded length, but they diverge in $\h^2 \times \R.$
\end{itemize}
Then  there exists a vertical line $E$ contained in $\partial(\h^2\times \R)$ such that $\{A_n\}$ converges, up to a subsequence,
smoothly on $\overline{\h^2 \times \R} \setminus E$, to $D_0^+ \cup D_0^-$, where $D_0^+$ and 
$D_0^-$ are the minimal disks spanned by $\gamma_0^+$ and $\gamma_0^-$, respectively.
In particular, the sequence of centers $\mathbf{C}(A_n)$ diverge in $\h^2.$
\end{lemma}
\begin{proof} Since the sequence $\{V_n\}$ diverges but their lengths remain bounded, their limit set is contained in a vertical line 
$ E:=\{q_\infty\} \times \R$ for some $q_\infty \in \sph^1.$ We know  by Proposition \ref{prop:V} that $A_n\setminus V_n$ is 
the union of two graphs $A_n^+$ and $A_n^-$. Reasoning as in Theorem \ref{th:catenoid}, the ends $A_n^\pm$ converge as 
minimal graphs, smoothly in the interior and in $\calC^{2,\alpha}$ on compact sets of 
$(\overline{\h^2}\setminus\{q_\infty\}) \times \RR$, to the minimal disks $D_0^\pm$.    
A direct application of Lemma \ref{lem:primo} gives that $|\mathbf{C}(A_n)| \to 1.$
\end{proof}

\begin{lemma} \label{lem:quarto}
Let $\{A_n\}$ be a sequence in $\fA^*$ satisfying that:
\begin{itemize}
\item $\Pi(A_n)$ converges to a pair of curves $\Gamma_0=(\gamma_0^+,\gamma_0^-) \in 
\mathcal{C}^{2,\alpha}(\sph^1)^2.$
\item The curves $V_n= \{ p \in A_n \; : \; \langle \nu_n(p), E_3\rangle =0\}$ escape from any compact region of $\h^2 \times \R$.
\end{itemize}
Then, $\displaystyle \lim_{n \to \infty}G_0(A_n) = -\infty.$
\end{lemma}
\begin{proof}
Recall that $\mbox{Flux}(A_n,V_n,E_3)=\int_{V_n} \langle E_3,\eta_{V_n}\rangle \, d\sigma$.
From the hypotheses of this lemma, we have that there are points $p_n \in V_n$,
so that the sequence $\{p_n\} \to p_\infty \in \partial_\infty \h^2 \times \R$. 

Then, using 
Corollary \ref{co:cuca} and Remark \ref{re:cuca} we have that $\displaystyle \lim_{n \to \infty} G_0(A_n) = -\infty.$
\end{proof}

We now prove the main result of this subsection.
\begin{theorem} \label{teo:properness}
Consider a sequence of elements $(A_n,x_n,y_n,z_n)\subset \widetilde{\fA^*}$ such that 
$$
(\widetilde{\Gamma}_n, \lambda_n,c_n)=\widetilde{\Pi} (A_n,x_n,y_n,z_n) \longrightarrow (\widetilde{\Gamma}_0,\lambda_0,c_0) \in
\fC \times \R \times \D.
$$ 
Then some subsequence of the $(A_n,x_n,y_n,z_n)$ converges to $(A_0,x_0,y_0,z_0) \in \widetilde{\fA}$ and 
$\widetilde{\Pi} (A_0,x_0,y_0,z_0) = (\widetilde{\Gamma}_0,\lambda_0,c_0)$.
\end{theorem}
\begin{proof}
Since the sequence 
\begin{eqnarray*} 
(\lambda_n,c_n) &=& (G_0(A_n),G_1(A_n)+ \imag \;G_2(A_n) = \\
&=& \left(f_0(A_n)-f_0(A_n)^{-1},\frac{f_1(A_n)}{f_0(A_n)}+\imag \; \frac{f_2(A_n)}{f_0(A_n)}\right)
\end{eqnarray*}
converges, we deduce that the sequence of vertical fluxes $\{f_0(A_n)\}_{n \in \n}$ converges to a negative
constant $h_0$.  As a consequence $\{(f_1(A_n),f_2(A_n))\}_{n \in \n}$ converges to a point $p_0 \in \R^2$.
On the other hand, by assumption, 
$$
\widetilde{\Gamma}_n=(\gamma_n^-,\gamma_n^+ + x_n+y_n \cos \theta +z_n\sin \theta) \longrightarrow 
\widetilde{\Gamma}_0=(\widetilde{\gamma}_0^-,\widetilde{\gamma}_0^+),
$$
where $\gamma_n^\pm=\Pi_\pm (A_n)$.    The first, easy, consequence is that $\gamma_n^- \to {\gamma}_0^-=\widetilde{\gamma}_0^-$.  
\begin{claim} \label{also}
The sequence $ \{\gamma_n^+\}$ also converges.
\end{claim}

To prove this, define $\omega_n:=\sup \gamma_n^-$ and $\alpha_n:=\inf {\gamma}_n^+$.  First observe that if the parameters 
$(x_n, y_n, z_n)$ are so large that $\alpha_n - \omega_n> \pi$, then we can use Corollary \ref{co:non} to get a contradiction. However, this does not yet bound the individual components of this parameter set.

To do this, we show first that $\gamma_n^+$ cannot be too `tilted', i.e., that $|(x_n,y_n, z_n)| \leq C$ for some $C$ which depends 
on $ \widetilde \gamma_0^+$. 

If this is not the case, then $\gamma_n^+$ becomes increasingly tilted and converges as $n \to \infty$ to one of these configurations in $\partial _\infty(\h^2 \times \R)$: 
\begin{enumerate}[(i)]
\item A vertical halfline, $S$.
\item A vertical line, $E$.
\item Two vertical lines, $L_1$ and $L_2$.
\end{enumerate}
In Cases (i) and (ii), the limit of $A_n$  consists of the minimal 
disk $D_0^-$ spanned by $\gamma_0^-$ and a subset of a vertical line $E$ in $\partial_\infty (\h^2 \times \RR)$. We then apply
Lemma \ref{lem:primo} to deduce $\{\mathbf{C}(A_n)\}$ diverges in $\h^2$, contrary to hypothesis.

In Case (iii), using tall rectangles with increasing height as barriers, we can prove
that the ideal boundary in $\partial_\infty (\h^2 \times \R)$ of the limit top end consists of $L_1$, $L_2$
 and the geodesics in $\h^2\times \{\pm \infty\}$ joining the lines $L_1$ and $L_2$. Then, the vertical flux
 of such an end should be zero. Indeed, using Theorem 3 in \cite{HMR} we deduce that the limit top end
 has finite total curvature. Hence, from Proposition 2.2 in \cite{MMR}, using Fermi coordinates $(\rho,s,t)$
 of the vertical 
 geodesic plane determined by $L_1$ and $L_2$, we can parametrize that end
in this way
 $$X(r,\theta)=(r \cos(\theta), s(r,\theta), r \sin(\theta)) \; , \quad r \in [0, \infty[ \; , \theta \in [0, 2 \pi] \; ,$$
 where
 $$ s(r, \theta)= A(\theta) r^{-\frac{1}{2}} {\rm e}^{-r}+ \mathcal{O} (r^{-\frac{3}{2}} {\rm e}^{-r}) $$
with $A(\theta)$ a bounded function. 
Recall that in these coordinates the metric of $\h^2\times \R$ is $\cosh(s)^2 d\rho^2+ds^2+dt^2$.
Therefore, the vertical flux of the curve $r=r_0$, for a big enough $r_0$, is given by
$$ \int_0^{2 \pi} \left(\frac{\sin(\theta) }{\sqrt{1+\sinh^2(s(r_0,\theta)) \cos^2(\theta)+s_r^2(r_0,\theta)}}+\mathcal{O}(r_0^{-1}) \right)d\theta \; .$$
Taking limit as $r_0 \to \infty$, we conclude that the vertical flux is zero, which is contrary to our assumptions.

This establishes that $\gamma_n^+$ also converges, and that $(x_n, y_n, z_n) \to (x_0, y_0, z_0)$, so
$\gamma_n^+ \to {\gamma}_0^+=\widetilde{\gamma}_0^+ -x_0 -y_0 \cos \theta +z_0 \sin \theta$, proving Claim \ref{also}.

We need to prove finally that the sequence of annuli $\{A_n\}_{n \in \n} \subset \fA^*$ converges smoothly to 
an annulus $A_0 \in \fA^*$ and  $\Pi(A_0)=(\gamma^-_0,\gamma^+_0)$.  By the Lemmas \ref{lem:tertio} and \ref{lem:quarto}, 
the curves $V_n= \{ p \in A_n \; : \; \langle \nu_n(p), E_3\rangle =0\}$ remain in a compact region of
$\h^2\times \R$, because the vertical fluxes and the centers are bounded. So, we can
apply Lemma \ref{lem:secondo} to deduce the existence of the $A_0 \in \fA^*$.

\end{proof}

\section{The asymptotic Plateau problem for minimal annuli} 
We now assemble the results above to prove various types of local and global existence theorems.   
Our goal, of course, is to determine as much information as possible about the space of minimally fillable 
curves in $\fC$.  We start with a few qualitative remarks.  First, it is apparent that the
projection $\Pi: \fA \to \fC$ has some sort of fold around the catenoid family.  Indeed, $\Pi^{-1}(\sph^1 \times \{\pm h\})$
is noncompact for every $0 < h < \pi/2$. In addition, we have exhibited a specific infinite dimensional
family of curves converging to a pair of circles which are not minimally fillable, while on the other hand, because
of the existence of nondegenerate minimal annuli arbitrarily near the catenoid, any one of these pairs
of parallel circles is the limit of pairs of curves which are in the interior of the image of $\Pi$.  Thus 
a precise characterization of this image may not be possible.   We present two
separate existence results which are nonperturbative and give the existence of infinite-dimensional
families of minimal annuli far away from the catenoid.  The key question not answered here is
whether there is indeed a failure of compactness, or equivalently, if $\Pi$ is proper away from the 
catenoid family.   We have reason to suspect that there are many other regions where properness
may fail, but have so far been unsuccessful in demonstrating this. 

The most general existence result that we can prove is the following theorem, which summarizes all the information that we have 
about the map $\widetilde \Pi.$

\begin{theorem} \label{th:principal}
The map $\widetilde \Pi : \widetilde{\fA^*}\longrightarrow \mathcal{C}^{2,\alpha}(\sph^1)^2 \times \RR \times \D$ is a proper Fredholm map 
of index $0$ and degree $1$. In particular, given any $\gamma^\pm \in \mathcal{C}^{2,\alpha}(\sph^1)^2$, there exist constants $a_0$, $a_1$, $a_2$ so that the pair $(\gamma^+ + a_0+ a_1\cos \theta +
a_2 \sin \theta, \gamma^-)$  bounds a proper, Alexandrov-embedded, minimal annulus with embedded ends. 
\end{theorem}
The proof of this theorem is a direct consequence of Theorem \ref{pfi0}, Proposition \ref{prop:uno}
and Theorem \ref{teo:properness}.

Theorem \ref{th:principal} asserts that we can prescribe the bottom curve of an Alexandrov-embedded minimal annulus with embedded ends,
as well as the top curve up to a translation and tilt, and in addition the ``center of the neck'' and certain fluxes. 

\subsection{Solutions with symmetry} 
Let $G$ be any finite group of isometries of $\HH^2 \times \RR$ which leaves invariant some fixed catenoid $A_0$. 
Assume in addition that no element of $\fJ^0(A_0)$ is left invariant by $G$.    There are two main examples: the group $\calR_k$,
$k \geq 2$, generated by the rotation by angle $2\pi/k$ around the vertical axis which is the line of symmetry of the catenoid, 
and the group $\ZZ_2 \times \ZZ_2$ generated by reflection across a vertical plane bisecting the catenoid and rotation by $\pi$ 
around the line orthogonal to that plane which intersects the midpoint of the catenoid's neck. Examples of curves with the first 
type of symmetry are obvious. For the second, a key example is a pair of parallel ellipses. 

\begin{theorem} \label{th:symmetric}
Let $\Gamma = \gamma^\pm$ be any $G$-invariant pair of curves such that 
\begin{equation}
\sup_\theta |\gamma^+(\theta) - \gamma^-(\theta)| < \pi.
\label{gap}
\end{equation}
Then $\Gamma$ is minimally fillable.
\end{theorem}
\begin{proof}  We shall work in the setting of $G$-invariant objects, mappings, etc.  In this context,
the catenoid $A_0$ is nondegenerate and the local deformation theorem is an immediate consequence
of the implicit function theorem.  Denoting by $\fA_G$ and $\fC_G$ the Banach manifolds of
$G$-invariant minimal annuli and boundary curves, we have that $\Pi_G: \fA_G \to \fC_G$ 
is Fredholm of index $0$, just as in the non-$G$-invariant setting.  Furthermore, by the
compactness arguments of the last section, this mapping is proper over the space of elements
of $\fC_G$ satisfying \eqref{gap}.   We do not need the full set of arguments developed in the last section. Indeed, 
for $G$-invariant sequences of minimal annuli $A_i$, it is necessary to rule out that the neck shrinks or 
expands, but not that it remains of bounded size and escapes to infinity, since that is ruled out by $G$-invariance.
On the other hand, we do require non-$G$-equivariant techniques to rule out 
the possibility that the necksize increases without bound. 

We have shown that $\Pi_G$ is a proper Fredholm map of index $0$. A fairly general argument which can be
found in the paper \cite{White-degree} implies that this map has a $\ZZ$-valued degree, defined by the formula
\[
\mathrm{deg}(\Pi_G) = \sum_{ A \in \Pi_G^{-1}(\Gamma)}  (-1)^{\mbox{\tiny index}_G(A)},
\]
where $\Gamma$ is any regular value of $\Pi_G$.  The key fact which allows this argument to be used
essentially verbatim in the present setting is that the Jacobi operator has discrete spectrum in a neighborhood
of $0$ so that eigenvalue crossings can be counted just as in the compact case. By the Sard-Smale theorem, a generic
element of $\fC_G$ is regular, and since we have shown that there exists a neighborhood in
$\fA_G$ around the catenoid which projects diffeomorphically to a neighborhood
in $\fC_G$, there exists a regular value $\Gamma$ for which $\Pi_G^{-1}(\Gamma)$ is
nonempty.  Recall that $\mbox{\rm index}_G(A)$ is the number of negative eigenvalues, of
the (negative of the) Jacobi operator acting on $G$-invariant functions.

It remains therefore to prove that the degree of $\Pi_G$ is nonzero. However, the pair of
parallel circles $\Gamma_0$ separated by $h < \pi$ bounds only the catenoid, so the sum
above has only one term, which means that $\deg (\Pi_G)$ is equal to either $1$ or $-1$. 
In any case it is nonzero.   This proves the $G$-invariant existence result.
\end{proof}

We single out one particularly interesting family of solutions. Let $G = \ZZ_2 \times \ZZ_2$,
acting as described above, and let $\Gamma_s$ denote a family of parallel ellipses,
the upper one a translate by some fixed amount $h < \pi$ of the lowe.  The parameter
$s$ measures the tilt, and varies between $s=0$ (where $\Gamma_0$ is just the pair of
parallel horizontal circles) to the extreme limit where these ellipses become more and more
vertical. 

\subsection{Solutions with admissible boundary} 
Our other existence result allows us to consider more general curves and annuli. 

\subsubsection{Properness over the space of admissible curves}
The compactness results that we got in Section \ref{sec:compact} motivate the following
\begin{definition}
Let $\Gamma=(\gamma^+, \gamma^-)$ be a curve in $\fC^\pi$.
We say that $\Gamma$ is {\bf admissible}  if $\displaystyle \frac{d}{d \theta} \left( \gamma^+(\theta),  
\gamma^-(\theta) \right) \neq (0,0), $ for all $\theta \in [0,2 \pi)$, and set $\Omega:= \{ \Gamma \in \Cb^\pi \; : \; 
\Gamma \mbox{ is admissible} \}$. This is an {\bf open} subset of $\Cb$.  We also write $\mathcal{W}:= \Pi^{-1} (\Omega).$
\end{definition}

\begin{figure}[htbp]
    \begin{center}
        \includegraphics[width=.35\textwidth]{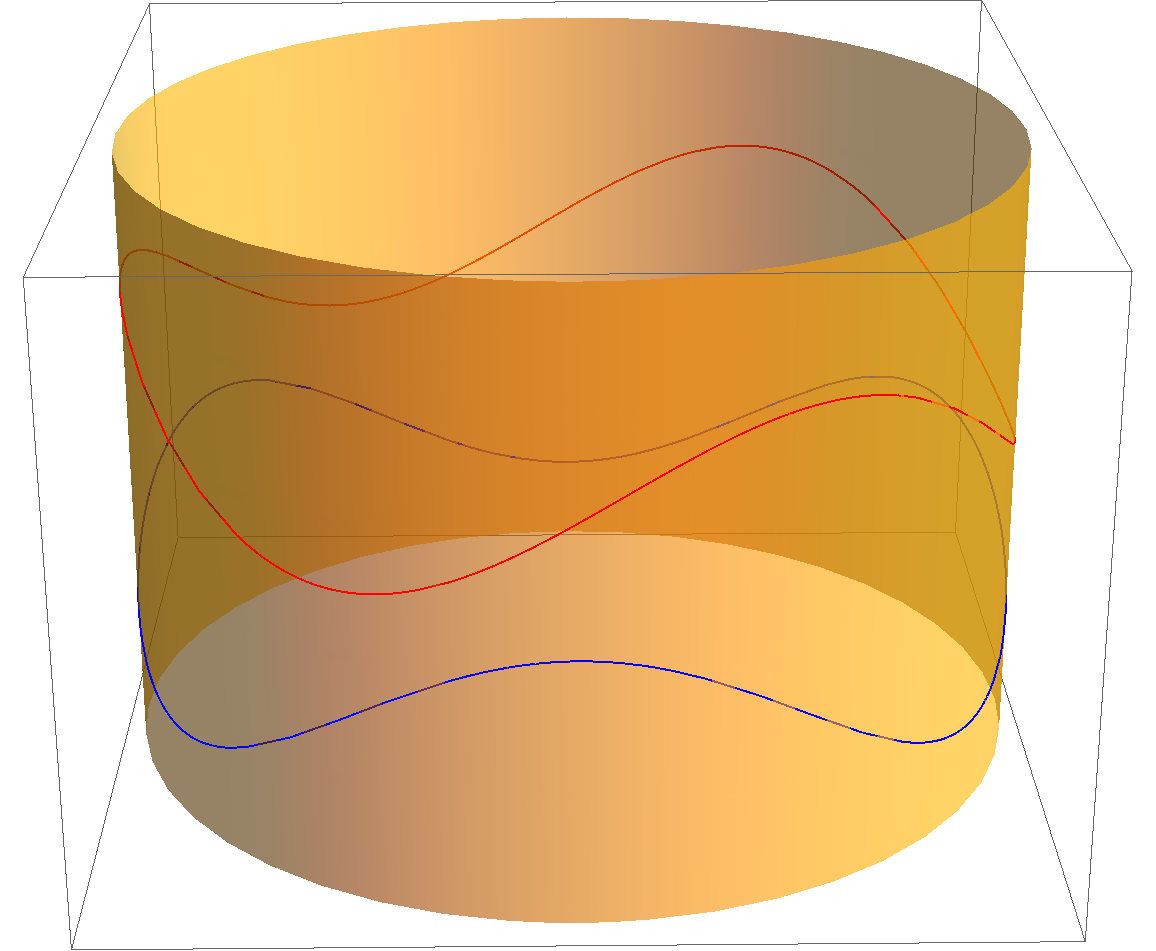}  
        \includegraphics[width=.35\textwidth]{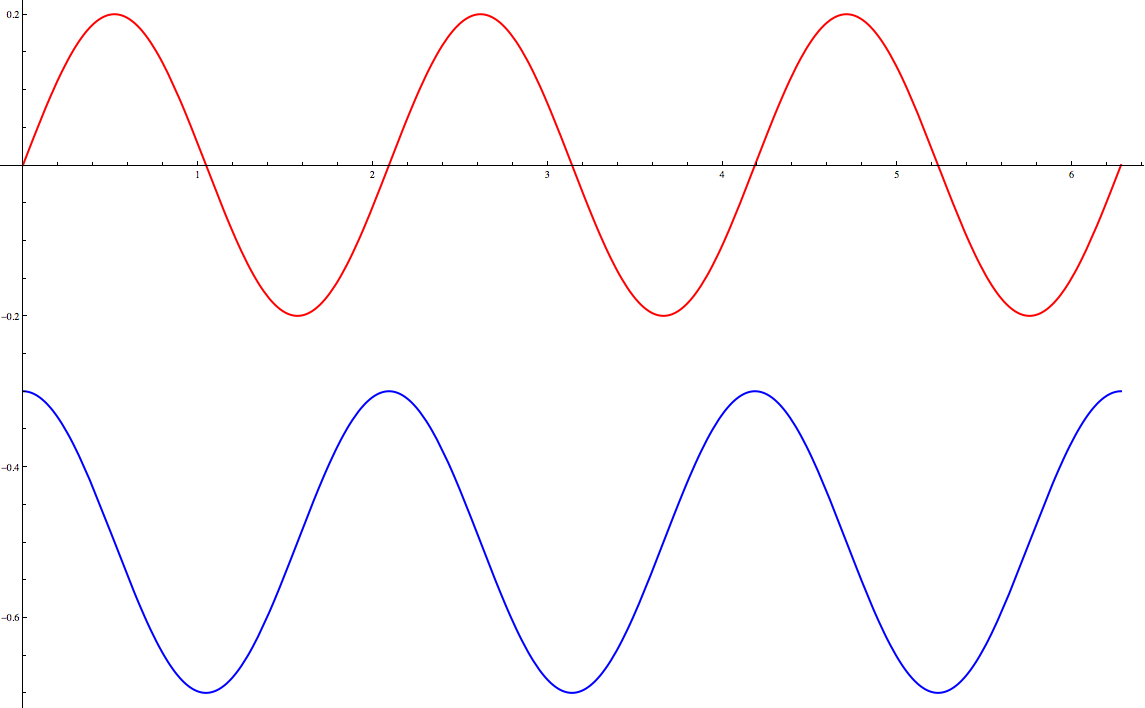}
     \end{center}
   \caption{An admissible curve at $\partial_\infty \h^2 \times \R.$} \label{fig:two}
\end{figure} 

Taking the previous definition into account,  we obtain the following compactness results: 
\begin{theorem}[Compactness] \label{th:compact2}
Let $A_n$ be a sequence in $\Ab$ such that $\Pi(A_n)$ converges to $ \Gamma_0 \in \Omega.$ 
Then, up to a subsequence, $A_n$ converges to  a minimal annulus $A_0\in \Ab^\pi.$
\end{theorem}
\begin{proof} Since $\Gamma_0 \in \Omega$, Theorem \ref{th:catenoid} and Theorem \ref{th:daniel} 
imply that $A_n$ is in Case I in Section 8.1, so Theorem \ref{th:annulus} concludes the proof.
\end{proof}
This has an immediate consequence.
\begin{corollary} \label{co:proper}
The map $\Pi |_{\mathcal{W}} : \mathcal{W} \rightarrow \Omega$ is {\bf proper}.
\end{corollary}
\vskip .5cm

If $\Gamma$ is admissible, then 
$$ 
\alpha_\Gamma(\theta)= \frac{d}{d \theta} \left( \gamma^+(\theta),   \gamma^-(\theta) \right): 
\sph^1 \to \CC \setminus \{0\}
$$
is a smooth loop. and hence homotopic in $\CC \setminus \{0\}$ to a standard $n$-cycle, 
$\alpha_n(\theta)= {\rm e}^{ \ir n \theta}$.  We say that $\Gamma$ is {\bf n-admissible}.
Given $n$-admissible curves $\Gamma_0$ and $\Gamma_1$, there exists a smooth 
isotopy $\Gamma_t$ between these amongst $n$-admissible curves. 

\begin{remark}
$\Omega$ has a countable number of path-connected components 
$$ 
\Omega_n:= \{ \Gamma \in \Omega \; : \; \Gamma\mbox{ is n-admissible} \} .
$$
Defining $ \mathcal{W}_n:= \Pi^{-1} (\Omega_n) \subset \mathcal{W},$ then the family of catenoids 
$\{ C_h \; : \; h \in(0,\pi/2)\}$, is in the boundary of $\mathcal{W}_n$ for every $n \in \z.$ 
\end{remark}

\subsection{Existence results.}
Consider, for any $m \in \n$, $m \geq 2$, the rotation $R_m$, and the finite group $\calR_m$
it generates. 

Fixing $h <\pi/2$, then $\Gamma_h= \sph^1 \times \{-h,h\} \in \fC^\pi$ spans a unique centered catenoid 
$C_h$. 
\begin{proposition} \label{pro:U}
There exists an open neighborhood of $\Gamma_h$, $U \subset \fC_m$, 
such that for any $\Gamma \in U$, then there is
a {\bf unique} annulus $A \in \fA_m$ with $\Pi(A)=\partial A=\Gamma$. 
\end{proposition}
\begin{proof}
First notice that all the arguments in Section \ref{sec:manifold} remain true restricted to the space of $\calR_m$-invariant 
functions over $C_h$.

By Proposition \ref{th:jacobin}, the space of decaying Jacobi fields $\fJ_h^0$ is generated by 
$$
\varphi=\frac{1-r^2}{r} \cos \theta, \quad \psi=\frac{1-r^2}{ r} \sin \theta,
$$
so ${\rm Ker} \left(D \left.\left(\Pi |_{\fA_m} \right) \right|_{C_h} \right)=\{0\}$. The Inverse Function Theorem gives
neighborhoods $C_h \in W \subset \fA_m$ and $\Gamma_h \in U \subset \fC_m$ such that $\Pi |_{W}:W \rightarrow U$ 
is a diffeomorphism.
 
To prove the uniqueness, we proceed by contradiction. Assume there exists a sequence $\Gamma_n \in \fC_m$ with 
$\Gamma_n \to \Gamma_h$ and such that there are two distinct $\mathcal{R}_m$-invariant minimal annuli 
$A_n^1 \neq A_n^2$, each satisfying $\Pi(A_n^i)=\Gamma_n$. By Theorem \ref{th:last}, $A_n^1 \to C_h$ and 
$A_n^2\to C_h$. 

Write $A_n^i$ as a normal graph over $C_h$ of a function $u_n^i$, $i=1,2$. Define
$$
v_n:=\left(\frac{u_n^2-u_n^1}{\|u_n^2-u_n^1\|_{\infty}} \right). 
$$
This vanishes on $\del C_h$ since $(u_n^2-u_n^1)|_{\Gamma_h}=0$.

Choose $p_n \in C_h$ so that $v_n(p_n)=1$. If $p_n$ diverges in $\h^2 \times \R$, then take a horizontal translation
$T_n$ such that $T_n(p_n)=(0,t_n)$. Clearly $\{t_n\} \to \{\pm h\}$, so we choose a subsequence for which 
$p_n$ converges to a point in $\sph^1 \times\{h\}.$  Then $v_n \circ {T_n}^{-1}$ converges to a Jacobi field 
on $\h^2 \times \{h\}$ which reaches a maximum at $(0,h)$. This is impossible.  This proves that, up to 
a subsequence, $p_n \to p_0 \in C_h$. 

It is now standard to deduce the existence of a limit $v:=\lim_{n \to \infty} v_n$, which is a nontrivial $\mathcal{R}_m$-invariant 
element of $\fJ^0(C_h)$. However, there are no such elements. This is a contradiction, and hence we have
proved the uniqueness. 
\end{proof}

\begin{remark}\label{re:non-degen}
Reasoning as above, we can also prove that if $A$ is $\calR_m$-invariant and sufficiently close to $C_h$, then 
there are no $\calR_m$-invariant elements of $\fJ^0(A)$. 
\end{remark}

When $m=2$, even more is true. In the following, $U$ denotes the neighborhood of $\Gamma_h$
provided by Proposition \ref{pro:U} for $m=2.$
\begin{proposition} \label{pro:non-degen}
For every neighborhood $U' \subset U \subset \fC_2$ containing $\Gamma_h$, and for any {\it even}
integer $n \neq 0$, there exists $\Gamma \in U' \cap \Omega_n$ such that the unique 
annulus $A \in \fA_2 \cap \mathcal{W}_n$ with $\Pi(A)=\Gamma$ is {\bf non-degenerate}, 
in the sense that $\fJ^0(A)=\{0\}$.
\end{proposition}
\begin{proof}
Since $n$ is even, we can construct a family $\{A_\varepsilon\}$ of minimal annuli satisfying,
for all $|\varepsilon| < \varepsilon_0$:
\begin{enumerate}[(a)]
\item $A_0=C_h$ and $\Pi(A_\varepsilon) \subset U'$;
\item $A_\varepsilon \in \mathcal{W}_n$;
\item $A_\varepsilon $ is $\mathcal{R}_2$-invariant, but {\bf not} $\mathcal{R}_{2 k}$-invariant, $k>1.$
\end{enumerate}
By Proposition  \ref{pro:nondegenerate} below (see also Remark \ref{re:tur}), we deduce that $A_\varepsilon$ is non-degenerate, for almost
all $\varepsilon >0.$ 
\end{proof}

\begin{theorem} \label{th:degree}
If $n$ is even and nonzero, the projection $$\Pi: \calW_n \longrightarrow \Omega_n$$ is a proper map of degree
$\pm 1 \; ($mod $2)$. In particular, given $\Gamma \in \Omega_n$, there exists a properly embedded minimal annulus
such that $\Pi( A)=\Gamma.$
\end{theorem}
\begin{proof} By Corollary \ref{co:proper}, $\Pi|_{\mathcal{W}_n}$ is proper. Thus $\Pi|_{\mathcal{W}_n}$ has a well-defined 
degree:
$$
\deg(\Pi|_{\mathcal{W}_n}):= \sum_{A \in \Pi^{-1}(\Gamma)} (-1)^{{\rm index}(A)},
$$
where $\Gamma$ is any regular value of $\Pi$.  Regular values are generic in $\fC$.

The rotation $R_2$ is a diffeomorphism of $\mathcal{W}_n$. Fix $\Gamma_0$ in the open neighborhood $\calU$ of 
Proposition \ref{pro:U}, for $m=2$. Then $\Gamma_0$ spans a unique $R_2$-invariant annulus $A_0$. By Proposition 
\ref{pro:non-degen} we may choose $\Gamma_0$ so that $A_0$ is non-degenerate. 

Enumerate the other (non-congruent) annuli in $\Pi^{-1}(\Gamma_0)$ by $A_1, \ldots ,A_k.$ If $\Gamma_0$ is sufficiently
close to $\sph^1 \times\{-h,h\}$, then any non-symmetric solution $A_i$, $i \in \{1, \ldots , k\}$,
creates $2$ different annuli $A_i^j:=R_2^j(A_i)$, $j=0,1,$ each in $\Pi^{-1}(\Gamma_0)$.

Let $\fU_0$ be an $R-2$-invariant neighborhood of $A_0$ in $\mathcal{W}_n$ for which $\fJ^0(A)=\{0\},$ for all $A \in \fU_0$.
Choose further neighborhoods $\fU_i$ of $A_i$ in $\mathcal{W}_n$, $i=1, \ldots,k$ which are pairwise disjoint
from each other and fro m$\fU_0$, and such that $R_2(\fU_i) \cap \fU_j =\varnothing,$ for all $i,j=1, \ldots,k$.
Now choose $\Gamma \in \Omega_n$ near $\Gamma_0$, which is a regular value of $\Pi$ (possibly $\Gamma_0=\Gamma$.) 
Our compactness results imply that 
$$
\Pi^{-1}(\Gamma) \subset \bigcup_{i=0}^k \left(\fU_i \cup R_2(\fU_i) \right).
$$
Furthermore, clearly $\deg \left(\Pi |_{\fU_i}\right)=\deg \left(\Pi |_{R_2(\fU_i)}\right),$ $i=1, \ldots, k$. Therefore
$$
\deg(\Pi)=(-1)^{{\rm index}(A)}+ 2 \cdot  \left( \sum_{i=1}^k \deg \left(\Pi |_{\fU_i}\right) \right),
$$
where $A$ is the unique element in $\Pi^{-1}(\Gamma) \cap \fU_0.$ This concludes the proof.
\end{proof}

\subsection{The asymptotic Dirichlet problem for non-disjoint curves.} 
In the preceding we have considered pairs of curves $(\gamma^+,\gamma^-)$ satisfying $\gamma^+(\theta) >\gamma^-(\theta)$ 
for all $\theta$.  However, we can extend these results about $R_m$-invariant solutions to allow boundary curves 
$\Gamma=(\gamma^+,\gamma^-)$ for which $\gamma^+(\theta) \geq \gamma^-(\theta)$ for all $\theta$, but 
$\gamma^+ \not\equiv \gamma^-$. With this sense of $\gamma_+ \geq \gamma_-$, define
$$
\fC^\ast:=\{ (\gamma^+,\gamma^-) \; :  \: \; \gamma^+ \geq \gamma^- \; \mbox{and} \; \sup_{\theta \in \sph^1} 
\left( \gamma^+(\theta) - \gamma^-(\theta)\right) < \pi  \}
$$
$$
\fC^\ast_m:=\{(\gamma^+,\gamma^-) \in \fC^\ast \; : \; \mbox{$(\gamma^+,\gamma^-)$ are $R_m$-invariant}  \}
$$

\begin{figure}[htbp]
    \begin{center}
        \includegraphics[width=.40\textwidth]{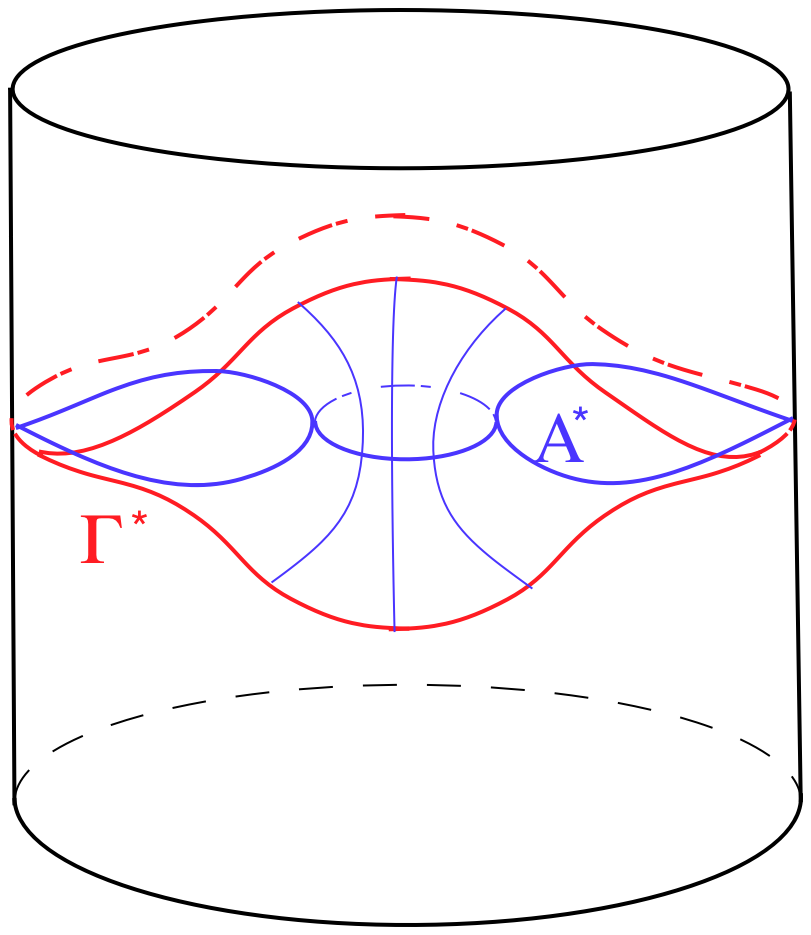}  
           \end{center}
   \caption{Annuli of this kind can be obtained as limits of our examples.} \label{fig:limit} 
\end{figure} 

\begin{theorem} \label{th:limit}
If $\Gamma^* \in \fC^*_m$, then there exists a complete, properly embedded minimal annulus $A^*$ such that $\Pi(A^*)=\Gamma^*.$
\end{theorem}
\begin{proof}
Consider a sequence of curves $\Gamma_n \in \fC_m$ converging to $\Gamma^*.$  By Theorem \ref{th:symmetric}, for each $n$
there exists a properly embedded minimal annulus $A_n \in \fA_m$ such that $\Pi(A_n)=\Gamma_n$. 
Now use Theorem \ref{th:last} and Remark \ref{re:contact-II} to deduce that a subsequece of $A_n$ converges, 
smoothly on compact sets, to a minimal properly embedded minimal annulus $A^*$. By construction, this annulus satisfies
$\Pi(A^*)=\Gamma^*.$
\end{proof}

\section{Existence of nondegenerate minimal annuli}
The main result of the previous section is of a somewhat general nature and does not preclude,
for example, the possibility that {\it every} $A \in \fA$ is degenerate, i.e., that it is possible
that $\fJ^0(A) \neq 0$ for all $A \in \fA$.  We prove here that this is not the case.
\begin{proposition} \label{pro:nondegenerate}
There exist minimal annuli arbitrarily near to any catenoid $A_0$ which are nondegenerate. 
\end{proposition}
\begin{proof}
We have proved that $\fA$ is a smooth Banach manifold, so it makes sense to talk about smooth curves 
$A_\epsilon$ of minimal annuli which are deformations of the catenoid $A_0$. These are normal graphs over $A_0$ of 
a smooth family of functions $u_\epsilon$, and we set $\psi \in \fJ(A_0)$ to be the $\epsilon$-derivative of $u_\epsilon$ 
at $\epsilon = 0$.  Thus $\psi \in \fJ(A_0)$, so by the results in the previous section, its leading coefficients
are orthogonal to the normal derivatives of every $\phi \in \fJ^0(A_0)$. In other words,
\begin{equation}
(\psi_0^+, \psi_0^-) \perp  (\cos \theta, \cos\theta), \quad (\psi_0^+, \psi_0^-) \perp  (\sin \theta, \sin\theta).
\label{constraint}
\end{equation}
Moreover, we can construct such families of minimal annuli for any Jacobi field which satisfies \eqref{constraint}.

Now, within the space of Jacobi fields $\fJ(A_\eps)$, there are two distinguished subspaces:  the decaying
Jacobi fields $\fJ^0(A_\eps)$ and another space $\calD(A_\eps)$ consisting of Jacobi fields on $A_\eps$ 
generated by horizontal dilations in $\HH^2 \times \RR$.  We are assuming that $\fJ^0(A_\eps)$ is nontrivial,
and standard eigenvalue perturbation theory implies that $\dim \fJ^0(A_\eps) \leq 2$. Furthermore, we
also have that $\calD(A_\eps)$ converges to $\fJ^0(A_0)$ since the latter is also generated by dilations.
Thus when $\eps \neq 0$, the two spaces are quite close to one another.  We are interested in the way in
which one approaches the other.  

Let us first assume that $\dim \fJ^0(A_\eps) = 2$ for all small $\eps$. We handle the other case later. 
Choose a codimension two subspace $W_\eps \subset \fJ(A_\eps)$ which varies smoothly in $\eps$ and which 
is always complementary to $\fJ^0(A_\eps)$ (for example, we can take an orthogonal complement with respect
to some weighted Hilbert structure). Choose any smooth family $\phi_\eps \in \calD(A_\eps)$, and
decompose it into components in each of these subspaces, $\phi_\eps = \psi_\eps + w_\eps$,
where $\psi_\eps \in \fJ^0(A_\eps)$ and $w_\eps \in W_\eps$.  Note that $w_\eps$ and $\phi_\eps$ agree 
at $\del A_\eps$. 

We now observe that the map which assigns to any $w \in W_\eps$ its leading coefficients, i.e.,
boundary values, $(w_0^+, w_0^-)$, has image equal to a codimension $2$ subspace of 
$(\calC^{2,\alpha}(\sph^1))^2$ which depends smoothly on $\eps$ and which equals the
subspace given by the orthogonality conditions \eqref{constraint} when $\eps = 0$.
This map is, by definition, injective and also surjective. Hence by the open mapping theorem,
the norm of any $w \in W_\eps$ is equivalent to the norm of its boundary values.  This means
that the rescaled sequence of functions
\[
\tilde{w}_\eps =  w_\eps/ \sup_{\del A_\eps} |w_\eps|
\]
is uniformly bounded, independently of $\eps$, and always attains the value $1$ at some point of the 
boundary.  Thus it has a well defined limit, $\tilde{w}$, which is therefore an element of $\fJ(A_0)$.  Its 
boundary values must satisfy \eqref{constraint}. 

Now let us compute the boundary values of the original Jacobi field $\phi_\eps$. Parametrizing
the boundary curves of $A_\eps$ by the functions $f_\eps^\pm(\theta)$, suppose that $D_\lambda$
is a family of horizontal dilations and let $A_{\eps,\lambda} = D_\lambda(A_\eps)$, corresponding
in turn to a family of normal graphs over $A_\eps$ by a family of functions $u_{\eps, \lambda}$. 
We calculate readily that at $r = 1$ and at the upper and lower halves, up to an overall constant factor, 
\[
\left. \frac{d\, }{d \lambda} u_{\epsilon, \lambda}(1,\theta)  \right|_{\lambda=0}  = 
(\sin (\theta + \beta) (f_\eps^+)'(\theta),  \sin (\theta + \beta) (f_\eps^-)'(\theta) ),
\]
for some $\beta$ depending on the family of dilations.  This is equal to
the pair boundary values of the Jacobi field $w_\eps$, and hence also of the Jacobi field $\tilde{w}_\eps$. 

In the final limit, we are dividing by the supremum of $\sin(\theta+B) (f_\eps^\pm)'$ and letting
$\eps \to 0$. However, this is equivalent to taking the derivative in $\epsilon$ of this family.
By the earlier definition, the derivative of $f_\eps^\pm$ with respect to $\eps$ equals the Jacobi field
$\psi$ on $A_0$.  This proves finally that the boundary values of the limiting Jacobi field $\tilde{w}$
above must be
\begin{equation}
(\tilde w_0^+, \tilde w_0^-) = ( \sin(\theta + \beta) (\psi_0^+)'(\theta), \sin(\theta + \beta) (\psi_0^-)'(\theta)).
\label{gerbil}
\end{equation}

Now let us choose the Jacobi field $\psi$ generating the family $A_\eps$. Expanding boundary values into 
their Fourier series
\[
\psi_0^\pm(\theta) =  \sum_{k = 0}^\infty (a_k^\pm \cos k\theta + b_k^\pm \sin k \theta)
\]
then the constraint \eqref{constraint} is equivalent to two conditions $a_1^+ + a_1^- = b_{1}^+ + b_{1}^- = 0$,  
and apart from these, all the other Fourier coefficients can be chosen arbitrarily (sufficiently small,
of course, and so that the resulting function has the correct regularity). However, evaluating the expressions 
on the right in \eqref{gerbil} yields a function which contains 
\[
\underbrace{(-a_2^\pm \cos \beta+b_2^\pm \sin \beta)}_{B_1^\pm} \cos \theta + \underbrace{(-b_2^\pm \cos \beta-a_2^\pm \sin \beta)}_{B_2^\pm} \sin \theta 
\]
for any choice of $B_1^\pm$, $B_2^\pm$, since these coefficients depend only on the coefficients $a_2^\pm, b_2^\pm$.  
However, this is a contradiction, since these are the leading coefficients of the element of $\tilde w \in \fJ$. 
This proves that it is impossible that $\dim \fJ^0(A_\eps) = 2$ for all $\eps$. 

We are reduced to the case where $\dim \fJ^0(A_\eps) = 1$ for almost every $\eps$.   In the preceding part 
of the proof, it is not hard to argue slightly differently to show that in fact there cannot even exist a
sequence $\eps_j \searrow 0$ for which $\dim \fJ^0(A_\eps) = 2$ and with appropriate Fourier coefficients nonzero,
so we can assume that the dimension is $1$ for all $\eps \neq 0$.  In this case we can actually still proceed almost 
as before.  The difference is that the family of dilations $D_\lambda$ generating $A_{\eps,\lambda}$ is no longer
arbitrary.   Instead,  observe that the eigenspace $\fJ^0(A_\eps)$ (at least along a sequence
$\eps_j \searrow 0$) must have a limit $E_0$, which is a one-dimensional subspace of $\fJ^0(A_0)$
Since $\calD(A_\eps) \to \fJ^0(A_0)$ smoothly, we can choose a subspace $E_\eps \subset \calD(A_\eps)$
which converges smoothly to $E_0$.  It is now clear that the difference vector $w_\eps$ defined 
in the earlier step of the proof may still be chosen so that its normalization has a limit as $\eps \to 0$.
The remainder of the argument proceeds exactly as before. 
\end{proof}
\begin{remark} \label{re:tur}
It is not hard to see from this argument that we can produce families of solutions $A_\epsilon$ converging
to $A_0$ which are nondegenerate and also invariant with respect to rotation by $\pi$ around the axis of $A_0$.
These are important  in one of our global existence theorems 
(Proposition \ref{pro:non-degen} and Theorem \ref{th:degree}).
\end{remark}

\appendix
\section{Two tall rectangles are not area-minimizing} \label{sec:tall}

The purpose of this appendix is to prove that a couple of tall rectangles in $\h^2 \times \R$ of the same height with
the vertical segments in common are not an area-minimizing surface. Recall that for a complete surface {\em 
area-minimizing} means that any compact piece minimizes the area among all the surfaces with the same boundary.
To prove this, we are going to see that, in general, a couple of horizontal disks connected by vertical 
totally geodesic planes near
infinity
(see Figure \ref{fig:non}) are better competitors for the area functional.
\begin{figure}[htbp]
\begin{center}
\includegraphics[width=.6\textwidth, height=.4\textheight]{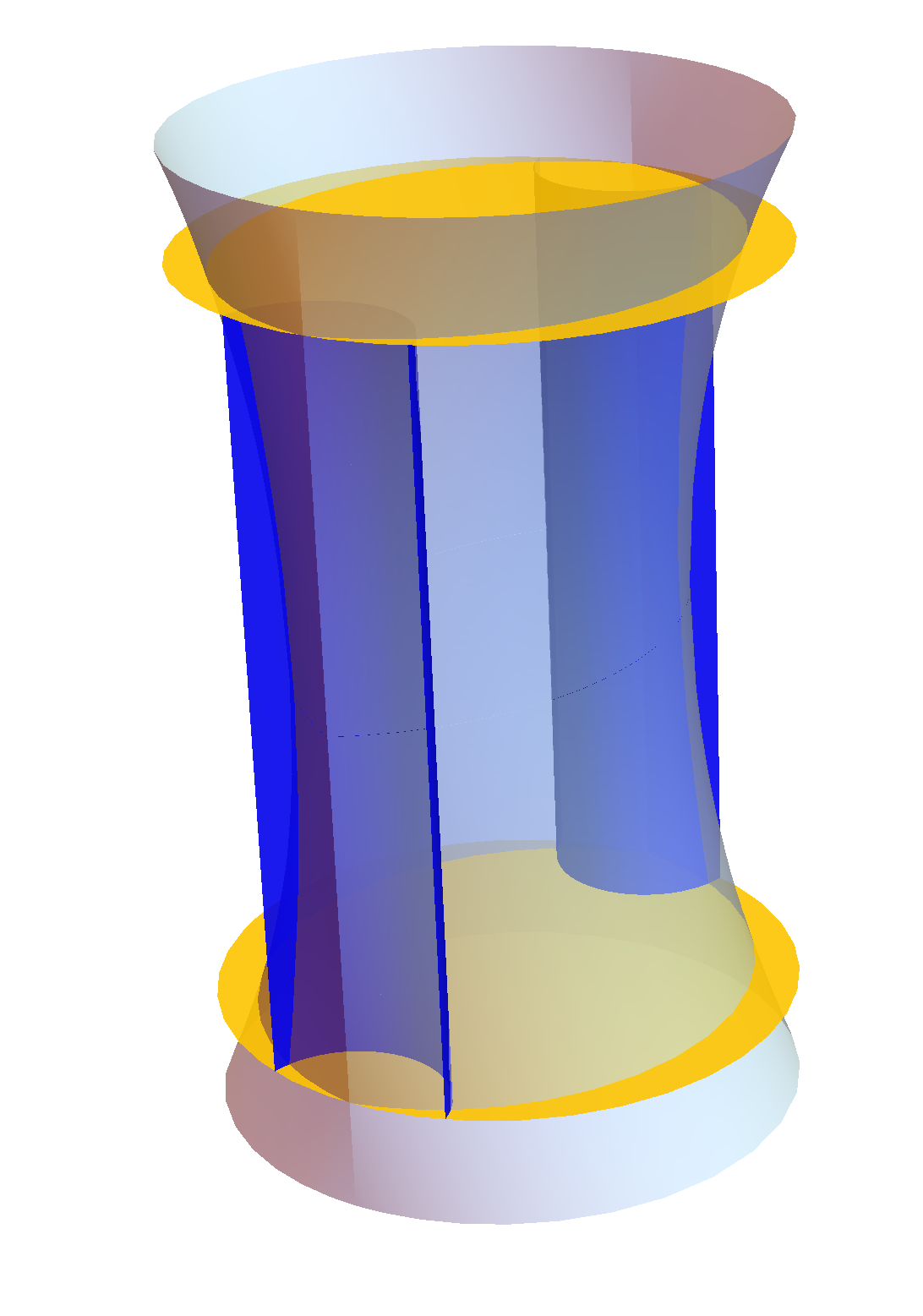}
\caption{The tall rectangles, the horizontal disks and the vertical totally geodesic planes.}
\label{fig:non}
\end{center}
\end{figure}

The simplest analytic representation of these surfaces is obtained using the upper half-space model for $\h^2$. Thus we use
standard coordinates $(x,y) \in \h^2$ with $y > 0$, and hence coordinates $(x,y,t)$ on $\h^2 \times \RR$.  We shall place
the vertical segments of the tall rectangles over the points $(0,0)$ and infinity.  For a given $r >0$ we denote $\Gamma_r$
the geodesic in $\h^2$ with end points $(r,0)$ and $(-r,0)$. Let $P_r:=\Gamma_r \times \R$ be the 
corresponding totally geodesic plane. For simplicity we will write $\Gamma=\Gamma_1$ and $P=P_1$.
Finally, we denote $\Lambda$ as the geodesic given by $\{x=0, y>0\}$. The point in the intersection $\Gamma \cap \Lambda$
is denoted by $o=(0,1)$.

On the geodesic $\Gamma$ we consider $\rho$ the signed distance to the point $o$. Then the induced metric in $P$
is given by $d \rho^2 +dt^2.$ We write a curve in $P$ in coordinates $(\rho,t)$ of the form 
$$\rho \mapsto (\rho,\lambda(\rho)).$$
If we impose that the surface generated by horizontal translations along the geodesic $\Lambda$ of the above curve 
is minimal, then we get (see \cite[page 316]{st1}) that the function $\lambda(\rho)$ satisfies the equation
\begin{equation} \label{eq:l2}
\lambda'(\rho)= \frac{d}{\sqrt{\cosh^2 \rho -d^2}}, \quad d\geq 0.
\end{equation}
If $d>1$ we have 
\begin{equation} \label{eq:l3}
\lambda(\rho)= \int_{\cosh^{-1}(d)}^\rho\frac{d}{\sqrt{\cosh^2 u -d^2}} du.
\end{equation}
We parametrize $\Gamma$ as $\theta \mapsto {\rm e}^{{\rm i}\theta},$ $\theta \in (0,\pi).$ Then it is straightforward
to check that $$\rho= \log \left( \tan \frac \theta2\right).$$
Combining the above expression and \eqref{eq:l2} then we obtain:
\begin{equation} \label{eq:l4}
\lambda'(\theta)= \frac{d \csc (\theta)}{\sqrt{\csc^2(\theta) - d^2}}, 
\end{equation}
and 
$$\lambda(\theta)= \int_{\theta}^{\theta_0} \frac{d \csc( t)}{\sqrt{\csc^2(t) -d^2}} dt, \quad\theta_0=\csc^{-1}(d). $$
So the upper half part of the tall rectangle is given by the parametrization
\begin{equation} \label{eq:parametrization}
\phi(r,\theta)=(r \cos(\theta), r \sin(\theta), \lambda(\theta)), \; \;  r \in (0,\infty), \; \; \theta \in (0,\theta_0).
\end{equation}
The lower half part is parametrized by $(r,\theta) \mapsto (r \cos(\theta), r \sin(\theta),- \lambda(\theta)).$
Let $R_1$ be the complete tall rectangle described above.

The height of the tall rectangle $R_1$ is given by $2 \lambda(0) >\pi$. Moreover, notice that the intersection of the tall rectangle with
the slice $t=\lambda(\theta)$ consists of the straight line $\{ r {\rm e}^{{\rm i} \theta} \; : \; r>0\}.$

Let $R_2$ be the tall rectangle obtained from $R_1$ applying the symmetry $(x,y,t) \mapsto (-x,y,t).$

Fix $r \in (0,1)$ and $\theta \in (0,\theta_0)$. We consider the two horizontal slices, $S_{\lambda(\theta)}$ and $S_{-\lambda(\theta)}$, at height $\lambda(\theta)$ and $- \lambda(\theta)$, respectively,  and the two vertical planes $P_r$ and $P_{1/r}.$ We cut the tall rectangles $R_1 \cup R_2$ with the union 
$$S_{\lambda(\theta)} \cup S_{-\lambda(\theta)} \cup P_r \cup P_{1/r}.$$
The resulting curve is a compact curve in $R_1 \cup R_2$ with two connected components: $\beta_i \subset R_i$, $i=1,2.$
These 
two curves are symmetric with respect to the plane $x=0$. Let $\Sigma_i$ the disk in $R_i$ spanned by $\beta_i$, $i=1,2$. Observe that $\Sigma_i$ is a symmetric bi-graph over the annular sector $\Delta_i$ (see Figure \ref{delta}.)
Moreover we will denote:

\begin{itemize}
\item $D_1=\{ (u \cdot {\rm e}^{{\rm i} v} ,  \lambda(\theta)) \; : \; (u,v) \in (r,1/r) \times (\theta,\pi-\theta)\}.$
\item $D_2=\{ (u \cdot {\rm e}^{{\rm i} v} , - \lambda(\theta)) \; : \; (u,v) \in (r,1/r) \times (\theta,\pi-\theta)\}.$
\item $B_1$ the region of $P_r$ bounded by $\beta_1 \cup\beta_2$ and the curves $D_1 \cap P_r$ and
$D_2 \cap P_r$.
\item $B_2$ the region of $P_{1/r}$ bounded by $\beta_1 \cup\beta_2$ and the curves $D_1 \cap P_{1/r}$ and
$D_2 \cap P_{1/r}$.
\end{itemize}

\begin{figure}[htbp]
\begin{center}
\includegraphics[height=.25\textheight]{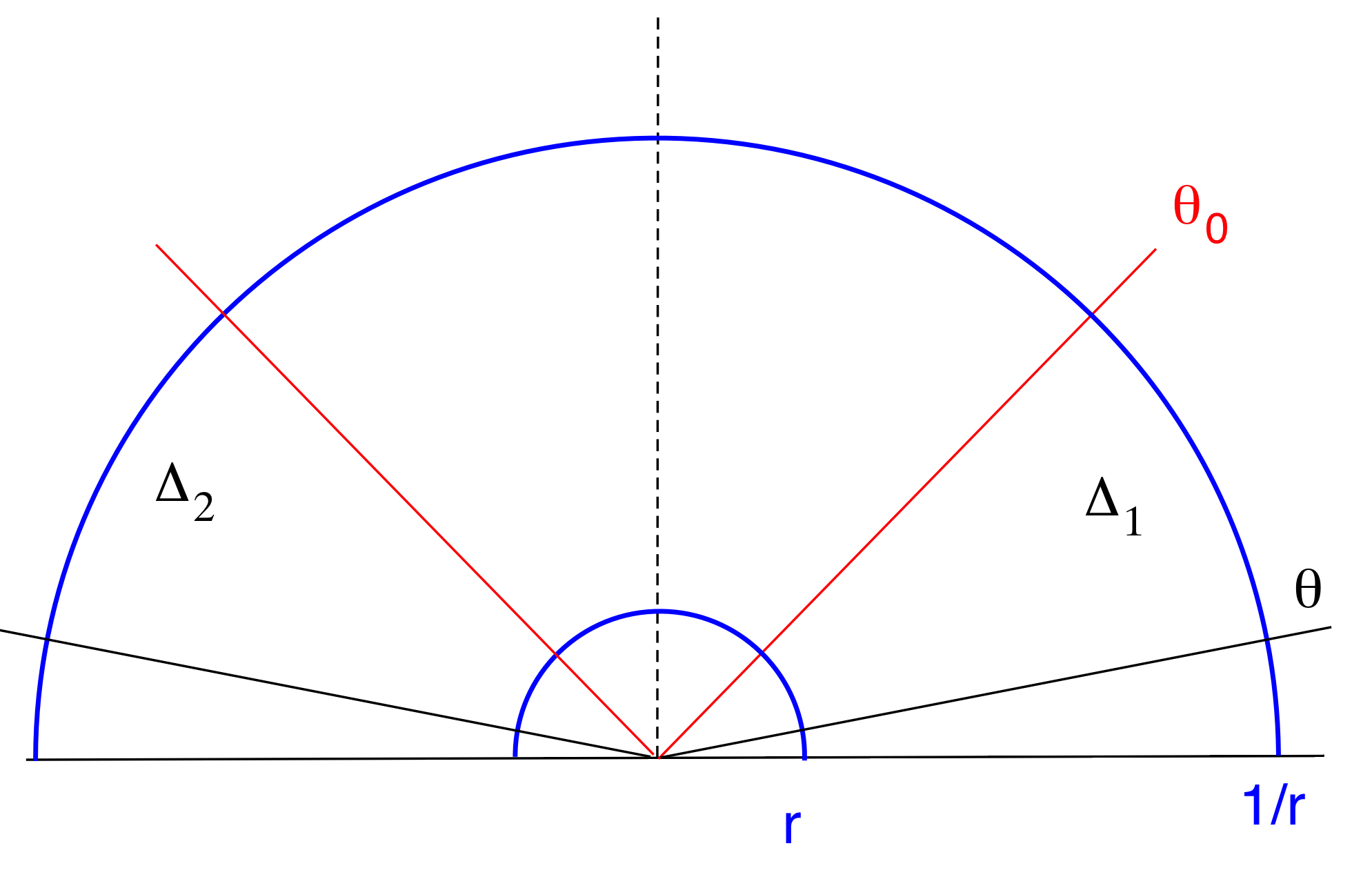}
\caption{The annular sectors $\Delta_1$ and $\Delta_2.$}
\label{delta}
\end{center}
\end{figure}

We want to prove that 
\begin{lemma}\label{lem:areamin}
The area of the region $D_1 \cup D_2 \cup B_1 \cup B_2$ is smaller than the area of $\Sigma_1 \cup \Sigma_2$.
In particular, $R_1 \cup R_2$ is not area-minimizing.
\end{lemma}
\begin{proof}By the symmety of the surface it is enough to prove that for some $\theta$ and $r$ 
$$\frac{Area(\Sigma_1)}{Area( D_1)+Area( B_1 )}>1 \;.$$
From \eqref{eq:parametrization} it is not difficult to see that
\begin{eqnarray*}
A_1(r,\theta)&:=&Area(\Sigma_1)= -2 \log(r)\left(\cot(\theta) \sqrt{4-2 d^2+2 d^2\cos(2 \theta)} + 2 E(\theta | d^2)-2F(\theta | d^2)\right. \\
&-& \left. \cot(\theta_0) \sqrt{4-2 d^2+2 d^2\cos(2 \theta_0)}-2 E(\theta_0|d^2)+2F(\theta_0 |d^2)\right) \; ,
\end{eqnarray*}
where $E$ and $F$ are the classical elliptic functions given by
$$E(\Phi|m)=\int_0^\Phi \sqrt{1-m \sin^2(s)} ds \quad , \quad F(\Phi|m)=\int_0^\Phi \frac{1}{\sqrt{1-m \sin^2(s)}} ds \; .  $$
Similarly, we have 
$$A_2(r,\theta):=Area(D_1)=- 4 \log(r) \cot(\theta) \; .$$
Finally, observe that $Area(B_1)$ is less than the area of the region in $P_r$ given by
$$\{ (r \cdot {\rm e}^{{\rm i} v} , t) \; : \; v \in  (\theta,\pi-\theta), t \in (-\lambda(0),\lambda(0))\}.$$
So, 
$$Area(B_1) < A_3(\theta):=2 \lambda(0) \log\left(\cot^2 \left(\frac{\theta}{2}\right)\right) \; .$$
Now we take $r(\theta)=\tan^n \left(\frac{\theta}{2}\right)$, where $n \in \n$ will be taken big enough in terms of $d$.
Notice that for this choice of $r$ we have 
$$\frac{Area(\Sigma_1)}{Area( D_1)+Area( B_1 )}> f(\theta):=\frac{A_1(r(\theta),\theta)}{A_2(r(\theta),\theta)+A_3(\theta)} \;.$$
It is straitforward that 
$$\lim_{\theta \to 0} f(\theta)=1 \quad ,\quad \lim_{\theta \to 0} f'(\theta)=-\frac{\lambda(0)}{n}+F\left(\theta_0\left|\csc ^2(\theta_0)\right.\right)-E\left(\theta_0\left|\csc ^2(\theta_0)\right.\right).$$
As $F\left(\theta_0\left|\csc^2(\theta_0)\right.\right)-E\left(\theta_0\left|\csc ^2(\theta_0)\right.\right)>0$, 
 then $$\lim_{\theta \to 0} f'(\theta)>0$$
 for $n$ big enough.
 So, we have that for $\theta$ small enough $f(\theta)>1$, which proves our lemma. \end{proof}

\end{document}